\documentclass{memoir}

\makeindex

\usepackage{amsmath,amsthm,amssymb}
\usepackage{graphicx}
\usepackage{hyperref}

\newcommand{\bdry}[1][\X]{\mathfrak{dG}_{#1}}
\newcommand{\half}[1][\X]{\overline{\G_x^{#1}}}

\newcommand{\coc}{\nu}
\newcommand{\arr}{\longrightarrow}
\newcommand{\alb}{\mathsf{X}}
\newcommand{\G}{\mathfrak{G}}
\newcommand{\Gh}{\mathfrak{H}}
\newcommand{\pG}{\widetilde{\G}}
\newcommand{\be}{\mathsf{o}}
\newcommand{\en}{\mathsf{t}}
\newcommand{\R}{\mathbb{R}}
\newcommand{\Z}{\mathbb{Z}}
\newcommand{\X}{X}
\newcommand{\Y}{Y}
\newcommand{\mS}{\mathcal{S}}
\newcommand{\wh}{\widehat}
\newcommand{\wt}{\widetilde}
\newcommand{\nuke}{Q}
\newcommand{\til}{\mathcal{T}}
\newcommand{\proj}{\mathsf{P}}
\newcommand{\M}{\mathcal{M}}
\newcommand{\C}{\mathbb{C}}
\newcommand{\N}{\mathbb{N}}

\newcommand{\hide}[1]{}

\newtheorem{theorem}{Theorem}[section]
\newtheorem{proposition}[theorem]{Proposition}
\newtheorem{corollary}[theorem]{Corollary}
\newtheorem{lemma}[theorem]{Lemma}

\theoremstyle{definition}
\newtheorem{defi}{Definition}[section]

\newtheorem{examp}{Example}

\numberwithin{section}{chapter}
\numberwithin{equation}{chapter}
\numberwithin{figure}{chapter}

\title{Hyperbolic groupoids and duality}
\author{Volodymyr Nekrashevych}

\date{}

\begin{document}



\maketitle

\thanks{The paper is based upon work supported by NSF grants
  DMS1006280 and DMS0605019}

\begin{abstract}
We introduce a notion of hyperbolic groupoids, generalizing the notion of a
Gromov hyperbolic group. Examples of hyperbolic groupoids include actions of
Gromov hyperbolic groups on their boundaries, pseudogroups generated
by expanding self-coverings, natural pseudogroups
acting on leaves of stable (or unstable) foliation of an
Anosov diffeomorphism, e.t.c..

We describe a duality theory for hyperbolic groupoids. 
We show that for every hyperbolic groupoid $\G$ there is a naturally defined
dual groupoid $\G^\top$ acting on the Gromov boundary of a
Cayley graph of $\G$. The groupoid $\G^\top$ is also hyperbolic and such that $(\G^\top)^\top$ is
equivalent to $\G$. 

Several classes of examples of hyperbolic groupoids and their applications are discussed.
\end{abstract}

\tableofcontents

\chapter*{Introduction}

Hyperbolic groups were introduced by M.~Gromov in the 1980s as
generalizations of fundamental groups of compact hyperbolic
manifolds, using a large-scale version of negative curvature for
metric spaces (see~\cite{gro:hyperb,ghys-h:gromov}). Since
then the theory of hyperbolic groups is one of
central topics of geometric group theory and dynamics.

The aim of our notes is to show how Gromov hyperbolic graphs appear
naturally in the study of hyperbolic and expanding dynamical systems,
and to generalize the notion of a
Gromov hyperbolic group to pseudogroups and groupoids. 
Another goal is to describe a duality theory for hyperbolic
groupoids, which in some sense interchanges the large-scale and
topological structures of hyperbolic groupoids.

There are many well known connections between negative curvature and
different types of dynamical hyperbolicity. For example, geodesic flow on a
compact negatively curved manifold is an important example of an
Anosov flow. Symbolic dynamics for Smale's Axiom A
diffeomorphisms~\cite{bowen:markovpartition} is
similar to symbolic dynamics of the geodesic
flow on a hyperbolic manifold~\cite{morse:hyperbolic}
and to automatic structures on Gromov hyperbolic
groups~\cite{cannon:hyperbolic,gro:hyperb,curnpap:symb}. We propose a
general framework uniting large-scale negative curvature (Gromov
hyperbolicity) with topological hyperbolicity (expansion).
We define \emph{hyperbolic groupoids} (or pseudogroups)
combining M.~Gromov's notion of a hyperbolic graph,
A.~Haefliger's notion of a compactly generated pseudogroup, and
D.~Ruelle's notion of a Smale space.

Examples of hyperbolic groupoids include actions of
Gromov hyperbolic groups on their boundaries, pseudogroups generated
by expanding self-coverings (e.g., by restriction of a hyperbolic
complex rational function to its Julia set), natural pseudogroups
acting on leaves of stable (or unstable) foliation of an
Anosov diffeomorphism, e.t.c..

A new aspect of the theory of
hyperbolic groupoids, not apparent in the case of Gromov hyperbolic
groups, is a duality theory interchanging
large-scale and infinitesimal properties of the groupoids. Namely, for
every hyperbolic groupoid $\G$ there is a naturally defined
\emph{dual} groupoid $\G^\top$ acting on the Gromov boundary of a
Cayley graph of $\G$, which is also hyperbolic and such that $(\G^\top)^\top$ is
equivalent to $\G$.

In this way we represent the space on which a hyperbolic groupoid acts
as the boundary of a Cayley graph of the dual groupoid. This makes it
possible to apply the methods of the theory of
Gromov hyperbolic graphs to hyperbolic dynamics. See, for example, the
paper~\cite{nek:psmeasure}, where the Patterson-Sullivan construction
and the visual metric on the boundary of a hyperbolic graph are used to
study metrics and measures on spaces on which the hyperbolic groupoid
acts.

\section{Main definitions and results}
A \emph{pseudogroup} $\pG$ of local homeomorphisms of a space $\X$ is a set
of homeomorphisms between open subsets of $\X$ that is closed under taking
inverses, compositions, restrictions to open subsets, and unions (of
homeomorphisms agreeing on the intersections of their domains). A
\emph{germ} of $\pG$ is the equivalence class of a pair $(F, x)$, where
$F\in\pG$ and $x$ is a point of the domain of $F$. We identify two
germs $(F_1, x)$ and $(F_2, x)$ if there exists a neighborhood $U$ of
$x$ such that $F_1|_U=F_2|_U$. The set of germs of elements of a
pseudogroup $\pG$ has a natural topology and is a groupoid (i.e., a
small category of isomorphisms) with
respect to the obvious composition. A pseudogroup is uniquely
determined by the topological groupoid of germs.
We will denote by $\pG$ the pseudogroup associated with a groupoid $\G$ (so that $\G$ is the groupoid of germs of the pseudogroup $\pG$).

It is more natural to consider groupoids and pseudogroups up to an
equivalence weaker than isomorphism (see~\cite{haefliger:someremarks}
and~\cite{muhlyrenault:equiv}). If $\X_0$ is an open subset of the space
$\X$ on which a pseudogroup $\pG$ acts, then we denote by
$\pG|_{\X_0}$ the set of elements of $\pG$ such that their range and
domain are subsets of $\X_0$.  We say that a subset $\X_0$ of the
space on which a pseudogroup $\pG$ acts is a $\pG$-\emph{transversal}
if it intersects every $\pG$-orbit.

\begin{defi}
Pseudogroups $\pG_1$, $\pG_2$ are \emph{equivalent} if there exists a
pseudogroup $\wt{\Gh}$ and $\wt{\Gh}$-transversals $\X_1$ and $\X_2$
such that restrictions of $\wt{\Gh}$ to $\X_1$ and $\X_2$ are
isomorphic to $\pG_1$ and $\pG_2$ respectively.
\end{defi}

Pseudogroups and their groupoids of germs appear naturally in the
study of \emph{local
symmetries} of structures, especially when the structure does not have enough
globally defined symmetries, so that group theory does not
describe the symmetric nature of the structure adequately. Structures of this
type are, for example, aperiodic tilings (like the Penrose tiling),
Julia sets, attractors of iterated function systems, non-wandering
sets of Smale's Axiom A diffeomorphisms, etc.

In many interesting cases local symmetries include
\emph{self-similarities}, i.e., symmetries identifying
parts of the object on different scales.
Definition of a hyperbolic groupoid combines the idea of ``multiscale''
symmetries of a self-similar structure with the Gromov's idea of a
negatively curved Cayley graph. In order to define Cayley graphs of
groupoids, we use the following definition due to A.~Haefliger~\cite{haefliger:compactgen}.

We say that a pseudogroup $\pG$ is
\emph{compactly generated} if there exists an open relatively compact
(i.e., having compact closure) $\pG$-transversal
$\X_0$ and a finite set $\mS$ of elements of $\pG$ with range and
domain in $\X_0$ such that
\begin{enumerate}
\item every germ $(F, x)\in\G$ such that $\{x, F(x)\}\subset
  \X_0$ can be represented as a germ of a product of elements of $\mS$
  and their inverses;
\item for every element $F\in\mS$ there exists an element $\wh F$ of
  $\pG$ such that $F$ is a restriction of $\wh F$, and closures of the
  domain and the range of $F$ are compact and contained in the domain and
  the range of $\wh F$ respectively.
\end{enumerate}

Suppose that $\X_0$ and $\mS$ are as above. We say that $\mS$ is
a \emph{compact generating set} of $\pG|_{\X_0}$. The \emph{Cayley
graph} $\G(x, \mS)$ for $x\in\overline{\X_0}$
is the oriented graph with the set of vertices equal to
the set of germs $(F, x)$ such that $F(x)\in\overline{\X_0}$. There is an
arrow starting at $(F_1, x)$ and ending in $(F_2, x)$ if and only if
there exists $F\in\mS$ such that $(F_2, x)=(FF_1, x)$.

Similarly to the case of Cayley graphs of groups, the Cayley graph
$\G(x, \mS)$ does not depend (up to a quasi-isometry) on the choice of
the generating set $\mS$ and the transversal
$\X_0$ (see~\ref{ss:cayley}). It depends, however, on the choice of the basepoint $x$.

We will use notation $|x-y|$ for distance between points $x$ and $y$
in a metric space.

\begin{defi}
\label{def:hyperbolicgroupoid}
A pseudogroup $\pG$ is \emph{hyperbolic} if there exists an open relatively
compact transversal $\X_0$, a metric $|\cdot|$ defined on a neighborhood
$\wh\X_0$ of the closure of $\X_0$, and a compact generating set
$\mS$ of $\pG|_{\X_0}$ such that
\begin{enumerate}
\item all elements of $\mS^{-1}$ are Lipschitz and there exists
  $\lambda\in (0, 1)$ such that $|F(x)-F(y)|<\lambda |x-y|$ for
  all $F\in\mS$, $x, y\in\mathop{\mathrm{Dom}}F$;
\item there exists $\delta>0$ such that for every $x\in\X_0$
the Cayley graph $\G(x, \mS)$ is $\delta$-hyperbolic;
\item for every $x\in\X_0$ every vertex has at least one incoming and
  one outgoing arrow in $\G(x, \mS^{-1})$, and all infinite directed
  paths in $\G(x, \mS^{-1})$ are (in a uniform way)
quasi-geodesics converging to one
point $\omega_x\in\partial\G(x, \mS^{-1})$.
\end{enumerate}
\end{defi}

Note that this definition is a combination of two hyperbolicity
conditions: topological (the elements of $\mS$ are contractions) and large-scale (Gromov
hyperbolicity). Interplay between these two conditions and
duality between them is the central topic of this monograph.

Denote by $\partial\G_x$ the hyperbolic boundary of the Cayley graph
$\G(x, \mS)$ minus the special point $\omega_x$.
The space $\partial\G_x$ does not depend on
the choice of $\X_0$ and $\mS$, and hence is defined for every
$x\in\X$. Denote by $\half$ the union of $\G(x, \mS)$ with
$\partial\G_x$ with the natural topology (i.e., the compactification
of $\G(x, \mS)$ by its boundary with the point $\omega_x$ deleted).

A map $F:W_1\arr W_2$ between open neighborhoods of points
$\xi_1\in\partial\G_x$ and $\xi_2\in\partial\G_y$ in $\half$ and
$\half[y]$ respectively, is an \emph{asymptotic isomorphism} if for
every sequence of pairwise different edges $(g_n, h_n)\in W_1$
the distance between $g_nh_n^{-1}$ and
$F(g_n)F(h_n)^{-1}$ (with respect to a metric on a compact neighborhood of
the set of germs of $\mS$) goes to zero as $n\to\infty$.

\begin{defi}
The \emph{dual groupoid} is the groupoid of germs of the action of
asymptotic isomorphisms on the boundaries of the Cayley graphs of $\G$.
\end{defi}

If the groupoid $\G$ is \emph{minimal}, i.e., if all its orbits are
dense, then we can choose any point $x\in X$, and define the dual groupoid $\G^\top$ as the groupoid
of germs of the action on $\partial\G_x$ of the asymptotic
automorphisms of the Cayley graph $\G(x, \mS)$.

Let us give a more explicit description of asymptotic automorphisms.
If $\mS$ is rich enough, then every point
$\xi\in\partial\G_x$ is the limit as $n\to\infty$
of a sequence of vertices $g_ng_{n-1}\cdots g_1\cdot g$, where $g_i$ are
germs of elements $F_i\in\mS$ and $g$ is a vertex of the Cayley graph
$\G(x, \mS)$. Then there exists an element $F\in\pG$ such that $g=(F,
x)$, and $g_i=(F_i, F_{i-1}F_{i-2}\cdots F_1F(x))$. Since $F_i$ are
contractions, there exists a neighborhood $U$ of $x$ such that all
maps $F_nF_{n-1}\cdots F_1F$ are defined on $U$.

If $h$ is such that the range (target) $y$ of $h$ belongs to $U$, then define
\[R_g^h(g_ng_{n-1}\cdots g_1g)=F_nF_{n-1}\cdots F_1h,\] see
Figure~\ref{fig:tube}. One can show that if $x$ and $y$ are
sufficiently close, then the transformation $R_g^h$ between subsets of 
the corresponding Cayley graphs is well defined. It follows from the
fact that the maps $F_i$ are contractions that the transformations $R_g^h$ are
asymptotic isomorphisms. One can also show that all asymptotic isomorphisms are locally
equal to transformations of this form.

\begin{figure}
\centering
\includegraphics{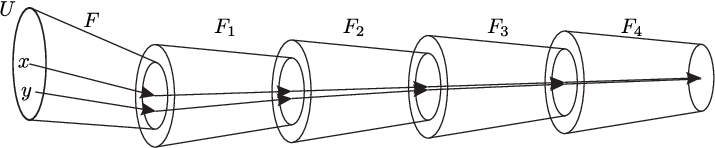}
\caption{}\label{fig:tube}
\end{figure}

The main results of the paper is the following duality theorem
(see Theorem~\ref{th:dualitytheorem}, which is proved for a more
general class of hyperbolic groupoids).

\begin{theorem}
\label{th:veryfirst}
Suppose that $\G$ is a minimal hyperbolic groupoid. Then the dual groupoid $\G^\top$ is also
hyperbolic and the groupoid $(\G^\top)^\top$ is equivalent to $\G$.
\end{theorem}

The maps $R_g^h$ can be used to define a topology on the disjoint
union $\partial\G$ of the boundaries $\partial\G_x$. Namely, define
\[[y, \xi]_F=R_g^y(\xi)=\lim_{n\to\infty} (F_nF_{n-1}\cdots F_1F,
y)\in\partial\G_y,\]
where $\xi=\lim_{n\to\infty}g_n\cdots g_1\cdot g$, and $y$ is a point
close to the target of $g$, as above.

One can show that the maps $[\cdot, \cdot]_F$ define a topology and a
\emph{local product structure} on $\partial\G$. Namely, every point
$\xi\in\partial\G_x\subset\partial\G$ has a neighborhood (called a
\emph{rectangle}) identified by the map $[\cdot, \cdot]_F$ with the
direct product of a neighborhood of $x$ in $\X$ and a neighborhood of
$\xi$ in $\partial\G_x$. Moreover, the defined direct product
decompositions agree on the intersections of the rectangles.

For a germ $g=(F, x)\in\G$ and a point $\xi\in\partial\G_{F(x)}$
represented as limit of the sequence $v_n$ of vertices of the Cayley
graph $\G(F(x), \mS)$ define $\xi\cdot g$ as limit of the sequence
$v_n\cdot g$. Then $\xi\cdot g$ belongs to $\partial\G_x$, and we
get for every $F\in\pG$ a local homeomorphism of $\partial\G$ mapping
$\xi\in\partial\G_{F(x)}$ to $\xi\cdot (F, x)\in\partial\G_x$ for
every point $x$ of the domain of $F$. Such local homeomorphisms
preserve the local product structure and generate a pseudogroup. Denote by
$\partial\G\rtimes\G$ the groupoid of germs of this pseudogroup, and
call it the \emph{geodesic quasi-flow} of the hyperbolic groupoid $\G$.

It is proved in~\ref{ss:actionofgonpg} that the groupoid
$\partial\G\rtimes\G$ is compactly generated. It is also proved that
it is a \emph{quasi-flow}, i.e., it has a natural \emph{quasi-cocycle}
$\coc:\partial\G\rtimes\G\arr\R$ such that restriction of $\coc$ to any
Cayley graph of $\partial\G\rtimes\G$ is a quasi-isometry of the
Cayley graph with $\R$.

Consider an open relatively compact transversal $W\subset\partial\G$
and a finite covering of $W$ by open rectangles $R_i=A_i\times B_i$,
where $A_i$ is an open subset of $\X$ and $B_i$ is an open subset of
$\partial\G_x$. Since
$\partial\G\rtimes\G$ preserves the local product structure, every
element of the pseudogroup $\wt{\partial\G\rtimes\G}$ can be locally
decomposed into a direct product of a homeomorphism $\proj_+(F)$
between open subsets of the sets $A_i$ and a homeomorphism
$\proj_-(F)$ between open subsets of the sets $B_i$. The groupoid
generated by the germs of the homeomorphisms of the form $\proj_+(F)$ is equivalent to
$\G$, while the groupoid generated by the germs of the homeomorphisms
$\proj_-(F)$ is equivalent to the dual groupoid $\G^\top$. In other
words, the groupoids $\G$ and $\G^\top$ are projections of
$\partial\G\rtimes\G$ on the factors of the natural local product
decomposition of $\partial\G$.

In the process of proving the Duality Theorem~\ref{th:veryfirst} we give an axiomatic
description of geodesic quasi-flows of hyperbolic groupoids. The
corresponding notion is a generalization of the classical notion of a
\emph{Smale space} (see~\cite{ruelle:therm}). Informally,
a compactly generated pseudogroup $\wt{\Gh}$ acting on a space $\X$
is a \emph{Smale quasi-flow} if the Cayley graphs of $\wt{\Gh}$ are
quasi-isometric to $\R$, and elements of $\wt{\Gh}$ expand one and
contract the other direction of a local product structure on $\X$. Precise definition has more
conditions (see Definition~\ref{def:hypflow}).

If $\Gh$ is a hyperbolic quasi-flow, then we define the
groupoids $\proj_+(\Gh)$ and $\proj_-(\Gh)$ (called the
\emph{Ruelle groupoids}) generated by the projections of the
elements of $\Gh$ onto the corresponding directions of the local
product structure. We prove the following results (see
Theorems~\ref{th:geodflowhyperbolic},~\ref{th:Ruellehyperbolic},
and~\ref{th:everyhypflowgeod}).

\begin{theorem}
\label{th:main2}
If a groupoid $\G$ is hyperbolic and minimal, then $\partial\G\rtimes\G$ is a Smale
quasi-flow.

The Ruelle groupoids of a Smale quasi-flow $\Gh$ are hyperbolic,
and $\Gh$ is equivalent to the geodesic quasi-flows of its Ruelle
groupoids.
\end{theorem}

In particular, we show that the natural projections of a Smale space (e.g., of an Anosov
diffeomorphism, or of restriction of an Axiom A diffeomorphism to
the non-wandering set) onto the stable and unstable directions of its local
product structure form a pair of mutually dual hyperbolic groupoids. We call
them the \emph{Ruelle groupoids} of the Smale space, since the convolution
algebras of these groupoids are the Ruelle algebras studied
in~\cite{putnam:lecturenotes,putnam}, see also the paper of
D.~Ruelle~\cite{ruelle:grpd} where similar algebras and groupoids are
considered.

As a particular example, consider
the groupoid generated by the one-sided shift of finite type defined by a
finite set of prohibited
words $P$. It is hyperbolic, and its dual is the groupoid generated by the shift of finite
type defined by the set of prohibited words obtained by writing the
elements of $P$ in the opposite order.

Another class of examples of hyperbolic groupoids consists of the groupoids of
the action of Gromov-hyperbolic groups $G$ on their
boundaries $\partial G$. Every such groupoid is self-dual.
The corresponding geodesic quasi-flow is
equivalent to the diagonal action of $G$ on $\partial G\times\partial
G$  minus the diagonal. (For the notion of a geodesic flow associated
with a Gromov hyperbolic group see~\cite[Theorem~8.3.C]{gro:hyperb}.)

If $f:\M\arr\M$ is an expanding self-covering of a compact
space, then the groupoid generated by the germs of $f$ is
hyperbolic. Its dual, when $\M$ is connected, is generated by an
action of a group (called the \emph{iterated monodromy group} of $f$)
on a full one-sided Bernoulli shift-space and by the shift. The action
of the iterated monodromy group has a nice symbolic
presentation by \emph{finite automata}, which makes computations in
the dual groupoid very efficient, which in turn gives an efficient
description of the symbolic dynamics of $f$. This particular case of
hyperbolic duality is the main topic of the monograph~\cite{nek:book}. See
also~\cite{bartnek:rabbit,nek:models,nek:dendrites}
for other applications of iterated monodromy groups,
and~\cite{nek:smale}
for a generalization of iterated monodromy groups, which is also an
instance of hyperbolic duality.

An interesting $K$-theoretic duality of the convolution $C^*$-algebras of the
Ruelle groupoids of Smale spaces was
proved by J.~Kaminker, I.~Putnam, and M.~Whittaker in their
paper~\cite{kaminkerputnamwhittaker} (see also~\cite{kaminkerputnam:shifts}).
H.~Emerson~\cite{emerson:duality} has proved Poincar\'e duality for the
convolution algebras of the groupoid of the action of a
Gromov-hyperbolic group on its boundary.
It would be interesting to generalize both results
to all pairs of mutually dual hyperbolic groupoids.

A particular case of hyperbolic duality
was defined by W.~Thurston
in~\cite{thurston:automata} in relation with self-similar tilings
and numeration systems, see also~\cite{gelbrich:dualpenrose} for a natural
generalization.

Relation between Gromov hyperbolicity and hyperbolic dynamics was
explored in~\cite{nek:hyplim,pilgrim:nekrash,haisinskypilgrim}.

\section{Structure of the paper}

Chapter~\ref{s:preliminaries} is an overview of some of technical tools
used in the paper. In its first section we define a notion that is
sometimes a convenient replacement of the notion of a metric, and is
especially natural in the setting of hyperbolic dynamics and
geometry. Instead of working with a metric we work with a
\emph{log-scale}, which is, up to an additive constant, logarithm (base
less than one) of a metric. The triangle inequality is replaced by
the condition
\[\ell(x, z)\ge \min(\ell(x, y), \ell(y,
z))-\delta\]
for a positive constant $\delta$. We also assume that $\ell(x, y)=\ell(y, x)$,
and that $\ell(x, y)=\infty$ if and only if $x=y$. It follows from
Frink metrization lemma (see~\cite[Lemma~6.12]{kelley:gentop}) that
any function satisfying these conditions is equal, up to an additive constant,
to a logarithm of a metric.

We define natural notions of Lipschitz and H\"older equivalence of
log-scales, and show that log-scales can be ``pasted together'' from
locally defined log-scales (which is one of the reasons for using them
in our paper).

Next section of Chapter~\ref{s:preliminaries} introduces the
necessary notions of the theory of Gromov-hyperbolic metric spaces
and graphs. In particular, we prove a hyperbolicity criterion, which
will be used later to prove large-scale hyperbolicity of Ruelle
groupoids.

The last section of Chapter~\ref{s:preliminaries} defines the
notion of a local product structure on a topological space. This
notion was introduced by Ruelle in~\cite{ruelle:therm}.
We adapt it to our slightly more general setting.

Chapter~\ref{s:groupoidprelims} is an overview of the theory of pseudogroups and
groupoids of germs. In particular, we fix there our notation and terminology.

Section~\ref{ss:cayley} introduces the notion of a Cayley graph of
a groupoid, which is a natural development of
A.~Haefliger's notion of a compactly generated groupoid.

In the last section of Chapter~\ref{s:groupoidprelims} we define
compatibility conditions between pseudogroups, quasi-cocycles,
log-scales, and local product structures.

In Chapter~\ref{s:hypgroupoids} we define the notion of a hyperbolic
groupoid $\G$, its boundary $\partial\G$, the local product structure
on $\partial\G$, the dual groupoid $\G^\top$, and
the \emph{geodesic quasi-flow}, i.e., the
action of the groupoid $\G$ on $\partial\G$.

Chapter~\ref{s:quasiflow} is devoted to the proof of the duality
theorem for hyperbolic groupoids and to axiomatic description of
Smale quasi-flows. In the last two sections we give another definition of the dual
groupoid (using partial maps between positive cones of the Cayley
graphs), and study irreducibility conditions for hyperbolic groupoids.

In the last chapter of the paper we describe different examples of
pairs of dual hyperbolic groupoids and the corresponding hyperbolic
quasi-flows: actions of Gromov-hyperbolic groups on their boundaries,
groupoids generated by expanding maps, groupoids associated with
self-similar groups, Smale spaces, etc.

\subsection{Acknowledgments}
I am grateful to A.~Haefliger, J.~Kaminker, K.~Pilgrim, I.~Putnam
for fruitful discussions on the topics of these notes, and to the
anonimous referee for suggestions leading to improvement of the paper.

\chapter{Technical preliminaries}
\label{s:preliminaries}

\section*{Some remarks on notation and terminology}

All spaces on which groups, pseudogroups, or groupoids act are assumed
to be locally compact and metrizable. A neighborhood $U$ of a point
$x$ is any set containing an open set $V$ such that $x\in V$. In
particular, neighborhoods are not assumed to be open.

We use notation $|x-y|$ for distance between the points $x$
and $y$ of a metric space.

We write $F\doteq G$, where
$F$ and $G$ are real functions, if there exists a constant
$\Delta>0$ such that $|F-G|<\Delta$ for all values of the
variables. Similarly, we write $F\asymp G$ if there exists a constant
$C>1$ such that $C^{-1}F\le G\le CF$ for all values of the variables.

Notation $x\doteq\lim_{n\to\infty}a_n$ means that there exists a constant $\Delta$
such that for every partial limit $y$ of the sequence $a_n$ we have
$|x-y|\le\Delta $.

Most of our group and pseudogroup actions are from the right. In
particular, in a product of maps $fg$ the map $g$ acts before $f$.

We denote by $\lfloor x\rfloor$ and $\lceil x\rceil$ the largest
integer not smaller than $x$ and the smallest integer not larger than
$x$ respectively. 

\section{Logarithmic scales}
\label{ss:logscales}

\begin{defi}
\label{def:logscale} A \emph{logarithmic scale} (a
\emph{log-scale})\index{log-scale}\index{logarithmic
  scale|see{log-scale}} on a set $\X$ is a
function $\ell:\X\times\X\arr\R\cup\{\infty\}$ such that
\begin{enumerate}
\item $\ell(x, y)=\ell(y, x)$ for all $x, y\in\X$;
\item $\ell(x, y)=+\infty$ if and only if $x=y$;
\item there exists $\delta\ge 0$ such that $\ell(x, z)\ge \min(\ell(x, y), \ell(y,
z))-\delta,$ for all $x, y, z\in\X$.
\end{enumerate}
\end{defi}

\begin{defi}
We say that a metric $|x-y|$ on $\X$ is \emph{associated} with the
log-scale $\ell$ if there exist constants $\alpha$ and
$c$ such that $0<\alpha<1$, $c>1$, and
\[c^{-1}\alpha^{\ell(x, y)}\le |x-y|\le c\alpha^{\ell(x, y)}\]
for all $x, y\in\X$. \index{metric!associated with a log-scale}
\end{defi}

\begin{proposition}
\label{pr:logscalemetric}
For any log-scale $\ell$ on $\X$ there exists a metric $|\cdot|$
associated with it.

If $|x-y|$ is a metric on $\X$, then the function
\[\ell(x, y)=-\lfloor\ln |x-y|\rfloor\]
is a log-scale associated with it.
\end{proposition}

\begin{proof}
Define for $n\in\Z$
\[E_n=\{(x, y)\;:\;\ell(x, y)\ge n\}.\]
Note that if $\delta$ is as in Definition~\ref{def:logscale}, then
\[E_{n+2\delta}\circ E_{n+2\delta}\circ E_{n+2\delta}\subset E_{n+\delta}\circ
E_{n+2\delta}\subset E_{n+\delta}\circ E_{n+\delta}\subset E_n.\]

It follows from Frink metrization lemma, see~\cite[Lemma~6.12]{kelley:gentop}, that there exists
a metric $|\cdot|$ on $\X$ such that
\[E_{2\delta n}\subset\{(x, y)\;:\;|x-y|<2^{-n}\}\subset
E_{2\delta(n-1)}.\]

Suppose that $\ell(x, y)=n$. Then $\ell(x, y)\ge 2\delta\lfloor
n/(2\delta)\rfloor$, hence \[|x-y|<2^{-\lfloor\frac{n}{2\delta}\rfloor}\le
2^{-\frac{n}{2\delta}+1}.\]

On the other hand $\ell(x, y)<2\delta(\lfloor n/(2\delta)\rfloor
+1)$, hence \[|x-y|\ge 2^{-\lfloor\frac{n}{2\delta}\rfloor-1}\ge
2^{-\frac{n}{2\delta}-2}.\]

We have shown that for $\alpha=2^{-\frac 1{2\delta}}$ we have
\[\frac 14\cdot\alpha^{\ell(x, y)}\le |x-y|\le 2\cdot
\alpha^{\ell(x, y)},\] which shows that $|\cdot|$ is associated with
$\ell$.

The second part of the proposition is obvious.
\end{proof}

If a log-scale is defined on a topological
space, then we assume that the topology defined by it (i.e., by an
associated metric) coincides with
the original topology of the space.

Taking into account the described relation between
log-scales and metrics, we give the following definitions.

\begin{defi}
Let $\ell_1$ and $\ell_2$ be log-scales on the sets $\X_1$ and
$\X_2$ respectively. A map $f:\X_1\arr\X_2$ is called
\emph{Lipschitz} if there exists a positive number $n$ such that
\[\ell_2(f(x), f(y))\ge\ell_1(x, y)-n\]
for all $x, y\in\X_1$.\index{Lipschitz!map} A map $f$ is
\emph{bi-Lipschitz} if it is invertible, and both $f$ and $f^{-1}$ are
Lipschitz.\index{bi-Lipschitz!map}

A function $f:\X_1\arr\X_2$ is said to be \emph{locally
  Lipschitz}\index{locally Lipschitz map}
if for every $x\in\X_1$ there exists a neighborhood $U$ of $x$ such
that restriction of $f$ to $U$ is Lipschitz.
Equivalently, $f$ is locally Lipschitz if for every $x\in\X_1$
there exist $N$ and $n$ such that $\ell_2(f(x), f(y))\ge\ell_1(x,
y)-n$ for all $y\in\X_1$ such that $\ell_1(x, y)\ge N$.
\end{defi}

\begin{defi}
Let $\ell$ be a log-scale on $\X$. A map $f:\X\arr\X$ is a
\emph{contraction} if there exists $c>0$ such that
\[\ell(f(x), f(y))\ge\ell(x, y)+c  \] for all $x, y\in\X_1$.
\end{defi}

Note that if $\ell_1$ is a log-scale, and $\ell_2$ is a function such
that $\ell_1-\ell_2$ is uniformly bounded on
$\X\times\X$, then $\ell_2$ is also a log-scale.

\begin{defi}
We say that log-scales $\ell_1$ and $\ell_2$ on $\X$ are
\emph{Lipschitz equivalent} if there exists $k>0$ such that
\[|\ell_1(x, y)-\ell_2(x, y)|\le k\]
for all $x, y\in\X$.\index{Lipschitz!equivalence}

They are \emph{H\"older equivalent} if there exist constants $c>1$
and $k>0$ such that
\[c^{-1}\ell_1(x, y)-k\le\ell_2(x, y)\le c\ell_1(x, y)+k\]
for all $x, y\in\X$.\index{Holder@{H\"older equivalence}}
\end{defi}

We say that a log-scale is \emph{positive} if all its
values are positive.
\index{log-scale!positive}\index{positive log-scale}

\begin{lemma}
\label{l:locglob}
Let $\ell_i$ be log-scales on spaces $\X_i$, for $i=1, 2$.
Suppose that $\X_1$ is compact and $\ell_2$ is positive.
If $f:\X_1\arr\X_2$ is locally Lipschitz, then it is Lipschitz.
\end{lemma}

In particular, two positive log-scales on a compact space are
Lipschitz equivalent if and only if they are locally Lipschitz equivalent.

\begin{proof}
The set $\X_1$ can be covered by a finite set $\mathcal{U}$ of open
sets such that restriction of $f$ to the elements of $\mathcal{U}$
are Lipschitz. There exists $n_1$ such that if $\ell_1(x, y)\ge n_1$,
then $x, y$ belong to one element of $\mathcal{U}$. Then there exists
a constant $c>0$ such that for any pair
$x, y\in\X_1$ either $\ell_1(x, y)\le n_1$, or $\ell_2(f(x),
f(y))\ge\ell_1(x, y)-c$. In the first case we have $\ell_2(f(x),
f(y))>0\ge\ell_1(x, y)-n_1$.
\end{proof}

\begin{theorem}
\label{th:pastingLipschitz} Let $\X$ be a compact metrizable space
and let $\{U_i\}_{i\in I}$ be a covering of $\X$ by open sets.
Suppose that $\ell_i$ is a positive log-scale on $\overline{U_i}$
such that for every pair $U_i, U_j$ the scales $\ell_i$ and
$\ell_j$ are Lipschitz equivalent on the intersection
$\overline{U_i}\cap\overline{U_j}$. Then there is a log-scale
$\ell$ on $\X$ such that $\ell$ is Lipschitz equivalent to
$\ell_i$ on $\overline{U_i}$ for every $i\in I$.
\end{theorem}

Note that the log-scale $\ell$ satisfying the conditions of the
theorem is necessarily unique up to Lipschitz equivalence.

\begin{proof}
After passing to a finite sub-covering, we may assume that $I$ is finite.
By Lebesgue's number lemma, there exists a symmetric
neighborhood $E$ of the diagonal of $\X\times\X$ such that for
every $x\in\X$ there exists $i\in I$ such that
\[\{y\in\X\;:\;(x, y)\in E\}\subset U_i.\]

For every $i\in I$ the set
$E\cap\overline{U_i}\times\overline{U_i}$ is a symmetric
neighborhood of the diagonal of $\overline{U_i}\times\overline{U_i}$, hence by compactness of
$\overline{U_i}$ there exists $n_i$ such that for every pair $x,
y\in\overline{U_i}$ such that $\ell_i(x, y)\ge n_i$ we have $(x,
y)\in E$. Let $n=\max_{i\in I}n_i$.

Choose for every two-element subset $\{x, y\}\subset\X$ such that
$(x, y)\in E$ an element $U_i$ of the covering such that $\{x,
y\}\subset U_i$ (which exists by the choice of $E$). Define
$\ell'(x, y)=\ell_i(x, y)$.

For every pair of points $x, y$, let
\[\ell(x, y)=\left\{\begin{array}{ll}\max\{\ell'(x, y), n\}, & \text{if $(x, y)\in E$,}\\
n & \text{otherwise.}\end{array}\right.\]
Let us show that $\ell$ is a log-scale on $\X$ satisfying the conditions of the
theorem.

Let $\Delta$ be such that
\[\ell_i(x, y)-\Delta\le \ell_j(x, y)\le\ell_i(x, y)+\Delta\]
for all $i, j\in I$ and $x, y\in \overline{U_i}\cap\overline{U_j}$.
Let $\delta$ be
such that \[\ell_i(x, z)\ge\min\{\ell_i(x, y), \ell_i(y,
z)\}-\delta\] for all $i\in I$ and $x, y, z\in\overline{U_i}$. Such numbers
exist by conditions of the theorem and finiteness of the set $I$.

Suppose that $x, y, z\in\X$ are such that $(x, y), (y, z)\in E$.
Then there exists $U_i$ such that $\{x, y, z\}\subset U_i$.
Suppose at first that $(x, z)\notin E$. Then $\ell_i(x, z)<n$.

If one of the numbers $\ell'(x, y), \ell'(y, z)$ is less than or equal
to $n$, then we have $\min\{\ell(x, y), \ell(y, z)\}=n=\ell(x, z)$.

If both numbers $\ell'(x, y), \ell'(y, z)$ are greater than $n$,
then $\ell(x, y)=\ell'(x, y)$, $\ell(y, z)=\ell'(y, z)$, and
\begin{multline*}\ell(x, z)=n>\ell_i(x, z)\ge
\min\{\ell_i(x, y), \ell_i(y, z)\}-\delta\\ \ge\min\{\ell'(x, y),
\ell'(y, z)\}-\Delta-\delta=\min\{\ell(x, y), \ell(y,
z)\}-\Delta-\delta.\end{multline*} Suppose now that $(x, z)\in E$. Then
\begin{multline*}\ell(x, z)\ge\ell'(x, z)\ge\ell_i(x, z)-\Delta\ge\\
\min\{\ell_i(x, y), \ell_i(y, z)\}-\delta-\Delta\ge\\ \min\{\ell'(x, y), \ell'(y,
z)\}-\delta-2\Delta.\end{multline*} It follows that if $\ell'(x, y)$
and $\ell'(y, z)$ are both greater than $n$, then we have
\[\ell(x, z)\ge\min\{\ell(x, y), \ell(y, z)\}-\delta-2\Delta.\]
Otherwise
$\ell(x, z)\ge n=\min\{\ell(x, y), \ell(y, z)\}$.

Suppose now that at least one of the pairs $(x, y)$ and $(y, z)$ does not
belong to $E$. Then $\min\{\ell(x, y), \ell(y, z)\}=n$, and
$\ell(x, z)\ge n=\min\{\ell(x, y), \ell(y, z)\}$.

We have shown that $\ell$ is a log-scale on $\X$. Let us
show that it is Lipschitz equivalent to $\ell_i$ on every set
$\overline{U_i}$.

Let $x, y\in\overline{U_i}$ and suppose that $(x, y)\notin E$. Then
$\ell_i(x, y)\le n$, hence
\[\ell(x, y)=n\ge\ell_i(x, y)\ge 1=\ell(x, y)+(1-n).\]

Suppose now that $(x, y)\in E$. Then either $\ell'(x, y)\le n$ and
$\ell(x, y)=n$, so that
\[\ell(x, y)-n-\Delta=-\Delta\le\ell'(x, y)-\Delta\le\ell_i(x, y)\le
\ell'(x, y)+\Delta\le\ell(x, y)+\Delta;\]
or $\ell'(x, y)>n$ and
\[\ell(x, y)-\Delta=\ell'(x, y)-\Delta\le\ell_i(x, y)\le
\ell'(x, y)+\Delta=\ell(x, y)+\Delta.\]
It follows that $\ell_i$ is Lipschitz equivalent to $\ell$.
\end{proof}

Let us reformulate the notions related to completion of a metric space
in terms of log-scales.

A sequence $x_n$ is \emph{Cauchy} if $\ell(x_n,
x_m)\to\infty$
as $n, m\to\infty$. Two Cauchy sequences $x_n$ and $y_n$ are
equivalent if $\ell(x_n, y_n)\to\infty$. A completion of $\X$ with
respect to the log-scale $\ell$ is the set of equivalence classes of
Cauchy sequences in $\X$.

\begin{proposition}
Suppose that $x_n$ and $y_n$ are non-equivalent Cauchy sequences. Then there exists a constant
$\Delta>0$ such that any two partial limits of the sequence $\ell(x_n,
y_n)$ differ not more than by $\Delta$ from each other.
\end{proposition}

\begin{proof}
We have for all $n, m$
\[\ell(x_m, y_m)\ge\min(\ell(x_n, x_m), \ell(x_n, y_n), \ell(y_n,
y_m))-2\delta.\] There exists a sequence $n_i$ such that $\ell(x_{n_i},
y_{n_i})$ is bounded from above by some number $l$. Then for all
$m$ and $i$ big enough we have $\ell(x_{n_i}, x_m)\ge l\ge\ell(x_{n_i},
y_{n_i})$ and $\ell(y_{n_i}, y_m)\ge l\ge\ell(x_{n_i}, y_{n_i})$, hence
\[\ell(x_m, y_m)\ge\min\{\ell(x_{n_i}, x_m), \ell(x_{n_i}, y_{n_i}),
\ell(y_{n_i}, y_m)\}-2\delta=\ell(x_{n_i}, y_{n_i})-2\delta.\]
In particular, for all $i$ and $j$ big enough we have
\[\ell(x_{n_j}, y_{n_j})\ge\ell(x_{n_i}, y_{n_i})-2\delta.\]
It follows that there exists $i_0$ such that all values of
$\ell(x_{n_i}, y_{n_i})$ for $i\ge i_0$ are less
than $2\delta$ away from each other. Let $k=\ell(x_{n_{i_0}}, y_{n_{i_0}})$.
Then, for all $m$ and $i$ big enough, we have $\ell(x_m, y_m)\ge
k-4\delta$ and
\begin{multline*}
k+2\delta>\ell(x_{n_i}, y_{n_i})\\
\ge\min(\ell(x_m, x_{n_i}), \ell(x_m, y_m), \ell(y_m, y_{n_i}))-\delta=\ell(x_m,
y_m)-\delta\\ \ge k-5\delta,
\end{multline*} since we may assume that $\ell(x_m, x_{n_i})$ and
$\ell(y_m, y_{m_i})$ are greater that $k+3\delta$.
It follows that the values of $\ell(x_m,
y_m)$ for all $m$ big enough are at most $3\delta$ away from $k$.
\end{proof}

If $x=\lim_{n\to\infty}x_n\ne y=\lim_{n\to\infty}y_n$
in the completion of $\X$ with respect to $\ell$, then
we can define $\ell(x, y)$ as any partial limit of the
sequence $\ell(x_n, y_n)$. The defined function is a log-scale on
the completion and is unique up to Lipschitz equivalence. We
will call this log-scale the \emph{natural extension} of the
log-scale $\ell$ to the completion.

\section{Gromov-hyperbolic metric spaces and graphs}
\subsection{General definitions and main properties}
\label{sss:hyperbolicdef}
Let $(\X, |\cdot|)$ be a metric space. Choose a point $x_0\in\X$ and
define
\[(x, y)_{x_0}=\frac 12(|x_0-x|+|x_0-y|-|x-y|),\]\index{Gromov product}
and
\[\ell_{x_0}(x, y)=\left\{\begin{array}{lr}(x, y)_{x_0} & \text{if\ } x\ne y,\\
\infty & \text{otherwise.}\end{array}\right.\] Note that it follows
from the triangle inequality that $(x,
y)_{x_0}$ is non-negative and not greater than $\min(|x_0-x|,
|x_0-y|)$. See Figure~\ref{fig:grprod}.

\begin{figure}
\includegraphics{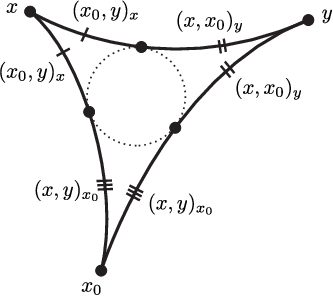}
\caption{Gromov product}\label{fig:grprod}
\end{figure}

\begin{defi}
The space $(\X, d)$ is called \emph{Gromov hyperbolic} if
$\ell_{x_0}$ is a log-scale, i.e., if there exists $\delta>0$ such that
\begin{equation}
\label{eq:ellhyperb}
\ell_{x_0}(x, z)\ge\min(\ell_{x_0}(x, y), \ell_{x_0}(y,
z))-\delta
\end{equation} for all $x, y, z\in\X$.\index{Gromov hyperbolic space}
\end{defi}

Our definition is slightly different from the classical one (which
uses $(x, y)_{x_0}$ instead of $\ell_{x_0}(x, y)$), but is
equivalent to it. Namely, we have the following.

\begin{lemma}
Inequality~\eqref{eq:ellhyperb} is equivalent to the inequality
\begin{equation}
\label{eq:gromovhyperb} (x, z)_{x_0}\ge\min((x, y)_{x_0}, (y,
z)_{x_0})-\delta
\end{equation}
\end{lemma}

\begin{proof}
If $x=z$, then $\ell_{x_0}(x, z)=\infty$, hence inequality~\eqref{eq:ellhyperb} is
true. In inequality~\eqref{eq:gromovhyperb} we will have $\min((x, y)_{x_0}, (y,
z)_{x_0})=(x, y)_{x_0}=(y, z)_{x_0}=1/2(|x_0-y|+|x_0-x|-|x-y|)$ and
$(x, z)_{x_0}=|x_0-x|$. Then inequality~\eqref{eq:gromovhyperb} is
\[|x_0-x|>1/2(|x_0-y|+|x_0-x|-|x-y|)-\delta,\]
and it is equivalent to
\[|x_0-x|+|x-y|>|x_0-y|-2\delta,\]
which is always true.

If $x=y$, then $\ell_{x_0}(x, y)=\infty$, hence
$\min(\ell_{x_0}(x, y), \ell_{x_0}(y, z))=\ell_{x_0}(x, z)$,
therefore~\eqref{eq:ellhyperb} is true.

In the case of~\eqref{eq:gromovhyperb} we have $(y, z)_{x_0}=(x, z)_{x_0}$,
therefore
\[(x, z)_{x_0}=(y, z)_{x_0}\ge\min((x, y)_{x_0}, (y, z)_{x_0})>
\min((x, y)_{x_0}, (y, z)_{x_0})-\delta.\]
\end{proof}

A proof of the following proposition
can be found in~\cite[Corollary~1.1B]{gro:hyperb} and
in~\cite[Proposition~2.2, p.~10]{ghysetaleds:geometrical}.

\begin{proposition}
If $(\X, d)$ is $\delta$-hyperbolic with respect to $x_0$, then it
is $2\delta$-hyperbolic with respect to any point $x_1\in\X$.
\end{proposition}

Note that
\[(x, y)_{x_0}\le (x, y)_{x_1}+|x_0-x_1|,\]
hence the log-scales $\ell_{x_0}$ and $\ell_{x_1}$ are Lipschitz
equivalent.

\subsection{Boundary and the Busemann cocycle}
\label{sss:busemann}

\begin{defi}
Let $(\X, d)$ be a Gromov hyperbolic metric space.
The \emph{boundary} $\partial\X$ of $\X$ is the complement
\index{boundary!of a Gromov hyperbolic space}
of $\X$ in its completion with respect to the log-scale
$\ell_{x_0}$.
\end{defi}

Since the Lipschitz class of the log-scale
$\ell_{x_0}$ does not depend on $x_0$, the boundary $\partial\X$
and the Lipschitz class of the extension of $\ell_{x_0}$ to
$\partial\X$ do not depend on the choice of $x_0$.

Note that the topology defined on $\X$ by the log-scale $\ell_{x_0}$ is
often different from the original topology (defined by the metric
$|x-y|$). In such cases, the completion of $\X$ with respect to the
log-scale $\ell_{x_0}$ is not homeomorphic to the classical
compactification of $\X$ by its boundary.

Let $\xi\in\partial\X$ be the limit of an
$\ell_{x_0}$-Cauchy sequence $x_n\in\X$. For any pair of points
$x, y\in\X$ consider
\[|x-x_n|-|y-x_n|=
2\ell_{x_0}(y, x_n)-2\ell_{x_0}(x, x_n)+|x-x_0|-|y-x_0|.\]

It follows from the last equality that all partial limits of the sequence
$|x-x_n|-|y-x_n|$ are, up to an additive constant, equal to $2(\ell_{x_0}(y,
\xi)-\ell_{x_0}(x, \xi))+|x-x_0|-|y-x_0|$.

\begin{defi}
\label{def:busemann}
The function
\[\beta_\xi(x, y)\doteq\lim_{x_n\to\xi}(|x-x_n|-|y-x_n|)\doteq
2(\ell_{x_0}(y, \xi)-\ell_{x_0}(x, \xi))+|x-x_0|-|y-x_0|\]
is called the \emph{Busemann quasi-cocycle} associated with
$\xi\in\partial\X$.\index{Busemann quasi-cocycle!on a hyperbolic space}
\end{defi}

The Busemann quasi-cocycle measures in some sense the difference
$|x-\xi|-|y-\xi|$ of distances from $x$ and $y$ to $\xi$, see
Figure~\ref{fig:buseman}.

\begin{figure}
\centering
\includegraphics{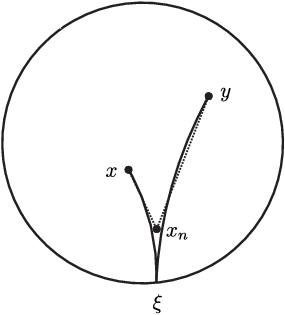}
\caption{Busemann cocycle}\label{fig:buseman}
\end{figure}

It follows directly from the definitions that for any three points
$x, y, z\in\X$ we have
\begin{equation}\label{eq:cocycle}\beta_\xi(x, z)\doteq\beta_\xi(x,
y)+\beta_\xi(y, z).\end{equation}

Let us recall the definition of a natural extension of the Gromov
product to the boundary and a natural metric on the space
$\partial\X\setminus\{\omega\}$ (see~\cite[Section~8.1]{ghys-h:gromov}). Define
\[\ell_{\xi, x_0}(x, y)=\frac 12(\beta_\xi(x, x_0)+\beta_\xi(y, x_0)-|x-y|).\]

Note that
\begin{multline*}\ell_{\xi, x_0}(x, y)\doteq
\lim_{x_n\to\xi}\left(\frac 12(|x-x_n|+|y-x_n|-|x-y|)-|x_0-x_n|\right)\doteq\\
\lim_{x_n\to\xi}\left(\ell_{x_n}(x, y)-|x_0-x_n|\right).
\end{multline*}
It follows that the Lipschitz class of $\ell_{\xi, x_0}$ does not
depend on $x_0$.

\begin{proposition}
\label{prop:completionhalfplane}
Let $\X$ be a hyperbolic metric space. Then the function
$\ell_{\xi, x_0}$ is a log-scale for all $\xi\in\partial\X$ and $x_0\in\X$.

The identity map $\X\arr\X$ extends to a homeomorphism from the
completion of $\X$ with respect to $\ell_{\xi, x_0}$ to the space
$\X\cup\partial\X\setminus\{\xi\}$.
\end{proposition}

We will denote $\partial\X_\xi=\partial\X\setminus\{\xi\}$.

\begin{proof}
Let us show at first that $\ell_{\xi, x_0}$ is a log-scale. For
any $x, y, z\in\X$ and $x_n\to\xi$ we have
\begin{multline*}\ell_{\xi, x_0}(x, z)\doteq\lim_{n\to
\infty}(\ell_{x_n}(x, z)-|x_0-x_n|)\ge\\
\lim_{n\to\infty}(\min(\ell_{x_n}(x, y), \ell_{x_n}(y, z))-|x_0-x_n|)-\delta\doteq\\ \min\left(\lim_{n\to\infty}\ell_{x_n}(x,
y)-|x_0-x_n|, \lim_{n\to\infty}\ell_{x_n}(y, z)-|x_0-x_n|\right)\doteq\\
\min(\ell_{\xi, x_0}(x, y), \ell_{\xi, x_0}(y, z)).
\end{multline*}

We have
\begin{multline*}
2\ell_{x_0}(x, y)-2\ell_{\xi, x_0}(x, y)=\\
|x-x_0|+|y-x_0|-\beta_\xi(x, x_0)-\beta_\xi(y,
x_0)=\\ (|x-x_0|-\beta_\xi(x, x_0))+(|y-x_0|-\beta_\xi(y,
x_0)),\end{multline*}
and
\begin{multline*}
|x-x_0|-\beta_\xi(x, x_0)=\lim_{n\to\infty}|x-x_0|-|x-x_n|+|x_0-x_n|=\\
\lim_{n\to\infty}2\ell_{x_0}(x, x_n)\doteq 2\ell_{x_0}(x, \xi).
\end{multline*} It follows that
\[\ell_{x_0}(x, y)-\ell_{\xi, x_0}(x, y)\doteq\ell_{x_0}(x,
\xi)+\ell_{x_0}(y, \xi),\] or
\[\ell_{\xi, x_0}(x, y)\doteq\ell_{x_0}(x, y)-\ell_{x_0}(x,
\xi)-\ell_{x_0}(y, \xi).\]

Consider for an arbitrary number $r$ the set $C_r$ of points
$\zeta\in\X\cup\partial\X$ such that $\ell_{x_0}(\zeta, \xi)<r$. It is
complement of the neighborhood $\{x\in\X\;:\;\ell_{x_0}(x, \xi)\ge
r\}$ of $\xi\in\X\cup\partial\X$. Then for any $x, y\in C_r$ we have
\[\ell_{\xi, x_0}(x, y)\doteq\ell_{x_0}(x, y)-\ell_{x_0}(x,
\xi)-\ell_{x_0}(y, \xi)>\ell_{x_0}(x, y)-2r.\] We also have
\[\ell_{\xi, x_0}(x, y)\le\ell_{x_0}(x, y)+k,\] for some constant $k$,
since $\ell_{x_0}$ is non-negative on $\X$.

Consequently, $\ell_{\xi, x_0}$ and $\ell_{x_0}$ are Lipschitz
equivalent on $C_r$, hence they are Lipschitz equivalent on
complement of any neighborhood of $\xi$. This finishes the proof
of the proposition.
\end{proof}

\subsection{Hyperbolic graphs}

Let $\Gamma$ be a connected graph. 
A \emph{geodesic}\index{geodesic path} in $\Gamma$
connecting vertices $v_1$ and $v_2$ is a shortest path (i.e., a path
with minimal numbers of edges) from $v_1$ to
$v_2$. Distance between two vertices is, by definition, the length
of a geodesic path connecting them.
We will consider (and denote) paths in $\Gamma$ as sequences of vertices.

The graph $\Gamma$ has \emph{bounded degree}\index{bounded degree} if there is a number
$k$ such that every vertex of $\Gamma$ belongs to not more than
$k$ edges.

\begin{defi}
\label{def:quasiisometry}
A subset $N$ of a metric space $\Gamma$ is called a \emph{net} \index{net} if
there exists a constant $\Delta>0$ such that for every point $x\in\Gamma$
there exists a point $y\in N$ such that $|x-y|<\Delta$.

A map $f:\Gamma_1\arr\Gamma_2$ is a \emph{quasi-isometry}
\index{quasi-isometry} if
$f(\Gamma_1)$ is a net in $\Gamma_2$ and there
exist $\Lambda>1$ and $\Delta>0$ such that
\[\Lambda^{-1}\cdot |x-y|-\Delta\le |f(x)-f(y)|\le\Lambda\cdot
|x-y|+\Delta\]
for all $x, y\in\Gamma_1$.
\end{defi}

\begin{defi}
Let $(\X, |\cdot|)$ be a metric space. A sequence $t_0, t_1, \ldots,
t_n$ of points of $\X$ is a \emph{$(\Lambda,
\Delta)$-quasi-geodesic}\index{quasi-geodesic} if $|t_i-t_{i+1}|\le\Delta$ and
$|i-j|\le\Lambda\cdot |t_i-t_j|$ for all $i, j\in\{0, 1, \ldots, n\}$.

A metric space $(\X, |\cdot|)$ is \emph{quasi-geodesic} if there exist
positive numbers $\Lambda$ and $\Delta$ such that for any two points $x,
y\in\X$ there exists an $(\Lambda, \Delta)$-quasi-geodesic sequence
$x=t_0, t_1, \ldots, t_n=y$.
\end{defi}

A metric space is quasi-geodesic if
and only if it is quasi-isometric to a geodesic space, i.e., to a
space in which any two points $x, y$ can be connected by an isometric arc
$\gamma:[0, |x-y|]\arr\X$. Every
quasi-geodesic space is quasi-isometric to a graph.

The following is a classical result of the theory of Gromov
hyperbolic spaces (see~\cite[Theorem~5.12]{ghys-h:gromov}
and~\cite[Proposition~2.1, p.~16]{ghysetaleds:geometrical}).

\begin{theorem}
\label{th:thintriangles}
Let $\Gamma_1$ and $\Gamma_2$ be quasi-isometric graphs of bounded
degree. If $\Gamma_1$ is Gromov hyperbolic, then so is $\Gamma_2$.

A graph is Gromov hyperbolic if and only if there exists $\delta$
such that in any geodesic triangle each side belongs to the
$\delta$-neighborhood of the union of the other two sides. (One
says that such a triangle is \emph{$\delta$-thin}, see Figure~\ref{fig:dtriangle}).
\end{theorem}

\begin{figure}
\includegraphics{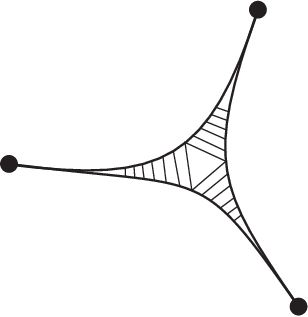}
\caption{$\delta$-thin triangle}\label{fig:dtriangle}
\end{figure}

Another important property of hyperbolic spaces is described in the following
theorem, which can be found, for instance,
in~\cite[Theorem~5.11]{ghys-h:gromov}.

\begin{theorem}
\label{th:rigidquasigeodesics}
Suppose that $\Gamma$ is a hyperbolic graph. Then for any pair of
positive numbers $\Lambda, \Delta$ there exists $\delta$ such that if $t_0,
t_1, \ldots, t_n$ and $s_0, s_1, \ldots, s_m$ are two $(\Lambda,
\Delta)$-quasi-geodesics connecting $x=t_0=s_0$ to
$y=t_n=s_m$, then the sequences $t_i$ and $s_i$ are at most
$\delta$ apart (i.e., for every point of one sequence there exists
a point of the other on distance less than $\delta$).
\end{theorem}

In particular, if $\Gamma$ is a hyperbolic graph, then there
exists $\delta$ such that every $(\Lambda, \Delta)$-quasi-geodesic is at most
$\delta$ away from a geodesic path connecting the same points.

The following proposition is a direct corollary of~\ref{th:rigidquasigeodesics}.

\begin{proposition}
\label{pr:thininftriangles}
Let $\Gamma$ be a hyperbolic graph, and let $\xi_1, \xi_2$ be
arbitrary points of the space $\Gamma\cup\partial\Gamma$. Then
there exists a geodesic path $\gamma$ connecting $\xi_1$ and
$\xi_2$, i.e., a geodesic path in $\Gamma$ such that its beginning (resp.\
end) is either equal to the vertex $\xi_1\in\Gamma$ (resp.\ $\xi_2$), or
converges to $\xi_1\in\partial\Gamma$ (resp.\ to $\xi_2$).

There exists $\delta$ such that for any two points $\xi_1,
\xi_2\in\Gamma\cup\partial\Gamma$  any two geodesic paths
$\gamma_1$ and $\gamma_2$ connecting $\xi_1$ and $\xi_2$ are on
distance not more than $\delta$ from each other.
\end{proposition}

\subsection{Hyperbolic graphs directed towards a point of the boundary}
\begin{defi}
\label{def:almcoc} Let $\Gamma$ be a graph of bounded degree. A
function $\coc:\Gamma\times\Gamma\arr\R$ is called a
\emph{quasi-cocycle}\index{quasi-cocycle!on a graph}
if there exist numbers $\Delta>0$ and $\eta>0$ such that
$|\coc(v_1, v_2)|\le\Delta$ for all pairs of adjacent vertices $v_1,
v_2\in\Gamma$, and
\[\coc(v_1, v_2)+\coc(v_2, v_3)-\eta\le\coc(v_1, v_3)\le\coc(v_1,
v_2)+\coc(v_2, v_3)+\eta,
\]
for all $v_1, v_2, v_3\in\Gamma$.
\end{defi}

Note that under the conditions of the definition we have
\begin{equation}
\label{eq:coclipschitz} |\coc(v_1, v_2)|\le |v_1-v_2|(\Delta+\eta),
\end{equation}
since for any path $u_0=v_1, u_1, \ldots, u_n=v_2$ we have
\[\coc(v_1, v_2)\le\coc(u_0, u_1)+\coc(u_1,
u_2)+\cdots+\coc(u_{n-1}, u_n)+(n-1)\eta<n\Delta+n\eta\] and
\begin{multline*}
\coc(v_1, v_2)\ge\coc(u_0, u_1)+\coc(u_1,
u_2)+\cdots+\coc(u_{n-1}, u_n)-(n-1)\eta\ge\\
-n\Delta-(n-1)\eta\ge-n(\Delta+\eta).
\end{multline*}

An example of a quasi-cocycle on a graph is the Busemann quasi-cocycle
$\beta_\xi$ associated with a point of the boundary of a hyperbolic
graph (see Definition~\ref{def:busemann}).

If $\coc_1$ is a quasi-cocycle, and
$\coc_2:\Gamma\times\Gamma\arr\R$ is such that $|\coc_1-\coc_2|$
is bounded, then $\coc_2$ is obviously also a quasi-cocycle.

\begin{defi}
Two quasi-cocycles $\coc_1$ and $\coc_2$ are said to be
\emph{strongly equivalent} if $|\coc_1-\coc_2|$ is uniformly
bounded.\index{strongly equivalent quasi-cocycles}
\end{defi}

Let $\Gamma, \Gamma'$ be hyperbolic graphs of bounded degree, and
let $F:\Gamma\arr\Gamma'$ be a quasi-isometry. We will also denote by $F$
also the extension of $F$ onto the boundaries of the graphs. Let $\beta_\xi$ be the
Busemann quasi-cocycle on $\Gamma$ associated with a point
$\xi\in\partial\Gamma$, let $\beta_{F(\xi)}'$ be the Busemann
quasi-cocycle associated with the corresponding point $F(\xi)\in\partial\Gamma'$.

\begin{lemma}
\label{lem:betaxi}
There exist constants $C>1, D>0$ such that the following condition is satisfied.
Let $(v_0, v_1, \ldots)$ be a geodesic path of vertices in $\Gamma$ converging to
$\xi$. Then for all $0\le i\le j$ we have
\[C^{-1}\beta_\xi(v_i, v_j)-D\le\beta_{F(\xi)}'(F(v_i), F(v_j))\le
C\beta_\xi(v_i, v_j)+D.\]
\end{lemma}

\begin{proof}
Note that $\beta_\xi(v_i, v_j)\doteq j-i$ for all $i, j$,
since $(v_0, v_1, \ldots)$ is a geodesic path converging to $\xi$.
The sequence $(F(v_0), F(v_1), \ldots)$ is quasi-geodesic path in $\Gamma'$
converging to $F(\xi)$. By Theorem~\ref{th:rigidquasigeodesics},
there exits $\Delta>0$, depending only on $\Gamma, \Gamma'$, and $F$, such
that there exits a geodesic path $(w_0, w_1, \ldots)$ in $\Gamma'$ 
on distance less than $\Delta$ from $(F(v_0), F(v_1), \ldots)$.
Then for any $0\le i\le j$, there exist $i', j'$ such that
$|F(v_i)-w_{i'}|, |F(v_j)-w_{j'}|\le \Delta$. If $j-i$ is big enough, then
$j'-i'$ is positive, and
\[\beta_{F(\xi)}'(F(v_i), F(v_j))\doteq\beta_{F(\xi)}'(w_{i'},
w_{j'})\doteq|w_{i'}-w_{j'}|.\]
It follows that there exist constants $C>1$ and $D_0>0$ such that
\[C^{-1}\beta_\xi(v_i, v_j)\doteq C^{-1}|v_i-v_j|-D_0\le |w_{i'}-w_{j'}|\le C|v_i-v_j|+D_0\doteq
C\beta_\xi(v_i, v_j),\]
which finishes the proof of the lemma.
\end{proof}

\begin{defi}
\label{def:genbusemann}
Let $\Gamma$ be a hyperbolic graph of bounded valency. Let
$\xi\in\partial\Gamma$. A quasi-cocycle $\beta$ is a
\emph{(generalized) Busemann quasi-cocycle}
\index{generalized!Busemann quasi-cocycle} associated with $\xi$ if
there exist constants $C>1$ and $D>0$ such that for all geodesic
vertex-paths $(v_0, v_1, \ldots)$ converging to $\xi$, and all $0\le i\le
j$ we have
\[C^{-1}(j-i)-D_0\le\beta(v_i, v_j)\le C(j-i)+D_0.\]
\end{defi}

The same arguments as in the proof of Lemma~\ref{lem:betaxi} show that
if $\beta$ is a generalized Busemann quasi-cocycle on $\Gamma$, and
$F:\Gamma'\arr\Gamma$ is a quasi-isometry, then $\beta'(v,
u)=\beta(F(v), F(u))$ is a generalized Busemann quasi-cocycle on
$\Gamma'$. In other words, the notion of a generalized Busemann
quasi-cocycle is quasi-isometry invariant.

\begin{lemma}
\label{lem:quasigeodesic}
Let $\Gamma$ be a directed graph and let $\coc$ be a quasi-cocycle on $\Gamma$.
Let $\Delta$ and $\eta$ be as in Definition~\ref{def:almcoc}.
Suppose that for any arrow of $\Gamma$ starting at a vertex $v$
and ending in a vertex $u$ we have $\coc(v, u)>2\eta$.

Then every directed path in $\Gamma$ is a $((\Delta+\eta)/\eta,
1)$-quasi-geodesic.
\end{lemma}

\begin{proof}
Let $(\ldots, x_1, x_2, \ldots)$ be an infinite or a finite
directed path in $\Gamma$.

We have for any $i<j$:
\begin{multline*}2\eta(j-i)-(j-i-1)\eta<\\ \coc(x_i,
x_{i+1})+\cdots+\coc(x_{j-1}, x_j)-(j-i-1)\eta\le\coc(x_i,
x_j)\le\\ \coc(x_i, x_{i+1})+\cdots+\coc(x_{j-1},
x_j)+(j-i-1)\eta<\\ \Delta(j-i)+(j-i-1)\eta,\end{multline*} hence
\[\eta(j-i)<\coc(x_i, x_j)<(\Delta+\eta)(j-i)+\eta.\]

By~\eqref{eq:coclipschitz} we have for any $i<j$
\[|x_i-x_j|\ge\frac 1{\Delta+\eta}|\coc(x_i, x_j)|=
\frac 1{\Delta+\eta}\coc(x_i, x_j)>\frac\eta{\Delta+\eta}(j-i),\] which
implies that the path is quasi-geodesic.
\end{proof}

\begin{theorem}
\label{th:contraction}
Let $\Gamma$ be a directed graph such that every vertex of $\Gamma$
has at least one outgoing arrow. Let $\coc$ be a quasi-cocycle on
$\Gamma$,
and let $\Delta$ and $\eta$ be as in Definition~\ref{def:almcoc}.
Suppose that for every arrow of $\Gamma$ starting at $v$
and ending in $u$ we have $\coc(v, u)>2\eta$. Fix a number $\Delta_1\ge\Delta+\eta$.

Then the following conditions are equivalent.
\begin{enumerate}
\item The graph $\Gamma$ is hyperbolic and $\coc$ is a generalized
Busemann quasi-cocycle $\beta_\omega$ associated with a
point $\omega\in\partial\Gamma$.
\item There exists $\rho_0>0$ such that for any $m>0$ there exists
$k_m>0$ such that for any directed paths $(u_0, u_1, \ldots)$ and
$(v_0, v_1, \ldots)$ with $|u_0-v_0|\le m$ we have $|u_i-v_j|<\rho_0$
whenever $|\coc(u_i, v_j)|\le\Delta_1$, $i>k_m$, and $j>k_m$.
\end{enumerate}
\end{theorem}

Note that condition (1) of the theorem together with the condition
$\coc(v, u)>2\eta$ mean that all arrows of the graph $\Gamma$ are
directed towards the point $\omega\in\partial\Gamma$ in the sense that
any infinite directed path converges to $\omega$.

The proof of Theorem~\ref{th:contraction} is very similar to the
proof of hyperbolicity of the \emph{selfsimilarity complex} given
in~\cite{nek:hyplim} (see~\cite[Theorem~3.9.6]{nek:book}). See also a
similar criterion of hyperbolicity of \emph{augmented trees}
in~\cite{kaim:augtree}.

\begin{proof}
Let us prove that (1) implies (2). Let $(u_0, u_1, \ldots)$ and $(v_0,
v_1, \ldots)$ be two directed paths such that $|u_0-v_0|\le m$. Both
paths converge to $\omega$ and are quasi-geodesic. Connect $u_0$ to
$v_0$ by a geodesic path $\gamma$, and consider the obtained triangle
with vertices $u_0, v_0, \omega$. There exists $\delta$ depending only
on $\Gamma$ such that all such triangles are $\delta$-thin (see
Theorem~\ref{th:thintriangles} and~\ref{pr:thininftriangles}). In
particular, the path $(u_0, u_1, \ldots)$ belongs to the
$\delta$-neighborhood of the union of the paths $\gamma$ and $(v_0,
v_1, \ldots)$. If $|v_0-v_i|>m+\delta$, then distance to $v_i$ from
every point of $\gamma$ is greater than $\delta$, hence there exists
$u_j$ such that $|v_i-u_j|\le\delta$. Similarly, if
$|u_0-u_i|>m+\delta$, then distance from $u_i$ to any point of
$\gamma$
is greater than $\delta$, hence there exist $v_j$ such that
$|v_j-u_i|\le\delta$. The paths $(v_0, v_1, \ldots)$ and $(u_0, u_1,
\ldots)$ are quasi-geodesics (see
Lemma~\ref{lem:quasigeodesic}), hence there exists $k_m$ not depending
on $v_0$ and $u_0$ such that for every $i>k_m$ there exists $j$ such
that $|v_i-u_j|\le\delta$, and for every $i>k_m$ there exists $j$
such that $|u_i-v_j|\le\delta$.

Suppose that $i>k_m$ and $j>k_m$ are such that $|\coc(v_i,
u_j)|<\Delta_1$. Then there exists $l$ such that $|v_i-u_l|\le\delta$ and
there exists $m$ such that $|v_m-u_j|\le\delta$. Then $|\coc(v_i,
u_l)|\le\delta(\Delta+\eta)$ and $|\coc(v_m,
u_j)|\le\delta(\Delta+\eta)$. Consequently,
\[|\coc(v_i, v_m)|\le|\coc(v_i, u_j)|+|\coc(u_j, v_m)|+\eta\le
\Delta_1+\delta(\Delta+\eta)+\eta.\]
We get a uniform bound $M=(\Delta_1+\delta(\Delta+\eta)+\eta)/\eta$ for
the difference $|i-m|$. It
follows that \[|v_i-u_j|\le |v_i-v_m|+|v_m-u_j|\le
M+\delta,\]
so that we can take $\rho_0=M+\delta$. Note that $\rho_0$ depends only on
$\Gamma$ and $\Delta_1$.

Let us prove now that (2) implies (1).
Choose a vertex $x_0\in\Gamma$, and consider the functions
\[\lambda(v)=\coc(v, x_0),\qquad\coc_1(v, u)=\coc(v, x_0)-\coc(u,
x_0)=\lambda(v)-\lambda(u).\]
Then we have $|\coc_1(u, v)-\coc(u, v)|\le\eta$ for any pair $u,
v\in\Gamma$. For any two adjacent vertices $u$ and $v$ of $\Gamma$, we have
$|\coc_1(u, v)|\le\Delta+\eta\le\Delta_1$. For every directed edge
$(u, v)$, we have $\coc_1(u, v)\ge\eta$. We also have
\[\coc_1(u, w)=\coc_1(u, v)+\coc_1(v, w)\]
for all triples of vertices $u, v, w$ of $\Gamma$, and
$\coc_1(u, u)=0$ for all $u\in\Gamma$.

Let $(u_0, u_1, \ldots)$ and $(v_0, v_1, \ldots)$ be directed
paths in $\Gamma$.
Suppose that $0\le\coc_1(u_i, v_j)\le\Delta_1$ for some $i, j$.
If $0\le\coc_1(u_i, v_i)\le\Delta_1-\eta$, then $-\eta\le\coc(u_i,
v_i)\le\Delta_1$. If $\Delta_1-\eta<\coc_1(u_i, v_i)\le\Delta_1$, then
\[\coc_1(u_i, v_j)-\Delta_1\le\coc_1(u_i, v_{j-1})=
\coc_1(u_i, v_j)+\coc_1(v_j, v_{j-1})\le
\coc_1(u_i, v_j)-\eta,\]
hence $-\eta\le\coc_1(u_i, v_{j-1})\le\Delta_1-\eta$, so that
$-\Delta_1<-2\eta\le\coc(u_i, v_{j-1})\le\Delta_1$. We have proved that if
$0\le\coc_1(u_i, v_j)\le\Delta_1$, then either $|\coc(u_i, v_j)|\le\Delta_1$, or
$|\coc(u_i, v_{j-1})|\le\Delta_1$. Similarly, if
$0\le\coc_1(v_j, u_i)=-\coc_1(u_i, v_j)\le\Delta_1$, then either $|\coc(u_i,
v_j)|\le\Delta_1$, or $|\coc(u_{i-1}, v_i)|\le\Delta_1$.

It follows that condition (2) of the theorem implies the following
condition (when we take $\rho_1=\rho_0+1$).
\begin{itemize}
\item[(2)'] There exists $\rho_1>0$ such that for any $m>0$ there exists
$k_m$ such that for any pair of directed paths $(u_0, u_1, \ldots)$
and $(v_0, v_1, \ldots)$ satisfying $|u_0-v_0|\le m$, we
have $|u_i-v_j|<\rho_1$, whenever $|\coc_1(u_i, v_j)|\le\Delta_1$,
$i>k_m$, and $j>k_m$.
\end{itemize}

Denote $k=k_{2\rho_1}$ and let $\Delta_2$ and $r$ be positive integers
such that $r$ is divisible by 4, and
\[\Delta_2>(k+1)\Delta_1,\qquad r\ge\frac{4\Delta_2}{\rho_1\eta}.\]

Consider a new graph $\Gamma_1$ with the set of vertices equal to the
set of vertices of $\Gamma$ in which two different vertices $u, v$ are
connected by an edge if and only if one of the following conditions holds
\begin{itemize}
\item there exists a sequence $(u=w_0, w_1, \ldots, v=w_n)$ of
vertices of $\Gamma$ and a number $l\in\Z$ such that $n\le r$, $|w_i-w_{i+1}|\le \rho_1$, and
$l\Delta_2\le\lambda(w_i)<(l+1)\Delta_2$ for all $i=0, 1, \ldots, n$;
\item there exists a \emph{directed} path $(u=w_0, w_1, \ldots,
v=w_n)$ in $\Gamma$ and a number $l\in\Z$ such that
$l\Delta_2\le\lambda(u)<(l+1)\Delta_2$ and $(l-1)\Delta_2\le\lambda(v)\le
l\Delta_2$.
\end{itemize}

\begin{figure}
\centering
\includegraphics{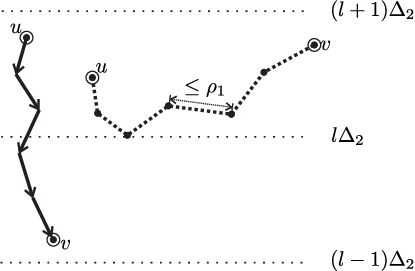}
\caption{Edges of $\Gamma_1$}\label{fig:gamma1}
\end{figure}

We will imagine the graphs $\Gamma$ and $\Gamma_1$ drawn in such a way
that $\lambda(v)$ is the height of the position of the vertex $v$. Then
the arrows and directed paths of $\Gamma$ go down.
See Figure~\ref{fig:gamma1}, where the two cases for edges of
$\Gamma_1$ are shown.

The edges of $\Gamma_1$ defined by the
first condition are called \emph{horizontal}.
The edges defined by the second conditions are called
\emph{vertical}. For a path $(v_1, v_2, \ldots, v_n)$ in
$\Gamma_1$, we say that a vertical edge $(v_i, v_{i+1})$ is
\emph{descending}, if $\lambda(v_{i+1})<\lambda(v_i)$ and
\emph{ascending} otherwise.

Let us call the number $\lfloor\lambda(v)/\Delta_2\rfloor$ the
\emph{level} of the vertex $v$. It follows from the definition of the
edges of $\Gamma_1$ that if $(v, u)$ is a horizontal edge, then $v$
and $u$ belong to the same level, and if $(v, u)$ is a descending
vertical edge, then the level of $u$ is one less than the level of
$v$.

We will denote by $|u-v|$ and $|u-v|_1$ distances
between the vertices $u$ and $v$ in $\Gamma$ and $\Gamma_1$ respectively.

\begin{lemma}
\label{lem:quasiisom}
The identity map between the sets of vertices of $\Gamma$ and
$\Gamma_1$ is a bi-Lipschitz equivalence, i.e., there exists a
constant $\Lambda$ such that $\Lambda^{-1}\cdot |u-v|\le
|u-v|_1\le\Lambda\cdot |u-v|$ for
all $u, v\in\Gamma$.
\end{lemma}

\begin{proof}
Suppose that $(u, v)$ is a directed edge of $\Gamma$. Let
$l=\lfloor\lambda(u)/\Delta_2\rfloor$. Then $\eta\le\nu_1(u,
v)=\lambda(u)-\lambda(v)\le\Delta_1<\Delta_2$,
hence either $\lfloor\lambda(v)/\Delta_2\rfloor=l$, or
$\lfloor\lambda(v)/\Delta_2\rfloor=l-1$. In the first case $(u, v)$ is a
horizontal edge of $\Gamma_1$, in the second case it is a vertical
edge of $\Gamma_1$. Consequently, the set of edges of $\Gamma$ is a
subset of the set of edges of $\Gamma_1$, hence $|u-v|_1\le |u-v|$
for all $u, v\in\Gamma$.

The distance in $\Gamma$ between vertices connected by a
horizontal edge of $\Gamma_1$ is bounded from above by $r\rho_1$. If $(u,
v)$ is a vertical edge in $\Gamma_1$, then $\nu_1(u,
v)=\lambda(u)-\lambda(v)\le 2\Delta_2$, and
\[\nu_1(u, v)=\nu_1(w_0, w_1)+\nu_1(w_1,
w_2)+\cdots+\nu_1(w_{n-1}, w_n)\ge n\eta,\] hence
$n\le\frac{2\Delta_2}{\eta}\le\frac{r\rho_1}2$.
Consequently, $|u-v|\le r\rho_1|u-v|_1$ for all $u, v\in\Gamma_1$.
\end{proof}

\begin{lemma}
\label{lem:contraction}
Let $(v_0, v_1, \ldots, v_n)$ be a path of horizontal edges in
$\Gamma_1$. Let $(u_0, v_0)$ and $(u_n, v_n)$ be ascending vertical
edges. Then the vertices $u_0$ and $u_n$ can be connected by a path of
horizontal edges of $\Gamma_1$ of length not more than
$\lceil (n+1)/2\rceil$.
\end{lemma}

\begin{proof}
Let $l$ be the level of the vertices $v_i$. Then the vertices $u_0$
and $u_n$ belong to the level $l-1$.

It follows from the definition of horizontal edges that
there exists a sequence $z_0, z_1,
\ldots, z_m$ of vertices of $\Gamma$ such that $z_0=v_0$, $z_m=v_n$,
$|z_i-z_{i+1}|\le 2\rho_1$, $m\le rn/2$, and
$\lfloor\lambda(z_i)/\Delta_2\rfloor=l$ for all $i=0, 1, \ldots, m$.

For every $i=0, 1, \ldots, m$ find an infinite directed path $(w_{0, i},
w_{1, i}, \ldots)$ in $\Gamma$ such that $w_{0, i}=z_i$, and there
exist $t_0$ and $t_m$ such that $w_{t_0, 0}=u_0$ and $w_{t_m, m}=u_n$.
(See Figure~\ref{fig:contraction}.)

\begin{figure}
\centering
\includegraphics{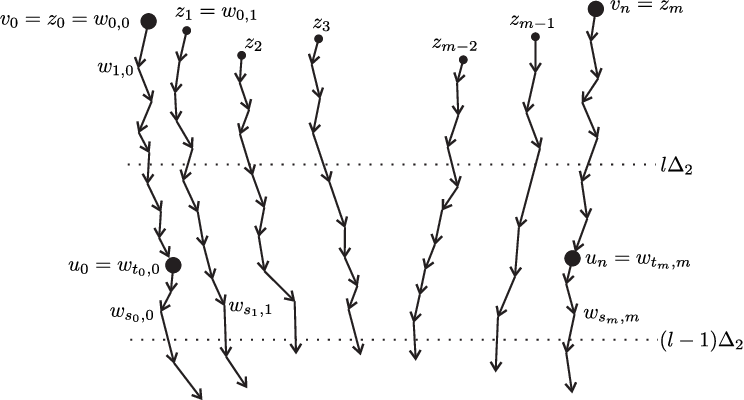}
\caption{} \label{fig:contraction}
\end{figure}

For every $i$ let $s_i$ be the maximal index such that
$\lambda(w_{s_i, i})>(l-1)\Delta_2$. Note that $s_0\ge t_0$ and $s_m\ge t_m$.

We have $\lambda(w_{s_i, i})-\lambda(w_{s_i+1, i})=\nu_1(w_{s_i, i},
w_{s_i+1, i})\le\Delta_1$. If
$\lambda(w_{s_i, i})>(l-1)\Delta_2+\Delta_1$, then
\[\lambda(w_{s_i+1, i})\ge\lambda(w_{s_i})-\Delta_1>(l-1)\Delta_2,\]
which contradicts the choice of $s_i$. Consequently,
\[(l-1)\Delta_2\le\lambda(w_{s_i, i})\le (l-1)\Delta_2+\Delta_1.\]
The length $s_i$ of the path $(w_{0, i}, w_{1, i},\ldots, w_{s_i, i})$
is not less than $\frac{\Delta_2-\Delta_1}{\eta}>k=k_{2\rho_1}$.
We have
\[|\coc_1(w_{s_i, i}, w_{s_{i+1}, i+1})|\le\Delta_1,\qquad |w_{0, i}-w_{0,
  i+1}|\le 2\rho_1,\]
hence by condition (2)' we have
\[|w_{s_i, i}-w_{s_{i+1}, i+1}|\le \rho_1.\]

The lengths of the $\Gamma$-paths
$\gamma_1=(u_0=w_{t_0, 0}, w_{t_0+1, 0}, \ldots, w_{s_0,
0})$ and $\gamma_2=(u_n=w_{t_m, m}, w_{t_m+1, m}, \ldots, w_{s_m, m})$ are not
greater than $\frac{\Delta_2}{\eta}\le\frac{r\rho_1}4$. We can
therefore split each of them into segments of length $\le \rho_1$ so
that we get not more
than $r/4$ segments. Appending these sequences of segments to
the sequence $w_{s_0, 0}, w_{s_1, 1}, \ldots, w_{s_m, m}$ we get
a sequence of vertices of $\Gamma$ consisting of at most
$m+r/2\le rn/2+r/2=r(n+1)/2$
segments of length at most $\rho_1$. All these vertices will belong to
the level $l-1$. Consequently, we can find a path
in $\Gamma_1$ consisting of at most $\lceil (n+1)/2\rceil$ horizontal
edges connecting the vertices $u_0$ and $u_n$.
\end{proof}

\begin{lemma}
\label{lem:downup}
If $(v_0, v_1, \ldots, v_n)$ is a path in $\Gamma_1$ of length at
least 2 such that $(v_1, v_2)$ is ascending,
$(v_{n-1}, v_n)$ is descending and all edges between them are
horizontal, then the path is not a geodesic.
\end{lemma}

\begin{proof}
By Lemma~\ref{lem:contraction}, distance between $v_0$ and $v_n$
is not more than $\lceil(n-1)/2\rceil<n$.
\end{proof}

\begin{corollary}
\label{cor:downup}
In any geodesic path of $\Gamma_1$ all descending edges come before
all ascending ones.
\end{corollary}

\begin{lemma}
\label{lem:less6}
A geodesic of $\Gamma_1$ can not have more than 6 consecutive
horizontal edges.
\end{lemma}

\begin{proof}
Let $v_0, v_1, \ldots, v_7$ be a horizontal path. Find descending
edges $(v_0, u_0)$ and $(v_7, u_7)$. Then by
Lemma~\ref{lem:contraction}, distance between $u_0$ and $u_7$ is
not more than $4$, hence the $\Gamma_1$-distance between $v_0$ and $v_7$ is not
more than 6.
\end{proof}

When we talk about distances between paths in $\Gamma$ or $\Gamma_1$
we consider the paths as sets of vertices and use the Hausdorff
distance between them, i.e., the smallest number $R$ such each of the
paths (as sets of vertices) belongs to the $R$-neighborhood of the other.

\begin{lemma}
\label{lem:less7}
A geodesic path of $\Gamma_1$ can not have more than 6 horizontal
edges, and it is on distance (in $\Gamma_1$) not more than 2 from a
path with at most 6 horizontal edges in which all descending edges are
at the beginning, and all ascending edges are
at the end.
\end{lemma}

\begin{proof}
Let $\gamma=(v_0, \ldots, v_n)$ be a geodesic path in $\Gamma_1$.
Suppose that $(v_{i-1}, v_i)$ and $(v_j, v_{j+1})$ for $j> i$
are descending, and all edges between $v_i$ and $v_j$ are
horizontal. Let $(v_i, v)$ be a descending edge. Then, by
Lemma~\ref{lem:contraction}, the distance $|v-v_{j+1}|_1$ is not
more than $\lceil(j-i+1)/2\rceil$. Hence, $j+1-i=|v_i-v_{j+1}|_1
\le\lceil(j-i+1)/2\rceil+1$.

But $\lceil(j-i+1)/2\rceil+1<j+1-i$ for $j-i\ge 3$. It follows that
the number $j-i$ of consecutive horizontal edges between descending
edges is not more than 2.

If $j-i=1$ or $j-i=2$, then $\lceil(j-1+1)/2\rceil=j-i+1$, and we can
replace the segment $v_i, v_{i+1}, \ldots, v_j, v_{j+1}$ of $\gamma$
by $v_i, v$ followed by a sequence of horizontal edges connecting $v$
to $v_{j+1}$ so that we again get a geodesic path, but the segment of
horizontal edges has moved towards the end of $\gamma$.

Similarly, we can not have more than two horizontal edges between two
ascending edges, and if we have one or two ascending horizontal edges
between ascending ones, we can move the horizontal edges towards the
beginning of $\gamma$.

Let us move the first segment of horizontal edges between descending
edges forward, as it is
described above for as long as we can, i.e., until it will meet
another segment of horizontal edges, or an ascending edge. In the
first case we will put the horizontal edges together, and continue
moving them forward. This may happen at most once (if both segments
consist of only one edge). In the same way, move all horizontal edges
between ascending edges towards the beginning of the path. At the end
we get a geodesic path $\gamma'$ that has all descending edges at the
beginning and all ascending edges at the end. It is also easy
to see that the Hausdorff distance between $\gamma'$ and $\gamma$ is at
most 2. The path $\gamma'$ (and hence the path $\gamma$)
can not have more than 6 horizontal paths by Lemma~\ref{lem:less7}.
\end{proof}

Let us say that a path $\gamma$ in $\Gamma_1$ is \emph{$V$-shaped} if
all its descending edges are at the beginning, all its ascending edges
are at the end, and it has not more than 6 horizontal edges. Note that we
allow any of the three sets of edges (descending, ascending, and
horizontal) to be empty.

We have shown above that every geodesic path in $\Gamma_1$ is on
distance not more than $2$ from a $V$-shaped geodesic path connecting
the same vertices.

We say that a $V$-shaped path $\gamma$ has \emph{depth} $l$ if $l$ is
the level of the vertices of its horizontal edges (or the level of the
last vertex of the descending part of $\gamma$, or the level of the
first vertex of the ascending part of $\gamma$).
We say that a $V$-shaped path $\gamma$ is
\emph{proper} if it has the maximal possible depth among all $V$-shaped
paths connecting the same pair of points as $\gamma$.

\begin{lemma}
\label{l:geodesicV}
Let $\gamma_1$ be a geodesic path (in $\Gamma_1$) connecting the
vertices $u$ and $v$, and let $\gamma_2$ be any proper $V$-shaped path
connecting the same vertices. Then distance between $\gamma_1$ and
$\gamma_2$ is not more than 9.
  \end{lemma}

\begin{proof}
Let $(v_1, v_2, \ldots)$ and $(u_1, u_2, \ldots)$ be descending
vertical paths in $\Gamma_1$ such that $v_1=u_1$. Then it follows by
induction from Lemma~\ref{lem:contraction} that $|v_i-u_i|_1\le 1$
for all $i$.

Let $\gamma_1$ and $\gamma_2$ be $V$-shaped paths connecting the same
pair of vertices, and suppose that $\gamma_1$ is geodesic and $\gamma_2$ is
proper. Let $l_2$ be the depth of $\gamma_2$. Then the depth $l_1$ of
$\gamma_1$ is not more than $l_2$. Let $h_1$ and $h_2$ be the lengths
of the horizontal parts of $\gamma_1$ and $\gamma_2$ respectively.

Then the length of the descending part of $\gamma_1$ minus the length of
the descending part of $\gamma_2$ is equal to $l_2-l_1$, and the same
is true for the ascending parts.

It follows that the length of $\gamma_2$ is equal to the length of
$\gamma_1$ minus $2(l_2-l_1)+h_1-h_2$. But $\gamma_1$ is a geodesic
path, hence $2(l_2-l_1)+h_1-h_2\le 0$. Consequently, $l_2-l_1\le
(h_2-h_1)/2\le 3$. It follows that the
Hausdorff distance between $\gamma_1$ and $\gamma_2$ is not more than
7. See Figure~\ref{fig:geodesic}.

\begin{figure}
\centering
\includegraphics{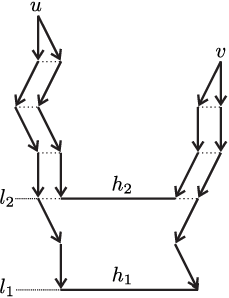}
\caption{} \label{fig:geodesic}
\end{figure}

If $\gamma_1$ is any (not necessarily $V$-shaped) geodesic, then it is
on distance at most 2 from a $V$-shaped geodesic connecting the same
points, hence it is at distance at most 9 from any proper $V$-shaped
path connecting the same points.
\end{proof}

We are ready now to prove that the graph $\Gamma_1$ (and hence also
$\Gamma$) is Gromov-hyperbolic. Let $a, b, c\in\Gamma_1$ be arbitrary
vertices. We have to prove that there exists a constant $\delta>0$ (not
depending on $a, b, c$) such that every geodesic triangle with the vertices
$a$, $b$, and $c$ is $\delta$-thin. By the
previous lemma, it is enough to show that every triangle of proper
$V$-shaped paths with the vertices $a$, $b$, and $c$ is $\delta$-thin
for some universal $\delta$.

Let $l_{xy}$ for $xy\in\{ab, bc, ac\}$ be the depth of a
proper $V$-shaped path $\gamma_{xy}$ connecting $x$ to $y$.
Without loss of generality we may assume that $l_{ab}\ge
l_{bc}\ge l_{ac}$ (see Figure~\ref{fig:triangle}).

\begin{figure}
\centering
\includegraphics{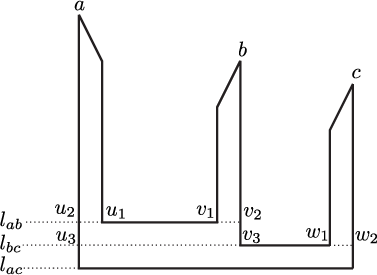}
\caption{} \label{fig:triangle}
\end{figure}

Denote by $u_1$ and $v_1$ the beginning and the end of the horizontal
part of the path $\gamma_{ab}$. Let $u_2$ and $v_2$ be the vertices of
level $l_{ab}$ of the descending parts of $\gamma_{ac}$ and $\gamma_{bc}$
respectively.
Then $|u_1-u_2|_1\le 1$ and $|v_1-v_2|_1\le 1$. It follows that
$|u_2-v_2|_1\le 8$.

Denote now by $v_3$ and $w_1$ the beginning and the end of the
horizontal part of $\gamma_{bc}$. Let $u_3$ and $w_2$ be the vertices
of the descending and ascending parts of $\gamma_{ac}$ of level
$l_{bc}$ (see Figure~\ref{fig:triangle}).

By Lemma~\ref{lem:contraction}, $|u_3-v_3|_1\le 5$. We also have
$|v_3-w_1|_1\le 6$ and $|w_1-w_2|_1\le
1$. It follows that $|u_3-w_2|_1\le 12$. Then by Lemma~\ref{lem:contraction},
$l_{bc}-l_{ac}\le 2$.

It follows that the segment $[a, u_2]$ of $\gamma_{ac}$ is on distance at most
1 from the segment $[a, u_1]$, the segment $[u_2, u_3]$ is on distance at
most 8 from $[b, v_3]$, the segment $[u_3, w_2]$ is on distance at most 8
from $[b, c]$, the segment $[w_2, c]$ is on distance at most 1 from $[c,
w_1]$ (see Figure~\ref{fig:triangle}). Consequently, the side
$\gamma_{ac}$ belongs to the 8-neighborhood of
$\gamma_{ab}\cup\gamma_{bc}$. It is also easy to see that
$\gamma_{ab}$ belongs to the 4-neighborhood of
$\gamma_{ac}\cup\gamma_{bc}$, and that $\gamma_{bc}$ belongs to the
8-neighborhood of $\gamma_{ac}\cup\gamma_{ab}$. We have proved that
the triangle with sides $\gamma_{ab}$, $\gamma_{bc}$, and
$\gamma_{ac}$ is 8-thin.

Consequently, by Lemma~\ref{l:geodesicV}, every geodesic triangle in
$\Gamma_1$ is 26-thin, and the graph $\Gamma$ is Gromov-hyperbolic. 

It
remains to prove that $\coc_1$ is a generalized Busemann quasi-cocycle
associated with a point of the boundary.
By Lemma~\ref{lem:quasigeodesic}, every infinite directed path $(v_1, v_2,
\ldots)$ of $\Gamma$ is a quasi-geodesic, hence it converges to a
point of $\partial\Gamma$. By condition (2), any two infinite directed
paths are eventually on distance not more than $\rho_0$, hence they
converge to the same point $\omega\in\partial\Gamma$.

Consequently, every infinite descending path of $\Gamma_1$ converges
to $\omega$. Let $\beta(v, u)$ be the Busemann quasi-cocycle on $\Gamma_1$
associated with $\omega$. Fix an infinite descending path $(w_0, w_1,
\ldots)$, where $\lambda(w_0)=0$.

Then $\beta(v, u)\doteq\lim_{n\to\infty}|v-w_n|_1-|u-w_n|_1$.
Let $(v=v_0, v_1, \ldots)$ and $(u=u_0, u_1, \ldots)$ be infinite
descending paths. Let $l_v$ and $l_u$ be the levels of $v$ and $u$
respectively. Then the vertices $v_n$ and $u_n$ will belong to the levels
$l_v-n$ and $l_u-n$ respectively.

By condition (2)' and Lemma~\ref{lem:contraction} there
exists $k$ such that for all $n\ge k$ the $|\cdot|_1$-diameter of
the set $\{w_n, v_{l_v-n}, u_{l_u-n}\}$ is not more than 2. Therefore,
\begin{multline*}
 |(|v-w_n|_1-|u-w_n|_1)-(l_v-l_u)|=\\
\left|(|v-w_n|_1-|u-w_n|_1)-(|v-v_{l_v-n}|_1-|u-u_{l_u-n}|_1\right|\le 4.
\end{multline*}

But $l_v\Delta_2\le\lambda(v)<(l_v-1)\Delta_2$ for every $v$, hence
$\Delta_2^{-1}\lambda(v)+1<l_v\le\Delta_2^{-1}\lambda(v)$, so that
\begin{multline*}
\Delta_2^{-1}\coc_1(v, u)-1=
\Delta_2^{-1}(\lambda(v)-\lambda(u))-1<l_v-l_u<\\
\Delta_2^{-1}(\lambda(v)-\lambda(u))+1=\Delta_2^{-1}\coc_1(v, u)+1.
\end{multline*}

We have shown that $\beta(v, u)\doteq\Delta_2^{-1}\coc_1(v,
u)\doteq\Delta_2^{-1}\coc(v, u)$, which by Lemma~\ref{lem:quasiisom}
shows that $\coc$ is a Busemann quasi-cocycle associated with $\omega$.
\end{proof}

\subsection{Boundary of a directed hyperbolic graph}

Let $\Gamma$ be a directed graph with a
quasi-cocycle $\coc$ such that for every directed edge $(u, v)$ we have
$2\eta\le\coc(u, v)\le\Delta$, where $\eta$ is as in
Definition~\ref{def:almcoc}. We assume that every vertex has at
least one outgoing arrow.

Suppose that $\Gamma$ satisfies the equivalent conditions of
Theorem~\ref{th:contraction}. Let $\omega\in\partial\Gamma$ be
such that $\coc$ is a Busemann quasi-cocycle associated with $\omega$. We denote
$\partial\Gamma_\omega=\partial\Gamma\setminus\{\omega\}$.

We say that a sequence $(v_1, v_2, \ldots)$ of vertices of $\Gamma$ is
an \emph{ascending path} if $(v_{i+1}, v_i)$ are directed edges of
$\Gamma$ for every $i$. It is a \emph{descending path} if $(v_i,
v_{i+1})$ is a directed edge for every $i$.

Let $v$ and $u$ be vertices of $\Gamma$. By
Proposition~\ref{prop:completionhalfplane} the space
$\Gamma\cup\partial\Gamma_\omega$ is the completion of the set of
vertices of $\Gamma$ with respect to the log-scale
\[\ell_{\omega, x_0}(v, u)=\frac 12\left(\beta_\omega(v,
  x_0)+\beta_\omega(u, x_0)-|v-u|\right),\]
where $x_0$ is a fixed vertex of $\Gamma$.
Denote
\[\lambda(v)=\beta_\omega(v, x_0).\]

\begin{proposition}
\label{prop:ushaped}
There exist constants $\Lambda>1$, $\rho>0$, and $k>0$ such
that any two vertices $u, v$ of $\Gamma$ can be connected by an
$(\Lambda, k)$-quasi-geodesic of the form $\gamma_0\gamma_1\gamma_2$, where
$\gamma_0$ is a descending path, $\gamma_1$ is a path of length not
more than $\rho$, and $\gamma_2$ is an ascending path.
\end{proposition}

\begin{proof}
It follows directly from Lemmas~\ref{lem:less7} and~\ref{lem:quasiisom}.
\end{proof}

\begin{proposition}
\label{prop:ellmin}
For a pair of vertices $u, v\in\Gamma$ choose a geodesic path
$\gamma_{u, v}$ connecting $u$ to $v$, and
denote by $\ell(u, v)$ the minimal value of $\lambda$ on a
vertex of $\gamma_{u, v}$. Then
there exists a constant $c>0$ (not depending on $u$, $v$, and
the choice of the geodesics) such that $|\ell(u,
v)-\ell_{\omega, x_0}(u, v)|<c$.
\end{proposition}

\begin{proof}
Let $w_n\in\Gamma$ be a sequence converging to $\omega$. Let $\gamma$
be a geodesic path connecting $v$ to $u$, and let $x$ be an arbitrary
vertex of $\gamma$. Then
\[|v-w_n|\le |v-x|+|x-w_n|,\qquad |u-w_n|\le |u-x|+|x-w_n|,\]
hence
\[\ell_{w_n}(v, u)\le\frac 12(|v-x|+|u-x|+2|x-w_n|-|v-u|)=
|x-w_n|.\]
On the other hand, there is $\delta>0$ depending only on $\Gamma$,
such that if $\gamma_v$ and $\gamma_u$ are geodesics
connecting $w_n$ to $v$ and $u$ respectively, and the points
$x\in\gamma$, $x_v\in\gamma_v$, $x_u\in\gamma_u$ are such that
\[|x-v|=|x_v-v|,\quad |x-u|=|x_u-u|, \quad |x_u-w_n|=|x_v-w_n|,\]
then diameter of $\{x, x_v, x_u\}$ is less than $\delta$.

Note that then $|x_u-w_n|=|x_v-w_n|=\ell_{w_n}(v, u)$, hence
$|x-w_n|\le |x_u-w_n|+|x-x_u|<\ell_{w_n}(v, u)+\delta$.

It follows that the minimal value of $|x-w_n|$ along the geodesic
$\gamma$ belongs to the interval $[\ell_{w_n}(v, u), \ell_{w_n}(v,
u)+\delta]$.

But $\ell_{\omega, x_0}(v, u)\doteq\lim_{w_n\to\omega}\ell_{w_n}(v,
u)-|x_0-w_n|$,
and $\beta_\omega(x, x_0)\doteq\lim_{w_n\to\omega}|x-w_n|-|x_0-w_n|$,
hence the minimal value of $\beta_\omega(x, x_0)$ along
$\gamma$ differs from $\ell_{\omega, x_0}(v, u)$ by a uniformly bounded constant.
\end{proof}

\begin{corollary}
\label{cor:minell}
For every pair of constants $k>0$, $\Lambda>1$ there exists $c>0$ such
that if $\ell(v, u)$ is the minimal value of $\lambda$ at a
vertex of an $(\Lambda, k)$-quasi-geodesic connecting $v$ to $u$, then
$|\ell_\omega(v, u)-\ell(v, u)|<c$.
\end{corollary}

\begin{proof}
It follows from Theorem~\ref{th:rigidquasigeodesics}.
\end{proof}

\begin{proposition}
\label{prop:ascendingconverge}
Every ascending path in $\Gamma$ converges to a point of
$\partial\Gamma_\omega$.
Every point of $\partial\Gamma_\omega$ is a limit of an ascending
path.
\end{proposition}

\begin{proof}
By Lemma~\ref{lem:quasigeodesic} every ascending path is a
quasi-geodesic, hence it converges to a point of
$\partial\Gamma$. The values of the Busemann quasi-cocycle
$\beta_\omega(v_n, v_1)$ increase, hence the limit is different from $\omega$.

Let $\xi\in\partial\Gamma_\omega$. Let $u_n\in\Gamma$ be a sequence
 such that $\xi=\lim_{n\to\infty}u_n$. For every $n$ find an infinite
descending path $\gamma_n$ starting at $u_n$. Then, by
Lemma~\ref{lem:contraction}, Proposition~\ref{prop:ushaped}, and
Proposition~\ref{cor:minell}, there exist constants $k_1, k_2$ such
that for any fixed $l$ and all sufficiently big $n$ and $m$ we have
$|s-t|\le k_1$ for all
$s\in\gamma_n$ and $t\in\gamma_m$ such that $\lambda(s), \lambda(t)\in
[l, l+k_2]$. By local finiteness of $\Gamma$
we can find a sequence $n_i$ such that the paths $\gamma_{n_i}$
converge to a bi-infinite directed path connecting $\omega$ with $\xi$.
\end{proof}

\section{Local product structures}

The notions presented here generalize the classical notions of a local
product structure on a manifold with an Anosov diffeomorphism,
and more generally the notion of a local product structure on a Smale
space (see~\cite{ruelle:therm}).

 The structure of a direct product decomposition of a
topological space can be formalized in the following way.

\begin{defi}
A \emph{rectangle} is a topological space $R$ together with a
continuous map
$[\cdot, \cdot]:R\times R\arr R$ such that
\begin{enumerate}
\item $[x, x]=x$ for all $x\in R$;
\item $[x, [y, z]]=[x, z]$ for all $x, y, z\in R$;
\item $[[x, y], z]=[x, z]$ for all $x, y, z\in R$.
\end{enumerate}\index{rectangle}
\end{defi}

For a direct product $R=A\times B$ the map
\[[(x_1, y_1), (x_2, y_2)]=(x_1, y_2)\]
obviously satisfies the conditions of the definition.

In the other direction, every rectangle comes
with a natural direct product decomposition. Define for $x\in R$
\[\proj_1(R, x)=\{y\in R\;:\;[x, y]=x\},\qquad\proj_2(R, x)=\{y\in
R\;:\;[x, y]=y\}.\]
The sets $\proj_i(R, x)$ are called \emph{plaques} \index{plaques}
\index{P@$\proj_i(R, x)$} of the rectangle.

Note that $[x, y]=x$ implies $[y, x]=[y, [x, y]]=y$. Similarly,
$[x, y]=y$ implies $[y, x]=[[x, y], x]=x$. On the other hand, $[y,
x]=y$ implies $[x, y]=[x, [y, x]]=x$, and $[y, x]=x$
implies $[x, y]=[[y, x], y]=y$. Consequently,
\[\proj_1(R, x)=\{y\in R\;:\;[y, x]=y\},\qquad\proj_2(R, x)=\{y\in
R\;:\;[y, x]=x\}.\]

\begin{proposition}
For every $x\in R$ the map $[\cdot, \cdot]:\proj_1(R,
x)\times\proj_2(R, x)\arr R$ is a homeomorphism.

In terms of the obtained direct product decomposition $R=\proj_1(R,
x)\times\proj_2(R, x)$ the map $[\cdot, \cdot]:R\times R\arr R$ is defined by the
rule
\[[(y_1, z_1), (y_2, z_2)]=(y_1, z_2).\]
\end{proposition}

See Figure~\ref{fig:rectangle} for the structure of $R$ and the map
$[\cdot, \cdot]$.

\begin{proof}
The map $[\cdot, \cdot]:\proj_1(R, x)\times\proj_2(R, x)\arr R$ is obviously
continuous.

For every $z\in R$ we have $[x, [z, x]]=x$, hence $[z, x]\in
\proj_1(R, x)$. We also have $[x, [x, z]]=[x, z]$, hence $[x, z]\in
\proj_2(R, x)$. It follows now from \[[[z, x], [x, z]]=[[z, x], z]=[z,
z]=z\] that the map
\[z\mapsto ([z, x], [x, z])\] is continuous and is inverse to
$[\cdot, \cdot]:\proj_1(R, x)\times\proj_2(R, x)\arr R$.

In terms of the decomposition $R=\proj_1(R, x)\times\proj_2(R, x)$ we have
\[[(y_1, z_1), (y_2, z_2)]=[[y_1, z_1], [y_2, z_2]]=[y_1,
z_2]=(y_1, z_2),\] for all $y_1, y_2\in\proj_1(R, x)$ and $z_1,
z_2\in\proj_2(R, x)$.
\end{proof}

\begin{figure}
\centering
\includegraphics{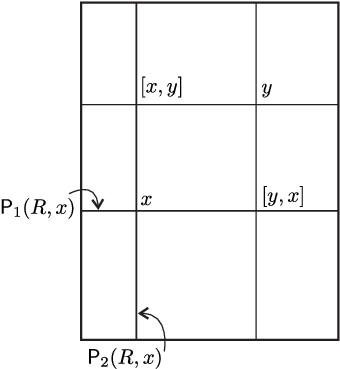}
\caption{Rectangle $R$} \label{fig:rectangle}
\end{figure}

The described direct product decomposition $R=\proj_1(R,
x)\times\proj_2(R, x)$
essentially does not depend on the choice of the point $x$.
Namely, we have the following natural homeomorphisms between
$\proj_i(R, x)$ and $\proj_i(R, y)$ for $x, y\in R$.

\begin{proposition}
For every $x, y\in R$ the maps
\[H_{1, y}:z\mapsto [z, y]:\proj_1(R, x)\arr\proj_1(R, y),\qquad
H_{2, y}:z\mapsto [y, z]:\proj_2(R, x)\arr\proj_2(R, y)\] are homeomorphisms.
\end{proposition}

\begin{proof}
We have $[y, [z, y]]=y$, hence $[z, y]\in\proj_1(R, y)$ for every $z\in
R$. Similarly, $[y, [y, z]]=[y, z]$, hence $[y, z]\in\proj_2(R, y)$
for every $z\in R$. We have $H_{1, x}(H_{1, y}(z))=[[z, y], x]=[z,
x]=z$ for all $z\in\proj_1(R, x)$ and $H_{2, x}(H_{2, y}(z))=[x, [y,
z]]=[x, z]=z$ for all $z\in\proj_2(R, x)$.
\end{proof}

Therefore, we can canonically identify the spaces $\proj_i(R, x)$ with
one space $\proj_i(R)$ and get a direct product decomposition
$R=\proj_1(R)\times\proj_2(R)$, which we
will call the \emph{canonical direct product decomposition} of the
rectangle $(R, [\cdot, \cdot])$.

\begin{defi}
\label{def:ldps} Let $\X$ be a topological space. A \emph{local
product structure} on $\X$ is given by a covering of $\X$ by open
subsets $R_i$, $i\in I$, together with a structure of a rectangle
$(R_i, [\cdot, \cdot]_i)$ on each set $R_i$, such that for every
pair $i, j\in I$ and every $x\in\X$ there exists a
neighborhood $U$ of $x$ such that $[y, z]_i=[y,
z]_j$ for all $y, z\in U\cap R_i\cap R_j$.\index{local product structure}

Two coverings of $\X$ by open rectangles define the same local
product structure on $\X$ if their union satisfies the above
compatibility condition.
\end{defi}

\begin{figure}
\includegraphics{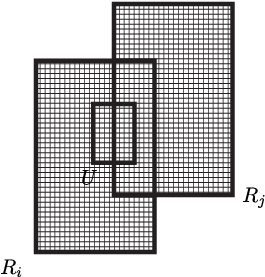}
\caption{Local product structure}\label{fig:locprod}
\end{figure}

Note that the condition of the definition is void for pairs $i,
j\in I$ such that the closures $\overline{R_i}$ and $\overline{R_j}$
are disjoint and for points $x\in\X$
that do not belong to the intersection
$\overline{R_i}\cap\overline{R_j}$.

A covering by rectangles $(R_i, [\cdot, \cdot]_i)$ satisfying the
conditions of the definition is called an \emph{atlas} of the local
product structure.

If the space $\X$ is compact, then any local product structure can be
defined by one function $[x, y]$ defined for all pairs $(x, y)\in
R\times R$ belonging to a neighborhood of the diagonal. It follows
that our notion of a local product structure is 
equivalent in the compact case to the notion of a local product
structure
defined in~\cite{ruelle:therm}.

\begin{defi}
Let $\X$ be a space with a local product structure on it. An open subset
$R\subset\X$ together with a direct product structure $[\cdot, \cdot]$
on $R$ is called a \emph{rectangle} of $\X$ if the union of an atlas of the local
product structure of $\X$ with $\{(R, [\cdot, \cdot])\}$ is also an
atlas of the local product structure, i.e., if it satisfies the
compatibility condition of Definition~\ref{def:ldps}.\index{rectangle}
\end{defi}

\chapter{Preliminaries on groupoids and pseudogroups}
\label{s:groupoidprelims}

\section{Pseudogroups and groupoids}
A \emph{groupoid}\index{groupoid} is a small category of isomorphisms. More
explicitly, it is a set $\G$ with a product $(g_1,
g_2)\mapsto g_1g_2$ defined on a subset $\G^{(2)}\subset\G\times\G$ and an
operation $g\mapsto g^{-1}:\G\arr\G$ such that the following
conditions hold.
\begin{enumerate}
\item if $g_1g_2$ and $g_2g_3$ are defined, then
  $(g_1g_2)g_3=g_1(g_2g_3)$ and the corresponding products are
  defined;
\item the products $gg^{-1}$ and $g^{-1}g$ are defined for all $g\in\G$;
\item if $gh$ is defined, then $ghh^{-1}=g$, $g^{-1}gh=h$, and the
  corresponding products are defined.
\end{enumerate}

We denote
\[\be(g)=g^{-1}g,\qquad\en(g)=gg^{-1}\]
and call $\be(g)$ and $\en(g)$ the \emph{origin} and the \emph{target} of $g$
respectively. We imagine $g$ as an arrow from $\be(g)$ to $\en(g)$.
The product $gh$ is defined if and only if $\en(h)=\be(g)$. Then
$\be(gh)=\be(h)$ and $\en(gh)=\en(g)$. We also have
$\be(g^{-1})=\en(g)$.\index{origin@$\be(g)$}\index{target@$\en(g)$}

Elements $g\in\G$ such that $g=\be(g)$ (which is equivalent to
$g=\en(g)$) are called \emph{units}\index{unit of a groupoid} of $\G$. The set of units is
denoted $\G^{(0)}$. In terms of category theory, we may identify the
units with the objects of the category. Then every element $g\in\G$ is an
isomorphism from $\be(g)$ to $\en(g)$.

For subsets $A, B\subset\G^{(0)}$, we denote $\G_A=\be^{-1}(A)$,
$\G^B=\en^{-1}(A)$, and $\G_A^B=\G_A\cap\G^B$. We also write $\G_x$
and $\G^x$ instead of $\G_{\{x\}}$ and $\G^{\{x\}}$.

For any $A\subset\G^{(0)}$ we denote by $\G|_A$ the sub-groupoid of elements
$g\in\G$ such that $\be(g), \en(g)\in A$, i.e.,
$\G|_A=\G_A^A$. The groupoid $\G|_A$ is called the \emph{restriction}
\index{restriction of a groupoid} of $\G$ to $A$.

Two units $x, y\in\G^{(0)}$ belong to the same \emph{orbit} \index{orbit} if there
exists $g\in\G$ such that $\be(g)=x$ and $\en(g)=y$ (i.e., if the
corresponding objects of the category are isomorphic). The relation of
belonging to the same orbit is obviously an equivalence.
A subset $A\subset\G^{(0)}$ is a \emph{$\G$-transversal} if it
intersects every orbit. \index{transversal}

If $x$ is a unit of a groupoid $\G$, then its \emph{isotropy
  group}\index{isotropy group} is the set $\G|_{\{x\}}$ of elements $g\in\G$ such
that $\be(g)=\en(g)=x$. The groupoid is \emph{principal}
\index{principal groupoid} \index{groupoid!principal} if all its
isotropy groups are trivial.

A \emph{topological groupoid}
\index{groupoid!topological}\index{topological groupoid}
is a groupoid $\G$ with a structure of a
topological space, such that multiplication and taking inverse are
continuous operations. We always assume that every element of $\G$ has
a compact Hausdorff neighborhood, and that the subspace $\G^{(0)}$ of units is
locally compact and metrizable.

\begin{defi}
Let $\X$ be a topological space. A \emph{pseudogroup}
\index{pseudogroup} $\wt\Gh$ acting on $\X$
is a set of homeomorphisms $F:U\arr V$ between open subsets of $\X$
closed under the following operations:
\begin{enumerate}
\item \textbf{composition}: if $F_1:U_1\arr V_1$ and $F_2:U_2\arr V_2$
  belong to $\wt\Gh$, then $F_1\circ F_2:F_2^{-1}(U_1\cap V_2)\arr
  F_1(U_1\cap V_2)$ is also an element of $\wt\Gh$;
\item \textbf{restrictions}: for every element $F:U\arr V$ of $\wt\Gh$
  and every open subset $W\subset\X$ the homeomorphism $F|_{U\cap W}$
  is an element of $\wt\Gh$;
\item \textbf{unions}: if $F:U\arr V$ is a homeomorphism between open subsets
  such that $U$ can be covered by a collection of open subsets $U_i$
  such that $F|_{U_i}\in\wt\Gh$, then $F\in\wt\Gh$.
\end{enumerate}
We always assume that the identical homeomorphism $\X\arr\X$ belongs to $\wt\Gh$.
\end{defi}

Let $\G$ be a topological groupoid. An open set $U\subset\G$ is called a
\emph{bisections}\index{bisection} if $\be:U\arr\be(U)$ and
$\en:U\arr\en(U)$ are
homeomorphisms. If $U$ is an open bisection, then the map
$\be(g)\mapsto\en(g)$ for $g\in U$ is a homeomorphism between open
subsets of $\G^{(0)}$. If every element of $\G$ has a neighborhood that is a bisection, then the
groupoid $\G$ is called \emph{\'etale}.
\index{etale groupoid@\'etale groupoid}
\index{groupoid!etale@\'etale}

The set of homeomorphisms defined by bisections is a
pseudogroup acting on $\G^{(0)}$, which will be called the
\emph{pseudogroup associated with the groupoid}
\index{pseudogroup!associated with a groupoid} and denoted $\pG$.

Note that we will consider bisections both as maps (elements of the
pseudogroup) and as subsets of the groupoid of germs. So, for example, if
$F_1, F_2\in\pG$, then notation $g\in F_1$ means that $g$ is a germ of
$F_1$, and $F_1\subset F_2$ means that $F_1$ is a restriction of $F_2$.

For any pseudogroup $\wt\Gh$ acting on a space
$\X$ a \emph{germ} \index{germ} of $\wt\Gh$ is an
equivalence class of a
pair $(F, x)$, where $F:U\arr V$ is an element of $\wt\Gh$ and $x\in
U$. Two pairs $(F_1, x_1)$ and $(F_2, x_2)$ are equivalent (define the
same germ) if $x_1=x_2$ and there exists a neighborhood $U$ of $x_1$
such that $F_1|_U=F_2|_U$. We define a topology on the set of all
germs of elements of $\wt\Gh$ choosing the basis of open sets consisting of
all sets of the form $\{(F, x);:\;x\in U\}$ for
$F:U\arr V$ an element of $\wt\Gh$. It is easy to see that the set of all
germs of $\wt\Gh$ is an \'etale groupoid with respect to the multiplication
\[(F_1, x_1)(F_2, x_2)=(F_1F_2, x_2),\] where the product is defined
if and only if $F_2(x_2)=x_1$. The inverse is given by the rule
$(F, x)^{-1}=(F^{-1}, F(x))$.

If $\G$ is the groupoid of germs of a pseudogroup $\wt\Gh$ acting on
$\X$, then $\wt\Gh$ coincides with the pseudogroup $\pG$ associated with
$\G$, and the groupoid
of germs of $\pG$ is naturally isomorphic to $\G$. Here we
identify $\X$ with the space of units $\G^{(0)}$.
Note that in general the groupoid of germs of a pseudogroup is not
Hausdorff.

We say that $\G$ is a \emph{groupoid of germs},
\index{groupoid!of germs} if there exists a pseudogroup
such that $\G$ is its groupoid of germs. In other words, a groupoid
$\G$ is a groupoid of germs if it is \'etale and coincides with the
groupoid of germs of the associated pseudogroup $\pG$.

\begin{defi}
Let $\G$ be a groupoid of germs. We say that $U\subset\G$ is
\emph{extendable} \index{extendable subset of a groupoid}
if there exists $V\in\pG$ such that $\overline U\subset V$.
\end{defi}

Since we assume that the space of units $\G^{(0)}$ is locally compact
and metrizable, every element of $\G$ has a compact extendable
neighborhood $U$. Moreover, we may find such $U$ that there exists an
element $\wh U\in\pG$ such that $\wh U\supset\overline U$ and $\wh U$ is
extendable.

\begin{defi} Suppose that $\G^{(0)}$ is a metric space.
For $U\in\pG$ and $g\in\G$ we say that $g$ is
\emph{$\epsilon$-contained}
\index{epsilon-contained@$\epsilon$-contained}
in $U$ if the $\epsilon$-neighborhood of
$\be(g)$ is contained in $\be(U)$.
\end{defi}

\begin{lemma}
\label{lem:Lebesguenumber}
Let $C\subset\G$ be a compact set, and let $\mathcal{U}\subset\pG$ be
an open covering of $C$. Then there exists $\epsilon>0$ such that for
every $g\in C$ there exists $U\in\mathcal{U}$ such that $g$ is
$\epsilon$-contained in $U$.
\end{lemma}

In conditions of the lemma, we say that $\epsilon$ is a \emph{Lebesgue's
  number} of the covering $\mathcal{U}$ of $C$.

\begin{proof}
The proof basically repeats the proof of the classical Lebesgue's
number lemma. We may assume that $\mathcal{U}$ is finite, and that no
element of $\mathcal{U}$ covers the whole set $C$. For every $g\in C$
and $U\in\mathcal{U}$ denote by $\delta_{g, U}$
the supremum of numbers $\epsilon\ge 0$ such that $g$ is $\epsilon$-contained
in $U$ (and zero if $g\notin U$). Denote
$f(g)=\frac{1}{|\mathcal{U}|}\sum_{U\in\mathcal{U}}\delta_{g, U}$. The
function $f:C\arr\R$ is continuous, since $\mathcal{G}$ is \'etale. It
is strictly positive, since $\mathcal{U}$ covers $C$. Let $\epsilon$
be a positive lower bound for the values of $f$ on $C$. Then for every
$g\in C$ the average $f(g)$ of the numbers $\delta_{g, U}$ is greater than
$\epsilon$, hence one of the numbers $\delta_{g, U}$ is greater than
$\epsilon$. Consequently, $\epsilon$ satisfies the conditions of the lemma.
\end{proof}

\section{Actions of groupoids and equivalence}

\subsection{Actions of groupoids}
The following definition is given in~\cite{muhlyrenault:equiv}
and~\cite[III.$\mathcal{G}$ Definition  3.11]{bridhaefl}.

\begin{defi}
\label{def:graction}
A \emph{(right) action} \index{action!of a groupoid} of a groupoid
$\G$ on a space $\mathfrak{B}$ over an
open map $P:\mathfrak{B}\arr\G^{(0)}$ is a continuous map $(x,
g)\mapsto x\cdot g$ from the set
\[\mathfrak{B}\times_P\G=\{(x, g)\;:\;P(x)=\en(g)\}\]
to $\mathfrak{B}$ such that $P(x\cdot g)=\be(g)$, and $(x\cdot
g_1)\cdot g_2=x\cdot g_1g_2$ for all $x\in\mathfrak{B}$ and $g_1,
g_2\in\G$ such that $P(x)=\en(g_1)$ and $\be(g_1)=\en(g_2)$.
\end{defi}

Similarly, the \emph{left action} is a map $(g, x)\mapsto g\cdot x$ from
$\G\times_P\mathfrak{B}=\{(g, x)\;:\;P(x)=\be(g)\}$ to
$\mathfrak{B}$ such that $P(g\cdot x)=\en(g)$ and $g_1\cdot
(g_2\cdot x)=g_1g_2\cdot x$.

See Figure~\ref{fig:graction} for a schematic illustration of Definition~\ref{def:graction}.

\begin{figure}
\centering
\includegraphics{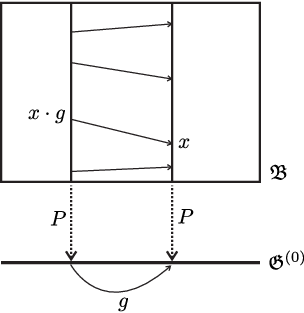}
\caption{Action of a groupoid}\label{fig:graction}
\end{figure}

\begin{defi}
\label{def:actiongroupoid}
Suppose that we have a right action of $\G$ on $\mathfrak{B}$ over
a map $P:\mathfrak{B}\arr\G^{(0)}$. The \emph{groupoid of the
action} \index{groupoid!of the action} $\mathfrak{B}\rtimes\G$ is the space
$\mathfrak{B}\times_P\G$ together with multiplication
\[(x_1, g_1)\cdot (x_2, g_2)=(x_1, g_1g_2),\]
where the product is defined if and only if $x_2=x_1\cdot g_1$.
\end{defi}

Note that the space of units
of $\mathfrak{B}\rtimes\G$ is the set of elements $(x,
g)\in\mathfrak{B}\times_P\G$
such that $g$ is a unit. The unit $(x, g)$ is
uniquely determined by $x$, thus the unit space of
$\mathfrak{B}\rtimes\G$ is naturally identified with $\mathfrak{B}$.

A right action is \emph{free} \index{action!free} \index{free action}
if $x\cdot g=x$ implies that $g$ is a unit. The action is \emph{proper} \index{action!proper}
\index{proper action} if the map
\[(\be,
\en):(x, g)\mapsto (x,
x\cdot g):\mathfrak{B}\rtimes\G\arr\mathfrak{B}^2\] is proper. 
In other words, the action is proper, if for every
compact set $C\subset\mathfrak{B}$ the set of elements $g\in\G$
such that $C\cdot g\cap C\ne\emptyset$ is compact.

Suppose that we have a free proper action of $\G$ on
$\mathfrak{B}$. Consider the space $\mathfrak{B}*_P\mathfrak{B}=\{(x,
y)\in\mathfrak{B}^2\;:\;P(x)=P(y)\}$. We have a natural groupoid
structure on $\mathfrak{B}*_P\mathfrak{B}$ given by
\[(x, y)\cdot (y, z)=(x, z),\]
where $(x, y)$ is a unit if and only if $x=y$ (so that the origin
and the target maps are given by the equalities $\be(x, y)=y$ and $\en(x, y)=x$).

The groupoid $\G$ acts on $\mathfrak{B}*_P\mathfrak{B}$ by the
diagonal action. Consider the space $\mathfrak{B}*_P\mathfrak{B}/\G$
of orbits of the action (which is Hausdorff, for instance
by~\cite[Proposition~4.2.1]{nek:book}). Note that the
groupoid structure on $\mathfrak{B}*\mathfrak{B}$ induces a groupoid
structure on the quotient, since from $(x_1, y_1)=(x, y)\cdot g_1$
and $(y_1, z_1)=(y, z)\cdot g_2$ follows $y\cdot g_1=y\cdot g_2$,
which implies $g_1=g_2$ by freeness of the action, so that $(x_1,
z_1)=(x, z)\cdot g_1$.

\subsection{Equivalence of groupoids and pseudogroups}
\label{s:equivalence}

The following definition was introduced in~\cite{muhlyrenault:equiv}.
\begin{defi}
Let $\G$ and $\Gh$ be topological groupoids. A $(\G, \Gh)$-equivalence is given by
a locally compact Hausdorff space $\mathfrak{B}$ together with a free
proper left $\G$-action and a free proper right $\Gh$-action over maps
$P_1:\mathfrak{B}\arr\G^{(0)}$ and $P_2:\mathfrak{B}\arr\Gh^{(0)}$, such that
\begin{enumerate}
\item the actions of $\G$ and $\Gh$ commute, i.e., if $g\cdot x$ and
$x\cdot h$ are defined, then $(g\cdot x)\cdot h$ and $g\cdot (x\cdot h)$ are
defined and are equal;
\item the maps $P_1$ and $P_2$ induce bijections
$\mathfrak{B}/\Gh\arr\G^{(0)}$ and
$\G\backslash\mathfrak{B}\arr\Gh^{(0)}$.
\end{enumerate} \index{equivalence!of groupoids}
\end{defi}

It is convenient to imagine elements $x\in\mathfrak{B}$ of an equivalence as arrows
from $P_2(x)$ to $P_1(x)$ and to interpret the actions of $\G$ and
$\Gh$ as compositions of these arrows with the arrows of the
groupoids.

Note that the first two conditions imply that $P_1$ and $P_2$
induce well defined maps $\mathfrak{B}/\Gh\arr\G^{(0)}$ and
$\G\backslash\mathfrak{B}\arr\Gh^{(0)}$ respectively. Namely, by
the second condition, if $g\cdot x$ and $x\cdot h$ are defined,
i.e., if $P_1(x)=\be(g)$ and $P_2(x)=\en(h)$, then $(g\cdot
x)\cdot h$ and $g\cdot (x\cdot h)$ are defined, i.e.,
\[P_2(g\cdot x)=\en(h)=P_2(x),\qquad P_1(x\cdot
h)=\be(g)=P_1(x).\]

Condition (2) of the definition implies that for every pair $(x,
y)\in\mathfrak{B}*_{P_1}\mathfrak{B}$ there exists $h\in\Gh$ such
that $x\cdot h=y$. The element $h$ is unique, by freeness of the
action. The map $(x, y)\mapsto h$ is invariant under the action of
$\G$ (since $x\cdot h=y$ is equivalent to $g\cdot x\cdot h=g\cdot
y$), and it induces a map from
$\G\backslash\left(\mathfrak{B}*_{P_1}\mathfrak{B}\right)$ to $\Gh$. It is easy
to check that this map is an isomorphism of groupoids.

If $\mathfrak{B}_1$ is an $(\G, \Gh)$-equivalence, and
$\mathfrak{B}_2$ is a $(\Gh, \mathfrak{K})$-equivalence, then
$\mathfrak{B}_1\otimes_{\Gh}\mathfrak{B}_2$ is a $(\G,
\mathfrak{K})$-equivalence, where
$\mathfrak{B}_1\otimes_{\Gh}\mathfrak{B}_2$ is the quotient of the
space \[\{(x, y)\;:\;x\in\mathfrak{B}_1, y\in\mathfrak{B}_2,
P_2(x)=P_1(y)\}\subset\mathfrak{B}_1\times\mathfrak{B}_2\] by the
equivalence relation $(x, y)=(x\cdot g, g^{-1}\cdot y)$.

\subsection{Equivalence of \'etale groupoids}

\begin{lemma}
\label{lem:etaleaction}
If $\mathfrak{B}$ is a right $\G$-space and $\G$ is \'etale, then
the action groupoid $\mathfrak{B}\rtimes\G$ is \'etale.
\end{lemma}

\begin{proof}
Suppose that the action is defined over a map
$P:\mathfrak{B}\arr\G^{(0)}$. The origin and the target maps of the action
groupoid are given by
\[\be(x, g)=(x\cdot g, \be(g)),\qquad\en(x, g)=(x, \en(g)).\]
If $U$ is an open neighborhood of $g$ such that $\be:U\arr\G^{(0)}$ and
$\en:U\arr\G^{(0)}$ are homeomorphic embeddings, then
$U'=\mathfrak{B}\rtimes\G\cap\mathfrak{B}\times U$ is an open set
such that restrictions of $\be$ and $\en$ to $U'$ are
homeomorphic embeddings. Namely, the local inverses of $\be$ and $\en$ are
\[(x, P(x))\mapsto (x\cdot(\be^{-1}(P(x)))^{-1},
\be^{-1}(P(x))),\qquad (x, P(x))\mapsto (x, \en^{-1}(P(x))),\]
where $\be^{-1}$ and $\en^{-1}$ are the inverses of
$\be:U\arr\be(U)$ and $\en:U\arr\en(U)$ respectively.
\end{proof}

Suppose that $\mathfrak{B}$ is a $(\G, \Gh)$-equivalence, where
$\G$ and $\Gh$ are \'etale groupoids. Let
$P_1:\mathfrak{B}\arr\G^{(0)}$ and $P_2:\mathfrak{B}\arr\Gh^{(0)}$
be the maps over which the actions are defined.

\begin{lemma}
\label{lem:P1P2} The maps $P_1$ and $P_2$ are
\'etale,\index{map!\'etale}
i.e., are local homeomorphisms.
\end{lemma}

\begin{proof}
Let $x\in\mathfrak{B}$ be an arbitrary point and let $U$ be a
compact neighborhood of $x$. By properness of the groupoid
$\mathfrak{B}\rtimes\Gh$, the set $C$ of elements $(y,
h)\in\mathfrak{B}\rtimes\Gh$ such that $y\in U$ and $y\cdot h\in U$
is compact. Since $\mathfrak{B}\rtimes\Gh$ is \'etale, the set of
units of $\mathfrak{B}\rtimes\Gh$ is open, hence
$C'=C\setminus(\mathfrak{B}\rtimes\Gh)^{(0)}$ is compact.

The groupoid $\mathfrak{B}\rtimes\Gh$ is \'etale and principal, hence
for $(y, h)\in C'$ there exists a neighborhood $V$ of $(y, h)$
such that $\overline V$ is compact and either
$x\notin \be(\overline V)\cup\en(\overline V)$
(if $x\ne y$ and $x\ne y\cdot h$), or $x\in\be(V)$ and
$x\notin\en(\overline V)$ (if $x=y$), or
$x\notin\be(\overline{V})$ and $x\in\en(V)$ (if $x=y\cdot h$). It
follows that there is a finite set $A$ of elements
$V\in\wt{\mathfrak{B}\rtimes\Gh}$ covering
$C'$ and satisfying the above conditions. Then
\[U'=U\setminus\bigcup_{V\in A, x\notin\be(V)}\be(\overline V)\cup\bigcup_{V\in A,
x\notin\en(V)}\en(\overline V)\] is a neighborhood of $x$ such
that there is no non-unit element $(y, h)\in\mathfrak{B}\rtimes\Gh$
such that $y\in U'$ and $y\cdot h\in U'$.

Similarly, we can find a neighborhood $U''\subset U$ of $x$ such that there
is no non-unit element $(g, y)\in\G\ltimes\mathfrak{B}$ such that
$y\in U''$ and $g\cdot y\in U''$. Consider the intersection
$U'\cap U''$. It is a neighborhood of $x$ such that its
closure is a compact subset of $U$ and such that the maps
$P_1:U'\cap U''\arr\G^{(0)}$ and $P_2:U'\cap U''\arr\Gh^{(0)}$ are
injective, since $P_1(y_1)=P_1(y_2)$ implies existence of
$h\in\Gh$ such that $y_1\cdot h=y_2$. Since $P_i$ are continuous
and open, we conclude that they are homeomorphic embeddings on
$U'\cap U''$.
\end{proof}

Let $\mathfrak{B}$ be a $(\G, \Gh)$-equivalence. Consider the
disjoint union
\[\G\vee_{\mathfrak{B}}\Gh:=\G\sqcup\mathfrak{B}\sqcup\mathfrak{B}^{-1}\sqcup\Gh,\]
where $\mathfrak{B}^{-1}$ is a copy of $\mathfrak{B}$. We denote by
$g^{-1}$ the element of $\mathfrak{B}^{-1}$ corresponding to
$g\in\mathfrak{B}$. Denote $\en(x)=P_1(x)$, $\be(x)=P_2(x)$,
$\be(x^{-1})=P_1(x)$, $\en(x^{-1})=P_2(x)$ for $x\in\mathfrak{B}$.
Define multiplication on $\G\vee_{\mathfrak{B}}\Gh$ using multiplications inside
$\G$ and $\Gh$, actions of $\G$ and $\Gh$ on $\mathfrak{B}$,
``flipped'' actions:
\[x^{-1}\cdot g=(g^{-1}\cdot x)^{-1},\qquad h\cdot x^{-1}=(x\cdot
h^{-1})^{-1}\] on $\mathfrak{B}^{-1}$, and multiplication between
elements of $\mathfrak{B}$ and $\mathfrak{B}^{-1}$ given by the
rules
\[x\cdot y^{-1}=g\in\G,\quad\text{if and only if}\quad g\cdot y=x\]
\[x^{-1}\cdot y=h\in\Gh,\quad\text{if and only if}\quad x\cdot
h=y\] for $x, y\in\mathfrak{B}$. Note that the corresponding
elements $g$ and $h$ exist and are unique by the definition of
equivalences.

It is easy to show that the defined multiplication defines a
groupoid structure on $\G\vee_{\mathfrak{B}}\Gh$.

\begin{proposition}
If $\G$ and $\Gh$ are \'etale groupoids, then
$\G\vee_{\mathfrak{B}}\Gh$ is also an \'etale groupoid.
\end{proposition}

\begin{proof}
It follows from Lemma~\ref{lem:P1P2} that the operations in the
groupoid $\G\vee_{\mathfrak{B}}\Gh$ are continuous and that the
obtained groupoid is \'etale.
\end{proof}

Now we can introduce a more convenient definition of equivalence in the case of groupoids of germs.

\begin{proposition}
Let $\G$ and $\Gh$ be groupoids of germs of pseudogroups $\pG$ and
$\wt{\Gh}$ respectively. The groupoids $\G$ and $\Gh$ are equivalent
if and only if there exists a pseudogroup
$\pG\vee\wt{\Gh}$ acting on the disjoint union
$\G^{(0)}\sqcup\Gh^{(0)}$ such that the restriction of
$\pG\vee\wt{\Gh}$ to $\G^{(0)}$ (resp.\ $\Gh^{(0)}$) coincides
with $\pG$ (resp.\ $\wt{\Gh}$), and the sets $\G^{(0)}$ and
$\Gh^{(0)}$ are $\pG\vee\wt{\Gh}$-transversals.\index{equivalence!of groupoids}
\end{proposition}

\begin{proof}
If $\G$ and $\Gh$ are equivalent, then it is easy to see that the
pseudogroup associated with the groupoid $\G\vee_{\mathfrak{B}}\Gh$
satisfies the conditions of the proposition.

In the other direction, the set $\mathfrak{B}$ of germs of elements of
$\pG\vee\wt{\Gh}$, such that their domain is in $\pG$ and range in
$\wt{\Gh}$, together with the actions of $\G$
and $\Gh$ on $\mathfrak{B}$ defined by composition, is an equivalence. The only
condition we have to check is properness of the actions. Any
compact subset $C$ of $\mathfrak{B}$ can be covered by compact
extendable closures of elements of $\pG\vee\wt{\Gh}$. Then for every
element $g$ of the action groupoid with origin and target in $C$ there
will exist germs $h_1$ and $h_2$ of elements of the covering such that
$g=h_1^{-1}h_2$.
\end{proof}

\begin{defi}
We say that two pseudogroups are \emph{equivalent}
\index{equivalence!of pseudogroups}
if their groupoids of germs are equivalent.
\end{defi}

\subsection{Localization}
\label{ss:localization}
We describe here standard methods of constructing pseudogroups equivalent to a
given one: \emph{localization} and \emph{restriction}.

\begin{proposition}
\label{pr:localization}
Let $\G$ be a groupoid, and let
$f:\mathcal{Y}\arr\G^{(0)}$ be an \'etale map such that the range of $f$
is a $\G$-transversal. Consider the pseudogroup $f^*(\pG)$
generated by homeomorphisms $F:U\arr V$ between open subsets
of $\mathcal{Y}$ such that $f|_U:U\arr f(U)$, $f|_V:V\arr f(V)$
are homeomorphisms, and $f|_V\circ F\circ f|_U^{-1}$ is an element
of $\pG$. Then  $f^*(\pG)$ is a pseudogroup equivalent to $\pG$.
\end{proposition}

We call the pseudogroup $f^*(\pG)$ the \emph{lift} of $\pG$ by
$f$. Its groupoid of germs $f^*(\G)$ is the lift of $\G$.\index{lift of a groupoid}

In particular, if $f$ is the identical embedding of an open subset
$Y\subset\G^{(0)}$, then $f^*(\pG)$ is the restriction
$\pG|_Y$. Therefore, restriction of a groupoid $\G$ to an open
transversal $Y$ is equivalent to $\G$.

\begin{proof} It is easy to check that we can define the pseudogroup
  $f^*(\pG)\vee\pG$
to be equal to the pseudogroup generated by $\pG$ and restrictions of
$f$.
\end{proof}

\begin{examp}
Consider an atlas of an $n$-dimensional manifold $M$, and let $\X$ be the disjoint
union of the corresponding open subsets of $\R^n$. Then we have a
natural \emph{pseudogroup of changes of charts} $\G$ acting on $\X$. It is
equal to the lift of the trivial pseudogroup on $M$ by the natural
quotient map $\X\arr M$. It follows that the pseudogroup $\G$
is equivalent to the trivial pseudogroup (i.e., the pseudogroup
consisting of identity maps on open subsets) of the manifold $M$.
\end{examp}

\begin{examp}
Let $G$ be a group acting properly and freely on a space $\X$. Then
the groupoid of the action $\X\rtimes G$ is equivalent to the trivial
(i.e., consisting only of units) groupoid of the space of orbits
$\X/G$. In particular, a manifold (as a trivial groupoid) 
is equivalent to the groupoid of the
action of the fundamental group on the universal covering.
\end{examp}

Another example of application of Proposition~\ref{pr:localization}
is the notion of localization to an open covering. If $\{U_i\}_{i\in I}$ is a collection of
open subsets of $\G^{(0)}$ such that their union is a $\pG$-transversal,
then \emph{localization} $\pG|_{\{U_i\}}$ is lift by the map
from the disjoint union $\bigsqcup_{i\in I}U_i$ to $\G^{(0)}$ equal to
the identical embedding on each $U_i$.

We denote by $(U_i, i)$ the
copies of $U_i$ in the disjoint union. If $g\in\G$ is such that
$\be(g)\in U_i$ and $\en(g)\in U_j$, then we denote by $(g, i, j)$
the corresponding element of the localization such that $\be(g, i,
j)\in (U_i, i)$ and $\en(g, i, j)\in (U_j, j)$. \index{localization}

\subsection{Groupoid equivalence for group actions}

Let $G_i$ be a group acting by homeomorphisms on $\X_i$, for
$i=1,2$.

\begin{proposition}
\label{prop:groupactionseq}
The groupoids $G_1\ltimes\X_1$ and
$G_2\ltimes\X_2$ of the actions are equivalent if and only if there exists a
space $\mathfrak{B}$ and commuting free proper left and right
actions of $G_1$ and $G_2$ respectively, on $\mathfrak{B}$ such
that $\mathfrak{B}/G_2$ (resp.\ $G_1\backslash\mathfrak{B}$) is
$G_1$-equivariantly (reps.\ $G_2$-equivariantly) homeomorphic to
$\X_1$ (resp.\ $\X_2$).
\end{proposition}

\begin{proof}
Suppose that $\mathfrak{B}$ is a $(G_1\ltimes\X_1,
G_2\ltimes\X_2)$-equivalence, and let $P_i:\mathfrak{B}\arr\X_i$ be
the corresponding maps. For every $x\in\mathfrak{B}$ and every
$g\in G_1$, $h\in G_2$ we can define
\[g\cdot x=(g, P_1(x))\cdot x\]
and
\[x\cdot h=x\cdot (h, P_2(x)).\]
It is easy to check that we get actions of $G_i$ on $\mathfrak{B}$.
The actions commute. Freeness of the actions of groupoids is
equivalent to the usual freeness of the group actions. The
groupoid of the action of $G_i\ltimes\X_i$ coincides with the
groupoid of the action of $G_i$, hence properness of the actions
of the groupoids is equivalent to properness of the actions of the
groups.

By the remaining condition of the definition of equivalence of
groupoids, the map $P_1$ induces a homeomorphism of
$\mathfrak{B}/G_2$ with $\X_1$, which is $G_1$-equivariant, since
the actions of $G_1$ and $G_2$ commute. The other direction of the
proof is similar.
\end{proof}

\begin{proposition}
If $\mathfrak{B}$ together with left and right actions of $G_1$ and
$G_2$ over maps $P_i:\mathfrak{B}\arr\X_i$
is an equivalence between the groupoids $G_1\rtimes\X_1$ and
$G_2\rtimes\X_2$, then $P_i:\mathfrak{B}\arr\X_i$ are covering maps.
\end{proposition}

\begin{proof}
The maps $P_1$ and $P_2$ are \'etale by Lemma~\ref{lem:P1P2}. Let
$x_1\in\X_1$, and let $y_1\in P_1^{-1}(x)$. Then there exists a
relatively compact
neighborhood $U$ of $y_1$ such that $P_1:U\arr P_1(U)$ is a
homeomorphism, and $P_1(U)$ is open. Moreover, by properness and
freeness of the action of $G_2$ on $\mathfrak{B}$, we may choose such
$U$ that $g(U)$ are disjoint for all $g\in G_2$. It follows then that
$P_1^{-1}(P_1(U))=\bigsqcup_{g\in G_2}g(U)$ and that restriction of $P_1$
to each set $g(U)$ is a homeomorphism with $P_1(U)$. This means
that $P_1$ is a covering map.
\end{proof}

Suppose that the space $\X$ is connected and semi-locally simply
connected. Let $G$ be a group acting faithfully on $\X$ by homeomorphisms. Let
$P:\wt{\X}\arr\X$ be the universal covering. For every $g\in G$
consider the set $P^*(g)$ of all lifts of $g$ to a homeomorphism of
$\wt{\X}$. Note that for any two $h_1, h_2\in P^*(g)$ the
homeomorphism $h_1^{-1}h_2$ belongs to the fundamental group of $\X$
(with respect to its natural action on the universal covering). The
union of the sets $P^*(g)$ is a group $\wt G$ of homeomorphisms of
$\X$. We have a natural epimorphism $\phi:\wt G\arr G$. Kernel of
$\phi$ is the fundamental group $\pi_1(\X)$. We call $\wt G$ the
\emph{lift of $G$ to the universal covering} of $\X$. Note that $\wt
G$ consists of all homeomorphism $h:\wt{\X}\arr\wt{\X}$ such that there exists
$g\in G$ such that $P(h(x))=g(P(x))$ for all $x\in\wt{\X}$.

\begin{theorem}
\label{th:equivalentgroupactions}
Let $\X_1$ and $\X_2$ be connected and semi-locally simply connected
topological spaces, and let, for $i=1, 2$, $G_i$ be a
group acting faithfully on $\X_i$ by homeomorphisms.

Then the groupoids $G_1\ltimes\X_1$ and $G_2\ltimes\X_2$ are
equivalent if and only if the lifts of the actions $(G_1, \X_1)$ and
$(G_2, \X_2)$ to the universal coverings $\wt\X_i$ are topologically conjugate.
\end{theorem}

Here two group actions $(H_1, \X_1)$ and $(H_2, \X_2)$ are
\emph{topologically conjugate} if there exist a homeomorphism
$F:\X_1\arr\X_2$ and an isomorphism $\phi:H_1\arr H_2$ such that
$F(g(x))=\phi(g)(F(x))$ for all $x\in\X_1$ and $g\in H_1$.

\begin{proof}
The lift of a group action to the universal
covering is equivalent as a groupoid to the original action (see
Proposition~\ref{pr:localization}). This
proves the ``if'' part of the theorem.

Suppose that the actions $(G_i, \X_i)$ are equivalent. Let
$\mathfrak{B}$ be as in Proposition~\ref{prop:groupactionseq}. Choose
a point $x\in\mathfrak{B}$, and let $\X$ be the connected component of
$\mathfrak{B}$ containing $x$. Since $P_i:\mathfrak{B}\arr\X_i$ are
covering maps, $P_i:\X\arr\X_i$ are also covering maps. In particular,
they are onto.

The actions of $G_i$ on $\mathfrak{B}$ taken together form an action
of $G_1\times G_2$ (by the rule $(g_1, g_2)(x)=g_1\cdot x\cdot
g_2^{-1}$). Let $H=\{g\in G_1\times G_2\;:\;g(\X)=\X\}$
be the stabilizer of $\X$ in $G_1\times G_2$. 

If $x, y\in\X$ are such that $P_1(x)=P_1(y)$,
then there exists a unique element $g\in G_2$ such that $g(x)=y$. Note
that then $g\in H$. It follows that $G_2\cap H$ is the group of deck
transformations of the covering map $P_1:\X\arr\X_1$, and that the
covering is normal.

Let us show that restriction to $H$ of the projection $G_1\times G_2\arr G_1$
is surjective. Let $g\in G_1$ be an arbitrary element. Since the map
$\mathfrak{B}/G_2\arr\X_1$ induced by $P_1$ is a homeomorphism, and
$G_2$ acts by permutations on the set of connected components of
$\mathfrak{B}$, there exists $g_2\in G_2$ such that $g\cdot\X\cdot
g_2=\X$. Then $(g, g_2^{-1})\in H$.

We have proved that there exists a pair of covering maps $\X\arr\X_i$
and a group $H$ acting on $\X$ such that the lifts of $(G_i, \X_i)$
by the covering maps coincide with $(H, \X)$. It follows that the
lifts of $(G_i, \X_i)$ to the universal coverings $\wt{\X_i}$ both are
topologically conjugate to the lift of $(H, \X)$ to the universal
covering of $\X$.
\end{proof}

\begin{corollary}
\label{cor:simplyconnected} Let $(G_i, \X_i)$, for $i=1,2$, be group
actions on simply connected spaces. The groupoids
$G_1\ltimes\X_1$ and $G_2\ltimes\X_2$ are equivalent if and only if the
actions are topologically conjugate.

Groupoids of \emph{free} proper actions $(G_i, \X_i)$ on connected
semi-locally simply connected spaces are equivalent if and only if
the spaces $\X_i/G_i$ are homeomorphic.
\end{corollary}

\section{Compactly generated groupoids and their Cayley graphs}
\label{ss:cayley}
We say that $\X\subset\G^{(0)}$ is a \emph{topological
transversal} \index{transversal!topological} if there is an open transversal $\X_0\subset\X$.

The following definition is equivalent to a definition due to
A.~Haefliger,  see~\cite{haefliger:compactgen}.

\begin{defi}
\label{def:compactlygenerated}
Let $\G$ be an \'etale groupoid. A \emph{compact generating pair}
\index{compact generating pair} \index{generating pair}
of $\G$ is a pair of  sets $(S, \X)$, where $S\subset\G$ and
$\X\subset\G^{(0)}$ are compact, $\X$ is a topological transversal, and for every
$g\in\G|_{\X}$ there exists $n$ such that
$\bigcup_{k\ge 1}^n(S\cup S^{-1})^k$ is
a neighborhood of $g$ in $\G|_{\X}$.
The set $S$  called a \emph{generating set}.\index{generating set!of a groupoid}

A groupoid $\G$ is \emph{compactly generated} if it has a compact
generating pair. \index{groupoid!compactly generated} \index{compactly
  generated groupoid}
\end{defi}

In other words, $(S, \X)$ is a compact generating pair, if $S$ is a
generating set of $\G|_{\X}$ and the word length on $\G|_{\X}$ defined
with respect to $S$ is a locally bounded function.

\begin{proposition}
\label{pr:differnttransv} Let $\G$ be a compactly
generated groupoid of germs. Then for every compact topological
transversal $\X\subset\G^{(0)}$ there exists a compact set $S\subset\G|_{\X}$
such that $(S, \X)$ is a generating pair.
\end{proposition}

\begin{proof}
Let $(S', \X')$ be a compact generating pair of
$\G$. Let $\X_0'\subset\X'$ and $\X_0\subset\X$ be
open transversals.

For every $x\in\X$ there exists an element $U\in\pG$ such that
$\be(U)\ni x$, $\en(U)\subset\X_0'$, and $\overline U$ is compact.
We get an open covering of $\X$ by the sets $\be(U)$.
Choose a finite sub-covering $\{\be(U_1),
\be(U_2), \ldots, \be(U_r)\}$.

Similarly, there exists a finite set $\{U_{r+1}, U_{r+2}, \ldots,
U_{r+s}\}$ of elements of $\pG$ such that $\overline{U_i}$ are
compact, $\en(U_i)$ cover $\X'$, and $\be(U_i)\subset\X_0$.

Denote $W=\bigcup_{i=1}^{r+s}U_i$, and consider the set
$S=\left(W^{-1}\cdot S'\cdot W\right)\cap\G|_{\X}$. The set $S$ has
compact closure. For every
$g\in\G|_{\X}$ there exist $i, j\in\{1, \ldots, r\}$ such that
$U_i\cdot g\cdot U_j^{-1}\in\G|_{\X_0'}$. Since $(S', \X')$ is a
generating pair, there exists $m$ such that
$\bigcup_{k=1}^m(S'\cup(S')^{-1})^k$ is a neighborhood of $U_i\cdot g\cdot
U_j^{-1}$. Then
\[W^{-1}\bigcup_{k=1}^m(S'\cup (S')^{-1})^k W\supset
U_i^{-1}\cdot \bigcup_{k=1}^m(S'\cup (S')^{-1})^k\cdot U_j\]
is a neighborhood of $g$. Since for every $x\in\X'$ there exists
$U_{r+l}$ such that $\en(U_{r+l})\ni x$ and $\be(U_{r+l})\subset\X_0$,
we have
\[\left(W^{-1}\cdot (S'\cup (S')^{-1})^k\cdot W\right)\cap\G|_{\X}\subset
\left(\left(W^{-1}S'W\cap\G|_{\X}\right)\cup
\left(W^{-1}S'W\cap\G|_{\X}\right)^{-1}\right)^k\]
for all $k$. It follows that $\bigcup_{k=1}^m(S\cup S^{-1})^k$ is a
neighborhood of $g$ in $\G|_{\X}$.
\end{proof}

\begin{corollary}
If $\G$ is a compactly generated groupoid of germs, then every
equivalent \'etale groupoid is also compactly generated.
\end{corollary}

\begin{defi}
Let $(S, \X)$ be a compact generating pair of
$\G$. For $x\in\X$ the \emph{Cayley graph} $\G(x, S)$
is the oriented graph with the set of vertices
$\G_x^{\X}$ in which there is an arrow from $g$ to $h$ if there exists
$s\in S$ such that $h=sg$. \index{Cayley
  graph of a groupoid}
\end{defi}

In most cases we will ignore the orientation on the Cayley graph (in
particular, when we talk about valency of vertices, or combinatorial distance on it).

Since a compact subset of $\G$ can be covered by a finite number of local
homeomorphisms, the Cayley graph $\G(x, S)$ has uniformly bounded
degree for all $x\in\X$.

If $x\in\X$ has trivial isotropy group, then the map $g\mapsto\en(g)$
from $\G_x^{\X}$ to the $\G|_{\X}$-orbit of $x$ is a bijection. Then
the Cayley graph $\G(x, S)$ is naturally isomorphic to the
\emph{orbital  graph}. \index{orbital graph}
The set of vertices of the orbital
graph is the orbit of $x$, and two vertices $y_1, y_2$ are connected
by an arrow from $y_1$ to $y_2$ if there exists a generator $s\in S$
such that $\be(s)=y_1$ and $\en(s)=y_2$.

It is not hard to show that if $\G$ is a compactly generated groupoid
of germs, then
the set of points $x\in\G^{(0)}$ with trivial isotropy group is
co-meager. It follows that generically the Cayley graph $\G(x, S)$
is naturally isomorphic to the orbital graph of $x$. In
general, the Cayley graph $\G(x, S)$ is a covering of the orbital
graph. 

\begin{proposition}
\label{pr:equivalencegensets} If $(S_1, \X)$ and $(S_2, \X)$ are compact
generating pairs of $\G$, then there exists $n$ such that
$\bigcup_{k=1}^n(S_1\cup S_1^{-1})^k$ is a neighborhood of $S_2$ and
$\bigcup_{k=1}^n(S_2\cup S_2^{-1})^k$ is a neighborhood of $S_1$.
\end{proposition}

\begin{proof}
For every $g\in S_1$ there exists $n_g$ such that
$\bigcup_{k=1}^{n_g}(S_2\cup S_2^{-1})^k$ is a neighborhood of $g$. Then, by
compactness of $S_1$ there exists $n$ such that
$\bigcup_{k=1}^n(S_2\cup S_2^{-1})^k$ is a neighborhood of every point
$g\in S_1$.
\end{proof}

\begin{corollary}
If $(S_1, \X)$ and $(S_2, \X)$ are compact generating pairs of $\G$,
then the identity map is a quasi-isometry of the Cayley
graphs $\G(x, S_1)$ and $\G(x, S_2)$.
\end{corollary}

\begin{lemma}
Let $\X_1\subset\X_2$ be compact topological transversals. Let
$x\in\X_1$. Then the set $\G_x^{\X_1}$ is a net in the
Cayley graph $\G(x, S)$, where $(\X_2, S)$ is a compact generating pair.
\end{lemma}

\begin{proof}
We can find a finite collection of homeomorphisms $\{U_1, \ldots,
U_k\}$ such that $\be(U_i)$ cover $\X_2$, $\en(U_i)\subset\X_0$, and
$\overline{U_i}$ are compact. Then there
exists $n$ such that $\bigcup_{i=1}^k
\overline{U_i}\cap\G|_{\X_2}\subset\bigcup_{m=1}^n(S\cup S^{-1})^m$. Then for every
$g\in\G_x\cap\G|_{\X_2}$
there exists $U_i$ such that $U_i\cdot g\in\G|_{\X_1}$, hence the
set $\G_x\cap\G|_{\X_1}$ is an $n$-net in $\G(x, S)$.
\end{proof}

\begin{corollary}
Any two Cayley graphs $\G(x, S_1)$, $\G(x, S_2)$ associated
with the same point $x\in\X$ of a compactly generated groupoid are
quasi-isometric.
\end{corollary}

\begin{examp}
Let $\theta\in\R$ be an irrational number. Consider the group $G\cong\Z^2$ acting on
$\R$ and generated by the transformations $x\mapsto x+1$ and $x\mapsto
x+\theta$. Let $\G$ be the corresponding groupoid of germs. Note that
all $G$-orbits are dense, hence any open subset of $\R$ is a $G$-transversal.

Consider the action of
$G$ on $\R^2$ generated by the maps $a:\left(\begin{array}{c}x\\
    y\end{array}\right)
\mapsto\left(\begin{array}{c}x+1\\ y\end{array}\right)$ and $b:\left(\begin{array}{c}x\\
y\end{array}\right)\mapsto\left(\begin{array}{c}x+\theta\\
y+1\end{array}\right)$. Projection of this action onto the first
coordinate is the action of $G$ on $\R$, defined above.

If $I=[x_1, x_2]\subset\R$ is a finite closed
interval and $t\in (x_1, x_2)$, then $\G_t^I$ can be represented by
the part $\hat I$ of the $G$-orbit of $\left(\begin{array}{c}t\\
    0\end{array}\right)$ that is projected to $[x_1, x_2]$. Each point
$\left(\begin{array}{c} r_1\\ r_2\end{array}\right)\in\hat I$ represents
the germ of the translation $x\mapsto x+(r_1-t)$ at $t$. See Figure~\ref{fig:rotation}.

Fix a sufficiently big set $R\subset G$, and consider the set $S_R$ of germs at
points of $I$ of elements of $R$. Then the Cayley graph $\G(t, S_R)$ is isomorphic
to the graph with the set of vertices $\hat I$ in which two vertices
are connected by an edge if and only if one is the image of the other
under the action of an element of $R$.

\begin{figure}
\includegraphics{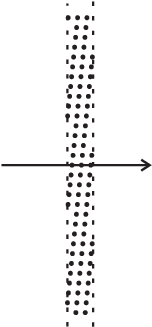}
\caption{Cayley graphs of irrational rotation}\label{fig:rotation}
\end{figure}

It is easy to see that the Cayley graphs of $\G$ are quasi-isometric to
$\R$. When we increase the interval $I$, the set of vertices of the
Cayley graph will increase, but the smaller Cayley graph is a net inside the larger one.

Note that the groupoid $\G$ is equivalent to the groupoid generated by
the rotation $x\mapsto x+\theta\pmod{1}$ of the circle $\R/\Z$. Orbits of
the latter can be naturally identified with $\Z$, and its Cayley
graphs coincide with the Cayley graphs of $\Z$.
\end{examp}

\begin{examp}
Consider the group $G$ acting on the Cantor set $\{0, 1\}^\infty$ and
generated by the transformations $a$ and $b$ defined inductively by the
rules
\[a(0w)=1w,\quad a(1w)=0b(w),\qquad b(0w)=0w,\quad b(1w)=1a(w)\]
for all $w\in\{0, 1\}^\infty$. This group is the \emph{iterated
  monodromy group} of the complex polynomial $z^2-1$
(see~\ref{ss:img}). It is often called the \emph{Basilica
  group}.\index{Basilica group}\index{group!Basilica}

We can take the whole $\{0, 1\}^\infty$ as our
transversal, and take the set of all germs of transformations $a$ and
$b$ as the generating set $S$. It is not hard to show that the groupoid of germs $\G$ of $G$ is
principal ($G$ is \emph{regular} in terminology of~\cite{nek:cpalg}). It follows
that the Cayley graphs $\G(x, S)$ coincide with the orbital graphs of
the action of $G$ on $\{0, 1\}^\infty$ (orbital graphs are also called
\emph{Schreier graphs} in the case of group actions).

The orbital graphs of the Basilica group where extensively studied
in~\cite{angelidonnomatternagnibeda}. In particular, it was shown that depending on the basepoint $x$
the graphs $\G(x, S)$ can have one, two, or four ends. In
particular, the quasi-isometry class of $\G(x, S)$ depends on $x$. In
fact, it follows from the results of~\cite{angelidonnomatternagnibeda}
that there are uncountably many
different quasi-isometry classes of the Cayley graphs $\G(x, S)$.
\end{examp}

\section{Relations in Hausdorff groupoids}

Let $\G$ be a \emph{Hausdorff} groupoid of germs. Let $\mS$ be a finite set of extendable compact
subsets $F\subset\G$. Let $\wh F\supset F$ be extensions of the sets $F\in\mS$ to elements of $\pG$.
We will denote the set of extensions $\wh F$ by $\wh\mS$. Denote $S=\bigcup_{F\in\mS}F$.

Suppose that $g_1g_2\cdots g_n=\be(g_n)$ for some $g_i\in S$.
Denote by $F_i$, for $i=1, 2, \ldots, n$, the element of $\mS$ such that $g_i\in F_i$.

Then $g_1g_2\cdots g_n=\be(g_n)$ is a germ of the
composition $\wh F_1\cdot\wh F_2\cdots\wh F_n$. Denote by $E$ the set of
points $x\in\be(\wh F_n)$ such that the germ of $\wh F_1\cdot \wh
F_2\cdots \wh F_n$ at $x$ is defined and is trivial (i.e., equal
to the germ of the identity).

\begin{lemma}
\label{lem:relations} The set $E$ is relatively closed and open in
$\be(\wh F_1\cdot\wh F_2\cdots\wh F_n)$.
\end{lemma}

\begin{proof}
The set $E$ is obviously open. Suppose that $x$ belongs to the
closure of $E$ in $\be(\wh F_1\cdot\wh F_2\cdots\wh F_n)$. Consider the germ
$g$ of the homeomorphism $\wh F_1\cdot\wh F_2\cdots\wh F_n$ at the point
$x$. Intersection of every open neighborhood $V$ of $x$ with $E$
is a non-empty open set, hence every neighborhood of $g$ has a
non-empty intersection with the identical homeomorphism, i.e., $g$
and the germ of the identity at $x$ do not have disjoint
neighborhoods. It follows that $g$ is a germ of identity, hence
$x$ belongs to $E$.
\end{proof}

\begin{corollary}
\label{cor:relationspseudogroup} For any finite set $R$ of
sequences $(F_1, F_2, \ldots, F_n)$ of elements of $\mS$
there exists $\epsilon>0$ such that for all $x, y\in\G^{(0)}$
such that $|x-y|<\epsilon$, and for every sequence
$(F_1, F_2, \ldots, F_n)\in R$, if the germ $(F_1\cdot F_2\cdots
F_n, x)$ is trivial, then so is the germ $(\wh F_1\cdot\wh
F_2\cdots\wh F_n, y)$.
\end{corollary}

\begin{proof}
For every sequence $r=(F_1, F_2, \ldots, F_n)\in R$ denote by
$E_r$ the set of points $x$ for which the germ of $\wh F_1\cdot\wh
F_2\cdots\wh F_n$ at $x$ is trivial. The set $E_r$ is open and
relatively closed in $\be(\wh F_1\cdot\wh F_2\cdots\wh F_n)$, by
Lemma~\ref{lem:relations}. Domain of $F_1\cdot
F_2\cdots F_n$ is contained in $\be(\wh F_1\cdot\wh F_2\cdots\wh
F_n)$ and is compact. It
follows that the intersection $E_r\cap\be(F_1\cdot F_2\cdots
F_n)$ is compact, and that $E_r$ is an open neighborhood of
$E_r\cap\be(F_1\cdot F_2\cdots F_n)$. Consequently, there exists
$\epsilon_r>0$ such that for every $x\in E_r\cap\be(F_1\cdot
F_2\cdots F_n)$ the $\epsilon_r$-neighborhood of $x$ is contained
in $E_r$. We can take $\epsilon$ equal to minimum of
$\epsilon_r$ for all $r\in R$.
\end{proof}

\section{Groupoids with additional structures}
\label{s:graddstr}
\subsection{Lipschitz structure}

\begin{lemma}
\label{lem:Lipschstructure}
Let $\G$ be a  groupoid of germs, and let $\X_1$ and $\X_2$ be
compact topological $\G$-transversals. Let $\ell_1$ be a positive log-scale
on $\X_1$ such that the elements
of  $\pG|_{\X_1}$ are locally Lipschitz with respect to $\ell_1$. Then
there exists a unique, up to Lipschitz equivalence, positive
log-scale $\ell_2$ on $\X_2$ such that the elements of $\pG|_{\X_1\cup\X_2}$ are
locally Lipschitz with respect to $\ell_1$ and $\ell_2$.
\end{lemma}

\begin{proof}
There exists a finite set $F_i$, $i\in I$, of relatively compact extendable
elements of $\pG$ such that $\be(F_i)$ cover $\X_2$
and $\en(F_i)$ are subsets of the interior of $\X_1$. Define
$\wt\ell_i(x, y)=\ell_1(F_i(x), F_i(y))$ for $x,
y\in\be(\overline{F_i})$. We get a covering of the compact set $\X_2$
 by open subsets $U_i=\be(F_i)$ and a collection of log-scales $\wt\ell_i$
defined on $\overline{U_i}$ such that $\wt\ell_i$ and $\wt\ell_j$ are
Lipschitz equivalent to each other on the intersection
$\overline{U_i}\cap\overline{U_j}$ (since an extension of
$\overline{U_j}(\overline{U_i})^{-1}$ is locally Lipschitz and
$\ell_1$ is positive).
Hence, by Theorem~\ref{th:pastingLipschitz}, there exists a log-scale
$\ell_2$ on $\X_2$ which is Lipschitz equivalent to every log-scale
$\wt\ell_i$, $i\in I$. It is easy to see that this log-scale satisfies
the conditions of the lemma.
\end{proof}

In view of Lemma~\ref{lem:Lipschstructure}, we adopt the following
definition.

\begin{defi}
A \emph{Lipschitz structure} on $\G$
\index{Lipschitz structure on a groupoid}
\index{groupoid!with a Lipschitz structure}
is given by a positive log-scale $\ell$ on a compact topological
transversal $\X$ such that all elements of $\pG$
act by locally Lipschitz transformations with respect to $\ell$.
If $\ell'$ is a log-scale on
a compact topological transversal $\X'$, then $\ell$ and $\ell'$ define the same
Lipschitz structure if the elements of $\pG|_{\X\cup\X'}$ are locally
Lipschitz with respect to $\ell$ and $\ell'$.
(In particular, $\ell$ and $\ell'$ are locally Lipschitz equivalent on
$\X\cap\X'$.)

Two equivalent groupoids $\G_1$ and $\G_2$ have equivalent Lipschitz
structures (are equivalent as groupoids with Lipschitz structure)
if their Lipschitz structures define a Lipschitz structure on $\G_1\vee\G_2$.
\end{defi}

\subsection{Local product structure}

\begin{defi}
Let $\X$ be a topological space with a local product structure. A
homeomorphism $F:U\arr V$ between open subsets of $\X$
\emph{preserves the local product structure} if for every point
$x\in U$ and a pair of rectangles $(R_i, [\cdot, \cdot]_i)$,
$(R_j, [\cdot, \cdot]_j)$ such that $x\in R_i$ and $F(x)\in R_j$
there exists a rectangular neighborhood $W$ of $x$ such that
\[F([y, z]_i)=[F(y), F(z)]_j\]
for all $y, z\in W$. \index{groupoid!preserving a local product
  structure} \index{pseudogroup!preserving a local product structure}

Let $\X$ be a space with a local product structure. A pseudogroup
$\pG$ acting on $\X$ (and its groupoid of germs)
is said to preserve the local product
structure if every element $F:U\arr V$ of $\pG$ preserves the
local product structure.
\end{defi}

Note that if a groupoid of germs $\G$ preserves a local product structure on
$\G^{(0)}$ then $\G$ itself has a natural
local product structure. Namely, for every germ $(F, x)$ we can find
a rectangular neighborhood $R$ of $x$ such $F(R)$ is also a
rectangle and $F([y, z])=[F(y), F(z)]$ for all $y, z\in R$. We
transform then the set of germs $\{(F, x)\;:\;x\in R\}$ into a
rectangle by setting $[(F, y), (F, z)]=(F, [y, z])$.

\begin{defi}
Let $\G$ be a groupoid preserving a local product structure and let
$\mathcal{R}$ be a covering of a $\G$-transversal by open rectangles.
An element $U\in\pG$ is called a \emph{rectangle subordinate to
$\mathcal{R}$} \index{subordinate rectangle} \index{rectangle!subordinate}
if there exist rectangles $R_i, R_j\in\mathcal{R}$ such
that $\be(U)\subset R_i$, $\en(U)\subset R_j$ and $U([x,
y]_{R_i})=[U(x), U(y)]_{R_j}$.
\end{defi}

Next proposition is a direct corollary of the definitions.

\begin{proposition}
\label{pr:extlocprod}
Let $\G$ be a groupoid of germs and suppose that there exists a
local product structure preserved by $\pG$ on an open transversal
$\X_0$. Then the local product structure on $\X_0$ can
be extended in a unique way to a local product structure on
$\G^{(0)}$ preserved by $\pG$.
\end{proposition}

\begin{corollary}
Let $\pG_1$ be a pseudogroup preserving a local product
structure of the space $\G_1^{(0)}$ and let $\pG_2$ be an
equivalent pseudogroup. Then there exists a unique local product
structure on $\G_2^{(0)}$ such that the equivalence pseudogroup
$\pG_1\vee\pG_2$ preserves the local product
structure of the disjoint union $\G_1^{(0)}\sqcup\G_2^{(0)}$.
\end{corollary}

Let $\pG$ be a pseudogroup preserving a local product structure of
$\G^{(0)}$. Let $\mathcal{R}=\{R_i\}_{i\in I}$ be a covering of $\G^{(0)}$ by
rectangles.
Let $F:U\arr V$ be a rectangular element of $\pG$ such that $U$
and $V$ are a sub-rectangles of $R_i$ and $R_j\in\mathcal{R}$ respectively.
Then there exist homeomorphisms $\proj_1(F):\proj_1(U)\arr\proj_1(V)$ and
$\proj_2(F):\proj_2(U)\arr\proj_2(V)$ such that $F(x, y)=(\proj_1(F)(x),
\proj_2(F)(y))$ with respect to the canonical decompositions
$U=\proj_1(U)\times\proj_2(U)$ and $V=\proj_1(V)\times\proj_2(V)$. We call the
homeomorphisms $\proj_1(F)$ and $\proj_2(F)$ \emph{projections}
\index{projection of an element of a pseudogroup} of the
rectangle $F$.

For a germ $g\in\G$, a rectangle $U\in\pG$ such that $g\in U$,
and an index $i=1,2$, the germ of the projection
$\proj_i(U)$ at $\proj_i(\be(g))$ is denoted $\proj_i(g)$.

\begin{defi}
Let $k=1$ or $2$. \emph{Projection} $\proj_k(\pG, \mathcal{R})$ of
$\pG$ (with respect to the covering $\mathcal{R}$) is the
pseudogroup of local homeomorphisms of the space
$\proj_k(\G^{(0)})=\bigsqcup_{i\in I}\proj_k(R_i)$ generated by the projections
$\proj_k(F)$ of rectangular elements $F$ of
$\pG_{\mathcal{R}}$. \index{projection of a groupoid or a pseudogroup}
\end{defi}

We will denote by $\proj_k(\G, \mathcal{R})$ the groupoid of
germs of the pseudogroup $\proj_k(\pG, \mathcal{R})$.

Let $\pG$ be a pseudogroup preserving a local product structure on
$\G^{(0)}$. Let $\X$ be a compact topological transversal. Let $|\cdot|$ be a
metric on a neighborhood $\wh\X$ of $\X$.

\begin{defi}
\label{def:compressible}
We say that a finite covering $\mathcal{R}$
of $\X$ by relatively compact rectangles $R_i\subset\wh\X$ has
\emph{compressible first (resp.\ second) direction} if there exists a constant
$\lambda\in (0, 1)$ such that
for every $R_i\in\mathcal{R}$ and every $x\in R_i$ there
exists a rectangular element $F\in\pG$ and a
rectangle $R_j\in\mathcal{R}$ such that
$\proj_k(R_i, x)\subset\be(F)$, $F(\proj_k(R_i, x))\subset R_j\cap\X$, and
\[|F(y_1)-F(y_2)|<\lambda |y_1-y_2|\]
for all $y_1, y_2\in\proj_k(R_i, x)$, where $k=1$ (resp.\ $k=2$).

We say that a covering by rectangles is \emph{compressible} if it is
compressible in both directions.  \index{compressible covering}
\end{defi}

\begin{proposition}
\label{prop:compressible}
If $\mathcal{R}$ and $\mathcal{R}'$ are coverings that are compressible in
the first direction, then the second projections $\proj_2(\pG,
\mathcal{R})$ and $\proj_2(\pG, \mathcal{R}')$ are equivalent.
\end{proposition}

Note that the conclusion of
Proposition~\ref{prop:compressible} holds in many other cases, not
only when the covering is compressible.

\begin{proof}
Let $\mathcal{R}$ be a compressible covering of $\X$.
It is enough to prove the proposition
for every covering $\mathcal{R}'$ of $\X$ by sub-rectangles of
elements of $\mathcal{R}$, since for any two coverings there exists a
covering by rectangles that are sub-rectangles of both coverings.

For $R'\in\mathcal{R}'$ and $R\in\mathcal{R}$ such that $R'\subset R$
the projection $\proj_2(R')$ is naturally identified with a subset of
$\proj_2(R)$. In this way the projections $\proj_2(R')$ for
$R'\in\mathcal{R}'$
cover a transversal
of the space $\bigsqcup_{R\in\mathcal{R}}\proj_2(R)$ on which $\proj_2(\pG,
\mathcal{R})$ acts. Let us show that $\proj_2(\pG, \mathcal{R}')$ is
equivalent to the localization
$\proj_2(\pG, \mathcal{R})|_{\proj_2(\mathcal{R}')}$. See~\ref{ss:localization} for the
definition of a localization.

The generators of the pseudogroup $\proj_2(\pG,
\mathcal{R}')$ are naturally identified with elements of the localization
$\proj_2(\pG, \mathcal{R})|_{\proj_2(\mathcal{R}')}$. Namely, the projection of
an element $F\in\pG$ onto $\proj_2(R')$ for $R'\in\mathcal{R}'$
is naturally identified with the restrictions to
$\proj_2(R')\subset \proj_2(R)$ of the projection of $F$ onto
$\proj_2(R)$, where $R\in\mathcal{R}$ is such that $R\supset R'$.

The only generators of the localization $\proj_2(\pG,
\mathcal{R})|_{\proj_2(\mathcal{R}')}$ which are not obtained as
projections of  elements of $\pG|_{\mathcal{R}'}$ are the
identifications of the common parts of $\proj_2(R_1')$  and
$\proj_2(R_2')$ for $R_1', R_2'\in\mathcal{R}'$ such that
$R_1'\cap R_2'=\emptyset$. It is enough therefore
to show that these identification still
belong to $\proj_2(\pG, \mathcal{R}')$.

Let $\epsilon>0$ be a Lebesgue's number of the covering of the compact
set $\X$ by elements of $\mathcal{R}'$. Let $D$ be an
upper bound on diameter of the elements of $\mathcal{R}$.

Let $R_1'\cup R_2'\subset R$ for $R\in\mathcal{R}$, and
let $x\in R$ be such that $\proj_2(x)\in \proj_2(R_1')\cap
\proj_2(R_2')$. By the condition of the proposition, there exists a rectangle
$F\in\pG$ such that $\proj_1(R, x)\subset\be(F)$, $F$ contracts the
distances inside $\proj_1(R, x)$ at least by $\lambda$, and
$F(\proj_1(R, x))\subset R_k\cap\X$ for some $R_k\in\mathcal{R}$.
Applying this condition $n$ times, where $n$ is such
that $D\lambda^n<\epsilon$, we conclude that there exists a rectangle
$F\in\pG$
and $R'\in\mathcal{R}'$ such that $\proj_1(R, x)\subset\be(F)$ and
$F(\proj_1(R, x))\subset R'$. Then projections onto $\proj_2(R_1')$
and $\proj_2(R_2')$ of the restrictions of $F$ to
$\be(F)\cap R_1'$ and $\be(F)\cap
R_2'$ are generators of $\proj_2(\pG, \mathcal{R}')$. But projections of
these generators onto $\proj_2(R)$ are equal. It follows that the
identical gluing map between the common parts of the subsets
$\proj_2(R_1')$ and $\proj_2(R_2')\subset \proj_2(R)$ belongs to $\proj_2(\pG,
\mathcal{R}')\subset \proj_2(\pG, \mathcal{R})|_{\proj_2(\mathcal{R}')}$.

We have shown that all generators of $\proj_2(\pG,
\mathcal{R})|_{\proj_2(\mathcal{R}')}$ belong to $\proj_2(\pG,
\mathcal{R}')$, hence these two pseudogroups coincide, which
finishes the proof.
\end{proof}

\subsection{Locally diagonal groupoids}

\begin{defi}
\label{def:locdiagonal}
Let $\G$ be a  groupoid of germs preserving a local product
structure on $\G^{(0)}$. We say that the groupoid $\G$ is
\emph{locally diagonal} (with respect to the local product structure)
\index{groupoid!locally diagonal}
if there exists a covering $\mathcal{R}$ of a topological transversal
of $\G$ by open rectangles such that if $g\in\G$ is such that $\be(g)\in R$,
$\en(g)\in R$ for some $R\in\mathcal{R}$ and one of the projections
$\proj_i(g)$, $i=1,2$, is a unit, then $g$ is a unit.
\end{defi}

It is clear that if a covering $\mathcal{R}$ satisfies the conditions
of Definition~\ref{def:locdiagonal}, then any covering by
subrectangles of elements of $\mathcal{R}$ also satisfies the
conditions of the definition.

\begin{proposition}
Suppose that a covering $\mathcal{R}$ of a topological transversal
$\X\subset\G^{(0)}$ satisfies the conditions of
Definition~\ref{def:locdiagonal}. Then for any topological transversal
$\X'$ there exists a covering $\mathcal{R}'$ of $\X'$ satisfying
the conditions of Definition~\ref{def:locdiagonal}.
\end{proposition}

\begin{proof}
We can find a covering $\mathcal{R}'$ of $\X'$ by open rectangles
such that for every $R'\in\mathcal{R}'$
there exists a rectangle $U\in\pG$ such that $\be(U)=R'$, $\en(U)$
is a sub-rectangle of a rectangle $R\in\mathcal{R}$, and $U$ agrees
with the product decomposition of $R$ and $R'$. Then for any element
$g\in\pG$ such that $\be(g), \en(g)\in R'$ we have the corresponding
conjugate $h=UgU^{-1}$ such that $\be(h), \en(h)\in R$. A projection
$\proj_i(g)$ is trivial if and only if the projection $\proj_i(h)$ is
trivial. The element $g$ is a unit if and only if $h$ is a unit. It
follows that $\mathcal{R}'$ satisfies the conditions of
Definition~\ref{def:locdiagonal}.
\end{proof}

\begin{corollary}
If a groupoid $\G$ preserving a local product structure on $\G^{(0)}$
is locally diagonal, then every equivalent groupoid is also locally diagonal.
\end{corollary}

\subsection{Quasi-cocycles}

\begin{defi}
\label{def:graded} A \emph{quasi-cocycle} \index{quasi-cocycle!on a
  groupoid} on a groupoid $\G$ is a map $\coc:\G|_{\X}\arr\R$, where
$\X$ is a compact topological transversal,  such that
there exists a constant $\eta\ge 0$ for which the following conditions hold:
\begin{enumerate} \item
for every $g\in\G|_{\X}$ there exists a neighborhood $U$ of $g$ such
that
\[|\coc(g)-\coc(h)|\le\eta\] for all $h\in U$;
\item
for all $(g_1, g_2)\in\G|_{\X}^{(2)}$
\[\coc(g_1)+\coc(g_2)-\eta\le\coc(g_1g_2)\le\coc(g_1)+\coc(g_2)+\eta.\]
\end{enumerate}
If the above inequalities hold for $\coc$ and $\eta$, then we say that $\coc$ is
an \emph{$\eta$-quasi-cocycle}.
\end{defi}

\begin{defi}
We say that two
quasi-cocycles $\coc_1$ and $\coc_2$ (defined for the same
transversal $\X$) are \emph{strongly equivalent}
\index{strongly equivalent quasi-cocycles}
\index{quasi-cocycles!strongly equivalent} if
$|\coc_1(g)-\coc_2(g)|$ is uniformly bounded. 
They are \emph{(coarsely) equivalent}
\index{coarsely equivalent quasi-cocycles}
\index{quasi-cocycles!coarsely equivalent}
if there exist $\Lambda>1$ and $k>0$ such that
$\Lambda^{-1}\cdot\coc_1(g)-k\le\coc_2(g)\le\Lambda\cdot\coc_1(g)+k$
for all $g\in\G|_{\X}$.
\end{defi}

\begin{lemma}
\label{lem:coconx2}
Let $\X_1$ and $\X_2$ be compact topological transversals of a
groupoid of germs $\G$. Let
$\coc_1$ be a quasi-cocycle on $\G|_{\X_1}$. Then there exists a
unique, up to strong equivalence, quasi-cocycle
$\coc$ on $\G|_{\X_1\cup\X_2}$ such that $\coc$ and $\coc_1$ are strongly
equivalent on $\G|_{\X_1}$.
\end{lemma}

\begin{proof}
Let us find a finite set $F_i$ of relatively compact elements of $\pG$ such
that $\be(F_i)$ cover $\X_2$ and $\en(F_i)\subset\X_1$.
For every $g\in\G|_{\X_1\cup\X_2}$ we either have
$g'=g\in\G_{\X_1}$, or $g'=h_2gh_1^{-1}\in\G_{\X_1}$, or
$g'=gh_1^{-1}\in\G_{\X_1}$, or $g'=h_2g\in\G_{\X_1}$ for some $h_1, h_2\in
F=\bigcup F_i$. If we define $\coc(g)=\coc_1(g')$, then $\coc$ will be
a quasi-cocycle satisfying the conditions of the lemma (it is defined
only up to a strong equivalence, since $h_i$ and $g$ are not unique).

For every quasi-cocycle $\coc$ on $\G|_{\X_1\cup\X_2}$ the set of
values of $\coc$ on $F\cap\G|_{\X_1\cup\X_2}$ is bounded. For every
element $g\in\G|_{\X_2}$ there exist elements $h_1, h_2\in
F\cap\G|_{\X_1\cup\X_2}$
such that $h_2gh_1^{-1}$ is an element of $\G|_{\X_1}$. It follows
that $\coc$ is unique up to a strong equivalence.
\end{proof}

We give then, in view of the previous lemma, the following
definition.

\begin{defi}
A \emph{graded groupoid} is a groupoid of germs $\G$ together with
a quasi-cocycle $\coc:\G|_{\X}\arr\R$. \index{groupoid!graded} \index{graded groupoid}

Two quasi-cocycles $\coc_1:\G|_{\X_1}\arr\R$ and
$\coc_2:\G|_{\X_2}\arr\R$ \emph{define the same grading} of $\G$ (or are
\emph{strongly equivalent}) if
there exists a quasi-cocycle $\coc:\G|_{\X_1\cup\X_2}\arr\R$ such
that the restrictions of $\coc$ to
$\G|_{\X_i}\subset\G|_{\X_1\cup\X_2}$ are
strongly equivalent to $\coc_i$.

Two graded groupoids $\G_1$ and $\G_2$ are \emph{equivalent as graded
groupoids} if the corresponding quasi-cocycles defined the same grading
of $\G_1\vee\G_2$.
\end{defi}

If a graded groupoid $\G$ is compactly generated, then the
corresponding quasi-cocycle $\coc$ defines a quasi-cocycle $\wt\coc$
in the sense of Definition~\ref{def:almcoc} on each
of its Cayley graphs $\G(x, S)$ by $\wt\coc(g,
h)=\coc(gh^{-1})$.

\begin{defi}
\label{def:positive}
Let $\coc:\G\arr\R$ be an $\eta$-quasi-cocycle. An element $g\in\G$ is
said to be \emph{positive} (with respect to $\coc$) if
$\coc(g)>2\eta$. \index{positive element of a groupoid}
\end{defi}

\subsection{Compatibility with the local product structure}

If $\ell_1$ and $\ell_2$ are log-scales on spaces $A$ and $B$
respectively, then we denote by $\ell_1\times\ell_2$ the log-scale on
$A\times B$ given by
\begin{equation}
\label{eq:prodscale}
(\ell_1\times\ell_2)((a_1, b_1), (a_2, b_2))=\min\{\ell_1(a_1, a_2),
\ell_2(b_1, b_2)\}.
\end{equation}

It is easy to check that it is a log-scale and that it defines the
product topology on the space $A\times B$.

\begin{defi}
\label{def:compatiblelipsch}
Let $\ell$ be a log-scale on a compact space $\X$ with a local product
structure. We say that the log-scale is \emph{compatible with the
local product structure} if there
\index{log-scale!compatible with a local product structure}
\index{metric!compatible with a local product structure}
\index{compatible with a local product structure!metric}
\index{compatible with a local product structure!log-scale}
exists a finite covering $\mathcal{R}=\{R_i\}$ of $\X$ by
open rectangles and log-scales $\ell_{i, 1}$ and $\ell_{i, 2}$ on
$\proj_1(R_i)$ and $\proj_2(R_i)$ such that $\ell$ is Lipschitz
equivalent to $\ell_{i, 1}\times\ell_{i, 2}$ on $R_i$.
\end{defi}

\begin{defi}
\label{def:compatiblegraded}
Let $\coc:\G|_{\X}\arr\R$ be a quasi-cocycle. We say that the
corresponding grading is \emph{compatible with the local product
structure}
\index{quasi-cocycle!compatible with a local product structure}
\index{compatible with a local product structure!quasi-cocycle}
if there exist a covering of $\X$ by open
rectangles $\mathcal{R}=\{R_i\}$ and a constant $c>0$ such that
if $g_1, g_2\in\Gh$ are such that $\{\be(g_1), \be(g_2)\}\subset R_i$,
$\{\en(g_1), \en(g_2)\}\subset R_j$ for some $R_i, R_j\in\mathcal{R}$,
and for some
$i\in\{1, 2\}$ we have $\proj_i(g_1)=\proj_i(g_2)$, then
$|\coc(g_1)-\coc(g_2)|\le c$.
\end{defi}

\chapter{Hyperbolic groupoids}
\label{s:hypgroupoids}

\section{Main definition}
\begin{defi}
\label{def:hyperbolic}
We say that a Hausdorff groupoid of germs $\G$ is \emph{hyperbolic}
\index{groupoid!hyperbolic} \index{hyperbolic groupoid} if
there exists a compact generating pair $(S, \X)$ of $\G$, a
metric $|\cdot|$ defined on an open neighborhood of $\X$, and
numbers $\lambda\in (0, 1)$, $\delta, \Lambda, \Delta>0$ such that
\begin{enumerate}
\item every element $g\in S$ is a germ of an element $U\in\pG$ that
  is a $\lambda$-contraction with respect to $|\cdot|$, i.e.,
$|U(x)-U(y)|\le\lambda |x-y|$ for all $x, y\in\be(U)$;
\item for every $x\in\X$ the Cayley graph $\G(x, S)$ is
  $\delta$-hyperbolic;
\item for every $x\in\X$ there exists a point $\omega_x$ of the
  boundary of $\G(x, S)$ such that every infinite directed path in
$\G(x, S^{-1})$ is a $(\Lambda, \Delta)$-quasi-geodesic converging to
$\omega_x$.
\item elements of $\pG$ are locally Lipschitz with respect to $|\cdot|$;
\item $\be(S)=\en(S)=\X$;
\end{enumerate}
\end{defi}

Definition~\ref{def:hyperbolic} is a combination of two contraction
conditions: the elements of $S$ are germs of contracting maps on
$\G^{(0)}$, while the elements of $S^{-1}$ are large-scale contracting
on the Cayley graphs (see Theorem~\ref{th:contraction}).

\begin{examp}
\label{ex:expanding}
As a simple example of a hyperbolic groupoid, consider the groupoid
$\mathfrak{F}$ of
germs of the pseudogroup of transformations of the circle $\R/\Z$
generated by the self-covering $f:x\mapsto 2x$. Consider the
generating pair $(S, \R/\Z)$, where $S$ is the set of all germs of
$f^{-1}$. The set of vertices of the Cayley graph $\mathfrak{F}(x, S)$
is the set of germs of the form $(f^{n_1}, y)^{-1}(f^{n_2}, x)$, where
$y\in\R/\Z$ and $n_1, n_2\ge 0$ are such that
$f^{n_2}(x)=f^{n_1}(y)$. Every point of $\R/\Z$ has exactly two
$f$-preimages and one $f$-image. The Cayley graph
$\mathfrak{F}(x, S)$ is a regular tree of degree three such that every
vertex has one incoming and two outgoing arrows.
See Figure~\ref{fig:cayleytree}. (Recall
that the edges of the Cayley graph $\mathfrak{F}(x, S)$ are oriented
according to the action of the elements of $S$, i.e., according to the
action of $f^{-1}$.) If the isotropy group of $x$ in $\mathfrak{F}$ is
trivial (i.e., if $x$ is not eventually periodic with respect to the
action of $f$), then the vertices of $\mathfrak{F}(x, S)$ are in a bijection with the
\emph{grand orbit} of $x$ (i.e., with the $\mathfrak{F}$-orbit of $x$), and the
graph $\mathfrak{F}(x, S)$ is the graph of the action of $f^{-1}$ on the grand orbit.

\begin{figure}
\includegraphics[width=2in]{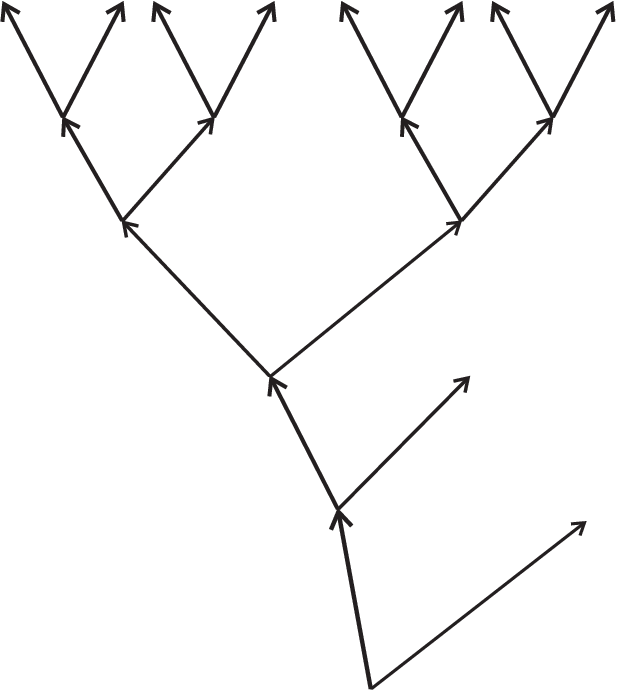}
\caption{Cayley graph of $\mathfrak{F}$}\label{fig:cayleytree}
\end{figure}

It is easy to see now that $\mathfrak{F}$ and $S$ satisfy the
conditions of Definition~\ref{def:hyperbolic}. Elements of $S$ are
germs of contractions (since $f$ is expanding), the Cayley graphs
$\mathfrak{F}(x, S)$ are Gromov-hyperbolic (are $0$-hyperbolic, in
fact), and the paths going against the orientation in the Cayley
graphs are geodesics converging to one point of the boundary (to the
limit of the forward $f$-orbit of $x$).
\end{examp}

\begin{examp}
\label{ex:dyadic}
 Let $\Z_2$ be the ring of dyadic integers, i.e., formal
  expressions \[\sum_{n=0}^\infty a_n2^n,\] where $a_n\in\{0,
  1\}$. Distance between $\sum_{n=0}^\infty
  a_n2^n$ and $\sum_{n=0}^\infty b_n2^n$ is $1/2^k$, where $k$ is the
 smallest index such that $a_k\ne b_k$.

Then the map $a:x\mapsto x+1$ is an isometry, and the map $s:x\mapsto
2x$ is a homeomorphism onto its range, contracting the distance twice.
Let $\pG$ be the pseudogroup generated by these transformations, and
let $\G$ be the corresponding groupoid of germs.

Consider then the generating set $S$ consisting of
germs of the transformations $s$, $as$, and $a^{-1}s$. Note that the
set of germs of $a$ is equal to the disjoint union of the sets of germs of
$as\cdot s^{-1}$ and $s\cdot (a^{-1}s)^{-1}$, therefore $S$ is a
generating set.

\begin{figure}
\includegraphics[width=3in]{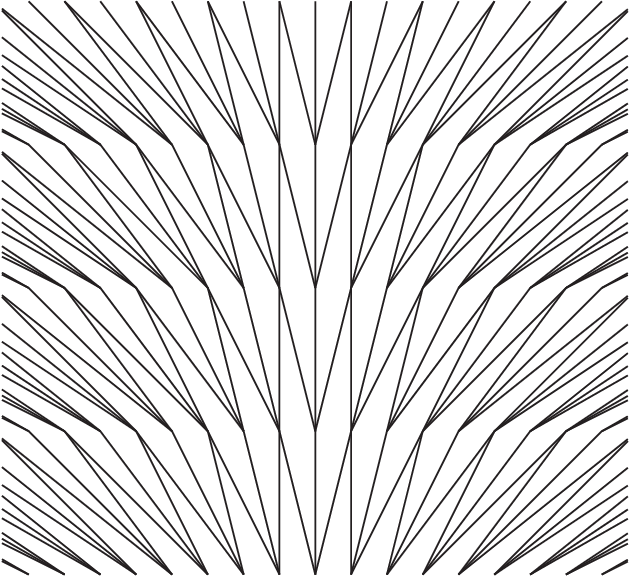}
\caption{Cayley graph of $\G$}\label{fig:admachcayley}
\end{figure}

The Cayley graph of $\G$ with respect to this generating set is shown
on Figure~\ref{fig:admachcayley}. Here the edges are directed upward.
It is not hard to check that the groupoid $\G$ and the generating set
$S$ satisfy the conditions of Definition~\ref{def:hyperbolic}.
\end{examp}

Let $\G$ be a hyperbolic groupoid.

\begin{defi}
\label{def:contracting}
We say that a subset $S\subset\G$ is \emph{contracting} (with
respect to a metric on a subset of $\G^{(0)}$) if
every germ $g\in S$ belongs to a contraction
$U\in\pG$. \index{contracting subset of a groupoid}
\end{defi}

\begin{proposition}
\label{prop:invariancehyperbolic}
Suppose that $\G$ is a hyperbolic groupoid. Let $\Y_1$ be a
compact topological $\G$-transversal.

Then there exists a generating pair $(S, \Y_1)$
of $\G$ satisfying the conditions of
Definition~\ref{def:hyperbolic}.
\end{proposition}

\begin{proof} Let $\X$, $\wh{\X}$, $|\cdot|$, and $S$ be as in
Definition~\ref{def:hyperbolic}. Let $\X_0\subset\X$ be an
open transversal. Denote by $\ell$ a log-scale on $\wh{\X}$ such
that $|\cdot|$ is associated with $\ell$.

Let $\wh Y$ be a relatively compact
open neighborhood of $Y_1$. Let
$Y_0\subset Y_1$ be an open transversal. Then, by
Lemma~\ref{lem:Lipschstructure}, there exists a unique, up to a
Lipschitz equivalence, log-scale $\ell'$ on $\wh Y$ such
that all elements of $\pG$ are locally Lipschitz with respect to
$\ell$ and $\ell'$.

Cover the closure of $\wh Y$ by a finite number of sets
of the form $\en(U_i)$, where $U_i\in\pG$ are relatively compact,
extendable, and such that $\be(U_i)\subset\X_0$.
Similarly, cover $\X$ by sets of the form $\be(V_i)$ where
$V_i\in\pG$ are relatively compact, extendable, and
$\en(V_i)\subset Y_0$.

For every $n$ consider the restriction $R_n$ of the set
\[A_n=\left(\bigcup\overline{U_i}\cup\bigcup\overline{V_i}\right)\cdot S^n\cdot
\left(\bigcup\overline{U_i}\cup\bigcup\overline{V_i}\right)^{-1}\subset\G\]
to $Y_1$. Let $S_n'=R_n\cup R_{n+1}$. Then $(S'_n, Y_1)$ is a
generating pair of $\G$ for every $n$ (see the proof of
Proposition~\ref{pr:differnttransv}). We have $\be(S'_n)=\en(S'_n)=Y_1$.
Since $V_i$ and $U_i$ are locally Lipschitz, all elements of $S'_n$ will have contracting
neighborhoods for all $n$ big enough. Moreover, if $|\cdot|'$ is a
metric associated with $\ell'$, then for all $n$ big enough all elements of $S'_n$ have
contracting neighborhoods with respect to $|\cdot|'$.

Let $y\in\Y_1$ be an arbitrary point. Consider the Cayley graph
$\G(y, S'_n)$. Let $U_i\in\pG$ be such that $y\in\en(U_i)$. For every
$g\in\G_y^{Y_1}$ find $U_{j(g)}\in\pG$ such that $\en(g)\in\en(U_{j(g)})$,
and define $F(g)=U_{j(g)}^{-1}gU_i$. Then for any choice of indices
$j(g)$, the map $F:\G_y^{Y_1}\arr\G_x^{\X}$ is a quasi-isometry of
the Cayley graphs $\G(y, S_n')$ and $\G(y, S)$. Moreover, there is a
uniform estimate on the coefficients of the quasi-isometry. It follows
that condition (2) of Definition~\ref{def:hyperbolic}
is satisfied for $(S_n', Y_1)$. The directed
paths in the Cayley graph $\G(y, (S_n')^{-1})$ are quasi-geodesics
converging to $F^{-1}(\omega_x)$, since their $F$-images are on
a uniformly bounded distance from
the quasi-geodesics converging to $\omega_x$ obtained from the oriented
paths in $\G(x, S^{-1})$ by making jumps of length $n$ or $n+1$.
\end{proof}

\begin{corollary}
\label{cor:equivhyperbolic}
A groupoid of germs equivalent to a hyperbolic groupoid is
hyperbolic.
\end{corollary}

\section{Busemann quasi-cocycle}

Let $\G$ be a hyperbolic groupoid. Consider its Cayley graph $\G(x,
S)$. Let $\omega_x\in\partial\G(x, S)$ be as in
Definition~\ref{def:hyperbolic}. Consider the associated Busemann
quasi-cocycle $\beta_{\omega_x}(g, h)$.

\begin{proposition}
The map $\beta:\G|_{\X}\arr\R$ given by $\beta(gh^{-1})=\beta_{\omega(x)}(g,
h)$ is a well defined (up to a strong equivalence) groupoid quasi-cocycle.
\end{proposition}

Note that the strong equivalence class of $\beta$ depends on the
generating pair $(\X, S)$.

\begin{proof}
If we have two pairs $g_1, h_1$ and $g_2, h_2$ of elements of $\G$ and two points $x_1,
x_2$ such that
$g_1h_1^{-1}=g_2h_2^{-1}$ and $\be(g_i)=\be(h_i)=x_i$, then
$g_1^{-1}g_2=h_1^{-1}h_2$, and the map
$F:\G(x_1, S)\arr\G(x_2, S)$ given by
\[F(g)=g\cdot g_1^{-1}g_2\]
is an isomorphism of directed graphs mapping $\omega_{x_1}$ to
$\omega_{x_2}$. It follows that $\beta_{\omega_{x_1}}(g_1,
h_1)\doteq\beta_{\omega_{x_2}}(g_2, h_2)$, i.e., that $\beta$ is well defined.

It satisfies the condition $\beta(gh)\doteq\beta(g)+\beta(h)$, since the
Busemann quasi-cocycle is a quasi-cocycle on the Cayley graph. It remains to
prove that there exists a constant $c>0$ such that for every
$g\in\G|_{\X}$ there exists a neighborhood $U$ of $g$ such that
$|\beta(g)-\beta(h)|<c$ for all $h\in U$. Since $\omega_x$ is the
limit of every directed path in $\G(x, S^{-1})$, we can compute
$\beta_{\omega_x}(g, h)$ by following arbitrary directed paths in
$\G(x, S^{-1})$ starting in $g$ and $h$, and finding a moment when
they are on some bounded distance $\Delta$ (depending only on the
constant $\delta$ of hyperbolicity) from each other. It follows that
the value of $\beta_{\omega_x}(g, h)$ depends only on a finite
subgraph of the Cayley graph, so using Corollary~\ref{cor:relationspseudogroup}
we conclude that $\beta$ is locally bounded.
\end{proof}

\begin{defi}
Let $\G$ be a hyperbolic groupoid. A quasi-cocycle $\beta$ on $\G$ is
called a \emph{Busemann quasi-cocycle} \index{Busemann
  quasi-cocycle!on a hyperbolic groupoid} if there exists a
generalized Busemann quasi-cocycle $\beta_{\omega_x}$ associated with
$\omega_x$ on the Cayley graph $\G(x, S)$, such that
$\beta(gh^{-1})\doteq\beta_{\omega_x}(g, h)$.
\end{defi}

\begin{examp}
\label{ex:derivativcoc}
Let $f$ be a complex rational function seen as a self-map of the
Riemann sphere. Suppose that it is \emph{hyperbolic}, i.e., that it is
expanding on a neighborhood of its Julia set $J_f$. Then the groupoid
of germs $\mathfrak{F}$ generated by the restriction of $f$ to $J_f$ is
hyperbolic (see  Example~\ref{ex:expanding}). Note
that a natural Busemann cocycle is given just by the degree:
\[\beta((f^n, x)^{-1}(f^m, y))=n-m.\]
This is precisely the Busemann cocycle on the Cayley graph (described
in Example~\ref{ex:expanding}).

On the other hand, it is easy to see that
\[\coc(F, x)=-\ln|F'(x)|,\]
for $(F, x)\in\mathfrak{F}$, satisfies
\[C^{-1}n-D\le\coc((f^n, x)^{-1})\le Cn+D\]
for all $n\ge 0$ and all $x\in J_f$, where $C>1$ and $D>0$ do not
depend on $n$ and $x$. Consequently, $\coc$ is 
also a Busemann cocycle on the hyperbolic groupoid generated by $f$.
\end{examp}

\section{Complete generating sets}

This section is purely technical. We need to prove existence of
generating sets of hyperbolic groupoids with special properties that
will be convenient later.

\begin{proposition}
\label{pr:hyperbolicgenset} Let $\G$ be a hyperbolic groupoid. Let
$\X$ be a compact topological transversal, and let $\X_0\subset\X$
be an open transversal. Choose a metric $|\cdot|$ on a neighborhood
$\wh\X$ of $\X$, satisfying the conditions of
Definition~\ref{def:hyperbolic}, and a Busemann $\eta$-quasi-cocycle
$\coc:\G|_{\wh\X}\arr\R$  .

Then there exist a compact generating set $S$ of $\G|_{\X}$, such that
\begin{enumerate}
\item for every $g\in S$ we have $\coc(g)>3\eta$;
\item $\be(S)=\en(S)=\X$, and for every $x\in  \X$ there
exists $g\in S$ such that $\en(g)=x$ and $\be(g)\in\X_0$;
\item there exists $\lambda\in(0, 1)$ such that for every
$g\in S$ there exists an element $U\in\pG|_{\wh\X}$ containing $g$ and
such that $|U(x)-U(y)|<\lambda |x-y|$ for all $x, y\in\be(U)$;
\item every element $g\in\G|_{\X}$ is equal to a
product of the form $g_n\cdots g_1\cdot (h_m\cdots h_1)^{-1}$ for
some $g_i, h_i\in S$.
\end{enumerate}
\end{proposition}

\begin{defi}
\label{def:complete}
We say that a generating set $S$ is \emph{complete} if it satisfies
the conditions of
Proposition~\ref{pr:hyperbolicgenset}. \index{complete generating set}
\index{generating set!complete}
\end{defi}

We will denote paths in the Cayley graph $\G(x, S)$ as sequences
$(g_1, g_2, \ldots, g_n)$ of elements of $S$ corresponding to its
edges, when the initial vertex of the path is clear from the
context. If $h$ is the initial vertex of the path, then it passes
through the vertices $h, g_1h, g_2g_1h, \ldots, g_n\cdots g_2g_1h$.

\begin{proof}
Let $(S, \X)$ be a generating pair satisfying the
conditions of Definition~\ref{def:hyperbolic}. 

All infinite oriented paths in
$\G(x, S^{-1})$ are quasi-geodesics converging to the point $\omega$
such that $\coc$ is a generalized Busemann
quasi-cocycle associated with $\omega$. Every such quasi-geodesic path is
on a bounded distance from a geodesic path converging to $\omega$.
It follows that there exist constants $C>1$ and $D>0$ such that $\coc(s_1s_2\ldots s_n)\ge
C^{-1}n-D$ for all $n$ and all composable sequences $s_i\in S$. In
particular, replacing $S$ by $S^n\cup S^{n+1}$, we may assume that
$\coc(g)>3\eta$ for all $g\in S$.

Since $\en(S)=\X$, for every $h\in\G|_{\X}$ there exist
infinite paths $\gamma'=(\ldots, g_n', g_2', g_1')$ and
$\gamma''=(\ldots, g_n'', g_2'', g_1'')$ in $\G(\be(h),
S)$ such that $\en(g_1')=\be(h)$ and $\en(g_1'')=\en(h)$. The
paths $\gamma'$ and $\gamma''$ are quasi-geodesics converging to
$\omega$. By hyperbolicity of $\G(\be(h), S)$ (see
Theorem~\ref{th:contraction}), there exists a
constant $\Delta$ depending only on $\G(\be(h), S)$ such that some
truncated paths
\[(\ldots, g_{n+2}', g_{n+1}', g_n'),\qquad(\ldots, g_{m+2}'',
g_{m+1}'', g_m'')\]
are on distance at most $\Delta$ from each other.

It follows that there exists a compact set $\nuke_0\subset\G|_{\X}$
such that every element of $\G|_{\X}$ is equal to the product $g_n\cdots
g_1\cdot s\cdot (h_m\cdots h_1)^{-1}$ for some $g_i, h_i\in S$
and $s\in\nuke_0$.

There exists $n_0$ such that $S\cup S^{n_0}\nuke_0$ is a generating
set satisfying the conditions of
Definition~\ref{def:hyperbolic}. Then every element $h$ can be
written as a product $g_n\cdots g_1\cdot (h_1\cdots h_m)^{-1}$.

Since the set $\X_0\cap\G_x$ is a net in the Cayley graph $\G(x,
S)$, there exists a compact set $A\subset\G|_{\X}$
such that for every $y\in\X$ there exists $g\in A$ such
that $\be(g)\in\X_0$ and $\en(g)=y$. Replacing $S$ by
$S\cup AS^{n_1}$ for some $n_1$, we will get a generating set
satisfying condition (2) of the
proposition (we use the fact that $\be(S)=\en(S)=\X$). The
remaining conditions of our proposition follow directly from the
conditions of Definition~\ref{def:hyperbolic}.
\end{proof}

\begin{proposition}
\label{pr:geodreduction} Suppose that $S$ is complete. There exists
a constant $\Delta_1$ such that every geodesic path $\gamma$ in the
Cayley graph is $\Delta_1$-close to a path of the form $(g_n, \ldots, g_1,
h_1^{-1}, \ldots, h_m^{-1})$ for $g_i, h_i\in S$.
\end{proposition}

\begin{proof}
The path $(g_n, \ldots, g_1, h_1^{-1}, \ldots, h_m^{-1})$ is the union of
two sides of a quasi-geodesic triangle. For every constant $\delta>0$
every pair of vertices of the Cayley graph on distance less than
$\delta$ from each other can be connected by a path of the form $(f_k,
\ldots, f_1, s_1^{-1}, \ldots, s_l^{-1})$ for uniformly bounded $k+l$.
Therefore, the statement of
the proposition follows from Theorem~\ref{th:rigidquasigeodesics} and
Proposition~\ref{prop:ushaped}.
\end{proof}

\section{Uniqueness of the point $\omega_x$}

\begin{proposition}
\label{pr:uniquegrading}
Let $\G$ be a hyperbolic groupoid, and let $x\in\G^{(0)}$.
Then the point $\omega_x$ is uniquely determined by $x$ and by the
topological groupoid $\G$.
\end{proposition}

\begin{proof}
Let $(S, \X)$ be a complete compact
generating pair. Note that for any two generating pairs $(S_i, Y_i)$
we can find one generating pair $(S, \X)$ such that $S_1\cup
S_2\subset S$ and $Y_1\cup Y_2\subset\X$.

Let $\X_2$ be a compact neighborhood of $\X$. Let
$\xi\in\partial\G(x, S)$ be an arbitrary point different from
$\omega_x$. If $(\ldots, g_2, g_1)$ is a geodesic path in $\G(x,
S)$ covering to $\xi$, then it is on a finite distance from a path
$(\ldots, h_2, h_1)$ in the
Cayley graph $\G(x, S)$ where $h_i\in S$. It follows that
there exist neighborhoods $U_i\in\pG$ of the elements $g_i$ such that
for every $n$ the composition $F_n=U_n\cdots U_2U_1$ is defined on
$\be(U_1)$ and $|F_n(y_1)-F_n(y_2)|\to 0$ as $n\to\infty$ for all
$y_1, y_2\in\be(U_1)$. Note
that since $\X_2$ is compact, the last condition is purely topological
and does not depend on the choice of the metric $|\cdot|$.

On the other hand, if $(\ldots, g_2, g_1)$ converges to
$\omega_x$, then it is on a bounded distance in the Cayley graph
$\G(x, S)$ to a path of the form $(\ldots, h_2^{-1}, h_1^{-1})$,
where $h_i\in S$. For any neighborhoods $W_i\in\pG$ of the elements
$h_i^{-1}$ the map $W_n\cdots W_2W_1$ is $\lambda^n$-expanding on a
neighborhood of $x$. It follows that for any neighborhoods $U_i\in\pG$
of the elements $g_i$ and any $y_1, y_2\in\be(U_1)$ such that $y_1\ne
y_2$ the sequence $|F_n(y_1)-F_n(y_2)|$ does not converge to $0$, if
it is defined for all $n$.

We see that the point $\omega_x$ is uniquely determined in
$\partial\G(x, S)$ by a condition that uses only purely topological
properties of $\G$ (and does not use the grading and the metric).
\end{proof}

\begin{corollary}
Let $(S, X)$ be a compact generating pair of a hyperbolic groupoid
$\G$ satisfying the conditions of Definition~\ref{def:hyperbolic}.
A quasi-cocycle $\coc:\G|_X\arr\R$ is a Busemann quasi-cocycle
if and only if there exist constants $C>1$ and $D>0$ such that
\[C^{-1}n-D\le\coc(s_1s_2\ldots s_n)\le Cn+D\]
for all composable sequences $s_1, s_2, \ldots, s_n$ of elements of $S$.
\end{corollary}

\begin{proof}
Every geodesic path in the Cayley graph $\G(x, S)$ converging to
$\omega_x$ is on a uniformly bounded distance from a directed path in
$\G(x, S^{-1})$. The statement of the corollary follows then from the
definition of a generalized Busemann cocycle, see
Definition~\ref{def:genbusemann} and the fact that directed paths in
$\G(x, S^{-1})$ are quasi-geodesics.
\end{proof}

\begin{defi}
A \emph{graded hyperbolic groupoid} is a hyperbolic groupoid together
with a choice of a strong equivalence class of a Busemann
quasi-cocycle. \index{graded hyperbolic groupoid} \index{groupoid!hyperbolic, graded}
\end{defi}

Note that a strong equivalence class of a Busemann
quasi-cocycle determines a Lipschitz class of the natural log-scale on
the boundaries $\partial\G_x$. See more on the natural metrics and
log-scales on $\partial\G_x$ in~\cite{nek:psmeasure}.

\section{Boundaries of the  Cayley graphs}

Let $\G$ be a hyperbolic groupoid, and let $\X$,
$\X_0\subset\X$, $\nu$, and $S$ satisfy the conditions of
Proposition~\ref{pr:hyperbolicgenset}.

Denote by $\partial\G_x$ \index{boundary!dGx@$\partial\G_x$}
\index{dGx@$\partial\G_x$} for $x\in\X$ the boundary of the
hyperbolic graph $\G(x, S)$ minus the point $\omega_x$. Since
the quasi-isometry type of the Cayley graph $\G(x, S)$ and the point $\omega_x$
do not depend on the choice of the generating pair $(S, \X)$, the
boundary $\partial\G_x$ does not depend on $(S, \X)$ and
is well defined for all $x\in\G^{(0)}$.

\begin{examp}
\label{ex:expandingboundary}
Let $\mathfrak{F}$ be the groupoid from
Example~\ref{ex:expanding}. Then $\partial\G_x$ is the boundary of the
tree of the grand orbit of $x$ without the limit of the forward orbit
of $x$.

\begin{figure}
\includegraphics{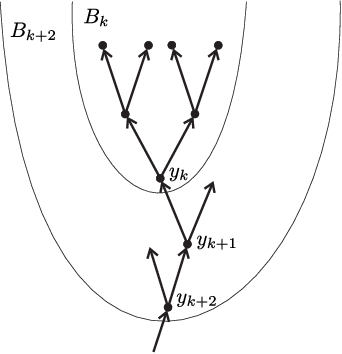}
\caption{Structure of $\partial\mathfrak{F}_x$} \label{fig:backtree}
\end{figure}

We adopt the following definition.

\begin{defi}
\label{def:treepreimages}
For a map $F:\X\arr\X$ and a point $t\in\X$ the \emph{tree of
  preimages} \index{tree of preimages} $T_{(F, t)}$ of the point $t$
is the tree with the set of vertices equal
to the disjoint union $\bigsqcup_{n\ge 0}F^{-n}(t)$ in which two
vertices $v\in F^{-n}(t)$ and $u\in F^{-(n-1)}(t)$ are connected by an
edge if $f(v)=u$.
\end{defi}

Every point $\xi\in\partial\G_x$ is the limit of a
path of germs of the form \[g_n=(f^n, x_n)^{-1}(f^k, x),\] where $k$ is
fixed, $x_n$ is a sequence of points such that
$f(x_{n+1})=x_n$ for all $n$, and $f^n(x_n)=f^k(x)$.

For a fixed value of $k$ the set of limits of the sequences $g_n$ is
naturally identified with the boundary of the tree of preimages
$T_{(f, y_k)}$, where $y_k=f^k(x)$.
The boundary $B_k=\partial T_{(f, y_k)}$ of the tree of preimages of
$y_k=f^k(x)$ is naturally a subset of the boundary $B_{k+1}$ of the
tree of preimages of $y_{k+1}=f^{k+1}(x)$. The boundary $\partial\G_x$ is
the inductive limit of the boundaries $B_k$. See Figure~\ref{fig:backtree}.
\end{examp}

\begin{examp}
Let $\G$ be the groupoid from Example~\ref{ex:dyadic}. Let
$x\in\Z_2=\G^{(0)}$. Then the boundary $\partial\G_x$ can be naturally
identified with the set of formal series
\[\sum_{n=-\infty}^{\infty}a_n2^n\]
where $a_n\in\{0, 1\}$ and $x-\sum_{n=0}^\infty a_n2^n\in\Z$. Here two
series $\sum_{n=-\infty}^\infty a_n2^n$ and $\sum_{n=-\infty}^\infty
b_n2^n$ are considered to represent the same point of the boundary if
both differences
\[\sum_{n=-\infty}^{-1}a_n2^n-\sum_{n-\infty}^{-1}b_n2^n,\qquad
\sum_{n=0}^\infty b_n2^n-\sum_{n=0}^\infty a_n2^n\]
belong to $\Z$ and are equal. Here the first difference is a
difference of real numbers, while the second difference is a difference
of elements of $\Z_2$.
\end{examp}

Denote by $\partial\G$ \index{dG@$\partial\G$} \index{boundary!dG@$\partial\G$}
the disjoint union (as a set) of the
spaces $\partial\G_x$ for all
$x\in\G^{(0)}$. We will introduce a natural topology on $\partial\G$ later.

The following proposition follows from
Proposition~\ref{prop:ascendingconverge}.

\begin{proposition}
For all $h\in\G_x$ and $g_i\in S$, $i\ge 1$, such that $\en(h)=\be(g_1)$
and $\en(g_i)=\be(g_{i+1})$ the sequence $g_n\cdots g_1\cdot h$
converges to a point of $\partial\G_x$.

For every point $\xi\in\partial\G_x$ there exist $h\in\G_x$ and a
sequence $g_i\in S$ such that
\[\xi=\lim_{n\to\infty} g_n\cdots g_1\cdot h.\]
\end{proposition}

We will use the following notation.
\[\cdots g_2g_1\cdot g=\lim_{n\to\infty} g_n\cdots g_2g_1\cdot
g\in\partial\G_{\be(g)}.\]
Similarly, for $U_i\subset\pG$ and $g\in\G$, we denote
\[\cdots U_2U_1\cdot g=\lim_{n\to\infty} U_n\cdots U_2U_1\cdot g.\]

Recall that a log-scale on $\partial\G_x$ is defined by
the function
\[\ell_{\omega_x}(\xi_1, \xi_2)\doteq
\lim_{g\to\omega_x, g_1\to\xi_2, g_2\to\xi_2}\ell_g(g_1, g_2)-|x-g|,\]
where the limit is
taken with respect to any convergent sub-sequence
(see~\ref{sss:busemann}), $\ell_g$ denotes the Gromov product in
the Cayley graph (see~\ref{sss:hyperbolicdef}), and $|x-g|$ is
combinatorial distance between vertices of the Cayley graph.

The log-scale $\ell_{\omega_x}(\xi_1, \xi_2)$ is H\"older
equivalent to the minimum value of $\coc$ along a geodesic
connecting $\xi_1$ to $\xi_2$ in the Cayley graph $\G(x, S)$, by
Corollary~\ref{cor:minell}.

We get the following fact from Proposition~\ref{pr:geodreduction} and
Theorem~\ref{th:rigidquasigeodesics}.

\begin{proposition}
\label{pr:topology} Denote by $\ell(\xi_1, \xi_2)$ the maximum
value of $n$ for which there exist representations
\[\xi_1=\cdots g_n\cdots g_1\cdot g,\qquad
\xi_2=\cdots h_n\cdots h_1\cdot g,\] where $g_i, h_i\in S$,
$g\in\G_x^{\X}$, and $\coc(g)\ge n$.

Then the function $\ell$ is a log-scale H\"older
equivalent to $\ell_{\omega_x}(\xi_1, \xi_2)$. The Lipschitz class of
$\ell$ does not depend on the choice of the generating set $S$ and the
transversal $\X$.
\end{proposition}

We call the log-scale $\ell$
defined in Proposition~\ref{pr:topology} (more precisely, its Lipschitz class)
the \emph{natural log-scale} \index{natural log-scale}
\index{log-scale!natural}
on $\partial\G_x$ defined by the quasi-cocycle $\coc$.

Denote by $\til_g$ the set of points of $\partial\G_{\be(g)}$ of the
form $\lim_{n\to\infty} g_n\cdots g_1\cdot g$ for $g_i\in S$. In
particular, $\til_x$ for $x\in\G^{(0)}$ is the set of points of
$\partial\G_x$ of the form $\lim_{n\to\infty} g_n\cdots g_1$, where
$g_i\in S$ and $\be(g_1)=x$.

\begin{proposition}
The sets $\til_g$ are compact. \index{Tg@$\til_g$}
\end{proposition}

\begin{proof}
For every $x\in\X$ the set $\Sigma_x$ of sequences $(g_1, g_2,
\ldots)\in S^\infty$ such that $\be(g_1)=x$ and
$\en(g_n)=\be(g_{n+1})$ is closed and hence compact in the product
topology of $S^\infty$. It follows from
Proposition~\ref{pr:topology} that for every $g\in\G|_{\X}$ the
map
\[\Sigma_{\en(g)}\arr\til_g:(g_1, g_2, \ldots)\mapsto \ldots g_2g_1\cdot
g\] is continuous. It is surjective by definition. Consequently,
$\til_g$ is a continuous image of a compact set.
\end{proof}

\begin{proposition}
\label{pr:nbhdtil} There exists a compact set
$A\subset\G|_{\X}$ such that for every $h\in\G|_{\X}$
there exists $a\in A$ such that $\til_{ah}$ is a
neighborhood of $\til_h$ and $\en(a)\in\X_0$.
\end{proposition}

\begin{proof}
It is enough to prove the proposition for the cases when $h$ is a unit
$x\in\X$. By Proposition~\ref{pr:topology}, a neighborhood of a point
$\xi\in\til_x$ is
the set of points $\zeta\in\partial\G_x$ such that
$\xi=\cdots g_2g_1\cdot g$ and $\zeta=\cdots
h_2h_1\cdot g$ for some $g\in\G_x^{\X}$ such that
$\coc(g)\ge 0$ and $g_i, h_i\in S$.

Since every vertex of $\G(x, S)$ is the end of a path
$(\ldots, f_2, f_1)$ of edges corresponding to elements of $S$, and
$\coc$ is bounded on $S$, there exists $\Delta>0$ such that every
$g\in\G_x^{\X}$ such that $\coc(g)\ge 0$
can be represented in the form $g=f_k\cdots f_1\cdot g'$ for
$f_i\in S$ and $0\le\coc(g')\le\Delta$.

Consequently, a neighborhood of $\xi$ is the union of the sets
$\bigcup_{g\in B_\xi}\til_g$, where $B_\xi$ is the set of elements
$g\in\G_x^{\X}$ such that $0\le\coc(g)\le\Delta$ and $\xi\in\til_g$.

Let $g\in B_\xi$, so that $\xi=\ldots s_2s_1\cdot g$ for some $s_i\in
S$. We know that $\xi=\ldots g_2g_1$ for some $g_i\in S$, since
$\xi\in\til_x$. It follows from Theorem~\ref{th:contraction} that
there exists a uniform constant $\Delta_1$ (not depending on $x$ and
$\xi$) such that the paths $\cdots g_2g_1$ and $\cdots s_2s_1\cdot g$ are
at most $\Delta_1$ apart. Consequently, there is a uniform upper
bound on the length of elements of $B_\xi$, hence there exists a
compact set $B$ such that $B_\xi\subset B$ for all $\xi$.
Then $\bigcup_{a\in B\cap\G_x}\til_a$ is a neighborhood of $\til_x$.

Let $\{a_1, a_2, \ldots, a_n\}\subset\G_x^{\X}$ be an arbitrary set.
Let us prove by induction on
$n$ that there exists a product $g=h_m\cdots h_1\in\G^x$ of
elements of $S$ such that for every $a_i$ the element
$a_ig$ can be written as a product $f_k\cdots f_1$ of elements of
$S$ and $m$ is bounded from above by a function of
$n$ and the lengths of representations of $a_i$ as products of
elements of $S$ and $S^{-1}$. It is true for $n=1$, by condition (4) of
Proposition~\ref{pr:hyperbolicgenset} and by Proposition~\ref{pr:geodreduction}.
If there is such an element
$g'$ for $\{a_1, a_2, \ldots, a_{n-1}\}$, then, again by
Propositions~\ref{pr:hyperbolicgenset} and~\ref{pr:geodreduction},
there exists a product $t_l\cdots t_1$ of elements of $S$ such that
$a_ng't_l\cdots t_1$ is equal to a product of elements of $S$. Then we can take
$g=g't_l\ldots t_1$ for the set $\{a_1, \ldots, a_n\}$, since then every element
$a_ig=a_ig't_l\ldots t_1$ is equal to a product of elements of
$S$.

The size of the sets $B\cap\G_x$ and the lengths of elements of $B$
are uniformly bounded. Consequently, there exists a compact set $A'$
such that for every $x\in\X$ there exists an element $a\in A'$ that
can be written as a product of elements of $S$, and is such that $ba$
is as a product of elements of $S$ for all $b\in B\cap\G_x$.
Then $\til_{\be(a)}\supset\til_{ba}$, hence
$\til_{a^{-1}}\supset\til_b$ for all $b\in B\cap\G_x$. It follows that
$\til_{a^{-1}}$ is a neighborhood of $\til_x$, hence we can take
$A=(A')^{-1}$. We can assume that $\en(a^{-1})\in\X_0$ by condition
(2) of Proposition~\ref{pr:hyperbolicgenset}.
\end{proof}

\begin{proposition}
\label{pr:internalshift}
If $\xi=\ldots g_2g_1$ for $g_i\in S$ belongs to the interior of
$\til_x$, then for every compact set $\nuke$ there exists $n$ such
that every point of the form $\ldots h_2h_1\cdot f\cdot g_n\cdots
g_1$, where $h_i, g_i\in S$ and $f\in\nuke$,
belongs to the interior of $\til_x$.
\end{proposition}

\begin{proof}
Let $\zeta=\ldots h_2h_1\cdot f\cdot g_n\cdots
g_1\in\partial\G_x$. 
Then $\ldots h_2h_1\cdot f\cdot
g_{n+1}^{-1}g_{n+2}^{-1}\cdots$ is a path connecting $\zeta$ to $\xi$,
where the segments $\ldots h_2h_1\cdot f$ and
$g_{n+1}^{-1}g_{n+2}^{-1}\cdots$ are uniformly quasi-geodesic.
It follows that for some uniform constant $\Delta>0$ we
have $\ell(\zeta, \xi)\ge \coc(g_n\cdots g_1)-\Delta$. Since $\xi$ is an
internal point of $\til_x$, for $n$ big enough, the point $\zeta$ will
belong to the interior of $\til_x$.
\end{proof}

\section{The boundary $\partial\G$}

Recall that $\partial\G$ is the disjoint union of the sets
$\partial\G_x$ for $x\in\G^{(0)}$. For $\xi\in\partial\G$ define $P(\xi)\in\G^{(0)}$ by the condition
$\xi\in\partial\G_{P(\xi)}$.

\begin{theorem}
\label{th:localproduct} Let $\G$ be a hyperbolic groupoid.
There exist unique topology and local product
structure on $\partial\G$ \index{dG@$\partial\G$}
\index{boundary!dG@$\partial\G$}
satisfying the following condition.

Let $S$ be a complete generating set of $\G$, and let $\mS$ be a finite
covering of $S$ by contracting positive elements of $\pG$ (see
Definition~\ref{def:positive}).
Then there exist a covering $\mathcal{R}$ of $\partial\G$ by open
rectangles and elements $U_R\in\pG$ for every
$R\in\mathcal{R}$ such that $\be(U_R)=P(R)$, and for
every $\xi\in R$ there exists a sequence $F_i\in\mS$ such
that $\be(F_n\cdots F_1U_R)=\be(U_R)$ for all $n$ and
\[[\zeta, \xi]_R=\ldots F_2F_1(U_R, P(\zeta))\]
for all $\zeta\in R$.
\end{theorem}

\begin{figure}
\centering
\includegraphics{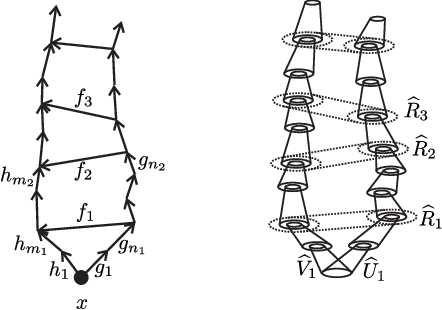}
\caption{} \label{fig:ladder}
\end{figure}

\begin{examp}
Let $\mathfrak{F}$ be as in Example~\ref{ex:expanding}.
Let $U$ be an arbitrary open arc of the circle
$\R/\Z=\mathfrak{F}^{(0)}$  (not coinciding with the whole circle).
Then $U$ is evenly covered by every iteration $f^n$ of $f$.

For every $x\in U\subset\R/\Z$ the boundary $B_x$ of
the tree of preimages $T_{(f, x)}$ of $x$ is an open
subset of $\partial\mathfrak{F}_x$ (see
Example~\ref{ex:expandingboundary}). For every $\xi\in B_x$ there
exists a unique sequence $U_0=U, U_1, U_2, \ldots$ of arcs of $\R/\Z$
such that $f(U_{n+1})=U_n$ and $\xi$ is the limit of the germs at $x$
of the local homeomorphisms $f^{-n}:U\arr U_n$. Then for every $y\in
U$, we get the corresponding point of $T_y$ equal to the limit of the
germs at $y$ of the same maps $f^{-n}:U\arr U_n$. We see that there is
a natural identification of the union of the boundaries $B_y$ of the
tree of preimages of points of $U$ with the direct product $U\times T_x$.
\end{examp}

\begin{proof}
Let $\X$ be a compact topological $\G$-transversal, and let $\wh\X$ be
its open neighborhood. Fix a metric $|\cdot|$ on $\wh\X$ and a
$\eta$-quasi-cocycle $\coc:\G|_{\wh\X}\arr\R$ satisfying the conditions of
Definition~\ref{def:hyperbolic}.

Let us start with some technical lemmas.

\begin{lemma}
\label{l:relationsboundary} Let $S\subset\G|_{\wh\X}$ be a compact
subset such that $\coc(g)>2\eta$ for every $g\in S$.

Then there exists a compact set $\nuke\subset\G|_{\wh\X}$ and a number
$\Delta>0$ such that for all sequences $g_n, h_n\in S$
such that $\be(g_1)=\be(h_1)$,
$\be(g_{n+1})=\en(g_n)$, $\be(h_{n+1})=\en(h_n)$ the following
conditions are equivalent
\begin{enumerate}
\item the sequences
$g_n\cdot g_{n-1}\cdots g_1$ and $h_n\cdot h_{n-1}\cdots
h_1$ converge to the same point of $\partial\G_x$, where $x=\be(g_1)=\be(h_1)$;
\item there exist  a sequence
$f_k\in\nuke$, and strictly increasing
sequences $n_k$ and $m_k$ such that $n_k-n_{k-1}\le\Delta$ and
$m_k-m_{k-1}\le\Delta$ for all $k\ge 1$,
\[f_k\cdot g_{n_k}\cdot g_{n_k-1}\cdots g_{n_{k-1}+1}=h_{m_k}\cdot
h_{m_k-1}\cdots h_{m_{k-1}+1}\cdot f_{k-1}\] for all $k\ge 1$, and
$f_0=\be(g_1)$.
\end{enumerate}
\end{lemma}

See the left-hand side of Figure~\ref{fig:ladder} for an illustration
of the lemma.

\begin{proof}
Embed $S$ into a compact generating set $S_1$ of $\G|_{\X'}$ for some
compact topological transversal $\X'\subset\wh\X$. There exists $n$
such that the set $S_1^n\cup S_1^{n+1}$ can be embedded into a
complete generating set $S_2$ (see the proof of
Proposition~\ref{pr:hyperbolicgenset}).
It is enough to prove our proposition for any compact set containing
$S^n$. Therefore, we assume that $S$ satisfies the conditions of
Proposition~\ref{pr:hyperbolicgenset} for a compact topological
transversal $\X'$.

The quasi-geodesics $(\ldots, g_n, \ldots, g_1)$ and
$(\ldots, h_n, \ldots, h_1)$ converge to the same point of
$\partial\G_x$ if and only if they are on a finite distance from
each other. If $\Delta$ is as in condition (2) and $r$ is the maximal
length of elements of $\nuke$, then every point of one of the
paths $(\ldots, g_2, g_1)$ and $(\ldots, h_2, h_1)$ is on distance not more
than $\Delta+r$ from the other, hence the quasi-geodesics $(\ldots,
g_n, \ldots, g_1)$ and $(\ldots, h_n, \ldots,
h_1)$ converge to the same point of the boundary.

In the other direction, there exists $\Delta_1$ such that
if quasi-geodesic paths $(\ldots, g_2, g_1)$ and $(\ldots, h_2, h_1)$,
$g_i, h_i\in S$, starting at a common vertex converge to the
same point of the boundary, then they are on
a distance not more than $\Delta_1$ from each other.

Let $\nuke$ be the set of elements of $\G$ which can be
represented as products of not more than $\Delta_1$ elements of $S\cup
S^{-1}$. Then $\nuke$ is a compact set.

Define $n_0=m_0=0$.
By induction, if $m_k$ and $n_k$ are defined, define
$n_{k+1}$ to be such that distance from $(g_{n_k}, \ldots, g_1)$
to $(g_{n_{k+1}}, \ldots, g_1)$ in $\G(x, S)$ is greater than
$2\Delta_1$, but less than some fixed constant $c$ (which we can find since
$(\ldots, g_2, g_1)$ is quasi-geodesic). Then define $m_{k+1}$ so that
distance from $(g_{n_{k+1}}, \ldots, g_1)$ to
$(h_{m_{k+1}}, \ldots, h_1)$ is not more than $\Delta_1$. Then
$m_{k+1}>m_k$ and the differences $m_{k+1}-m_k$ and $n_{k+1}-n_k$
are uniformly bounded.
\end{proof}

\begin{lemma}
\label{lem:relationsbndm} Let
$\mS$ and $\mathcal{B}$ be finite sets of elements of $\pG|_{\wh\X}$ such that
all elements of $\mS$ are positive contractions, and all
elements of $\mathcal{B}$ are bi-Lipschitz.
Let $\epsilon>0$.
There exists a number $\delta_0>0$ for which the following
statements hold.

Suppose that the sequences $U_n, V_n\in\mS\cup\mathcal{B}$
and $g_n, h_n\in\G$ are
such that each sequence $(U_n)$, $(V_n)$ contains at most
one element of $\mathcal{B}$, $g_n$ and $h_n$ are $\epsilon$-contained
in $U_n$, $V_n$ respectively, $\be(g_0)=\be(h_0)$,
and the sequences $g_ng_{n-1}\cdots g_0$ and $h_nh_{n-1}\cdots
h_0$ are defined and converge to the same point of
$\partial\G_{\be(g_0)}$.

Then for every $y\in\wh\X$ such that $|y-\be(g_0)|<\delta_0$ the
sequences
\[U_n\cdots U_{n-1}\cdots(U_0, y),\quad\text{and}\quad
V_n\cdots V_{n-1}\cdots (V_0, y)\]
are defined and converge to the same point.
\end{lemma}

See the right-hand side of Figure~\ref{fig:ladder}.

\begin{proof} There exists $k$ such that for every $F\in\mathcal{B}$ and
for every sequence $F_1, \ldots, F_k$ of elements of $\mathcal{S}$ the map
$F_1\cdots F_k\cdot F$ is either empty or is a contraction such
that $\coc(g)>2\eta$ for every its germ $g$. Consequently, it
is enough to prove our lemma for the case when $\mathcal{B}$ is empty.

Let $S$ be the closure of the
sets of elements $g\in\G$ such that there exists $U\in\mS$
such that $g$ is  $\epsilon$-contained in $U$.

Let $\nuke$ and $\Delta$ be as in Lemma~\ref{l:relationsboundary} for
the set $S$. The
set $\nuke$ can be covered by a finite set of relatively compact
extendable elements of $\pG|_{\wh\X}$. Therefore, there exists a finite
collection $\mathcal{R}$ of elements $\wh R_i\in\pG|_{\wh\X}$ and open subsets
$R_i\subset\wh R_i$ such that $R_i$ cover $\nuke$, $\overline R_i$
are compact and are contained in $\wh R_i$.

Let us apply Corollary~\ref{cor:relationspseudogroup} to the set
$\mS\cup\mathcal{R}$, and some extensions $\wh F$ of its elements.

The number of possible sequences of the form \[(R_k, G_{n_k}, \ldots,
G_{n_{k-1}+1}, H_{m_k}, \ldots, H_{m_{k-1}+1}, R_{k-1}),\]
where $n_k-n_{k-1}, m_k-m_{k-1}\le\Delta$, $G_i, H_i\in\mS$, and
$R_i\in\mathcal{R}$
is finite. Consequently, by Corollary~\ref{cor:relationspseudogroup},
there exists $\delta_0>0$ such that for every relation of the form
\[f_k\cdot g_{n_k}\cdots g_{n_{k-1}+1}=h_{m_k}\cdots
h_{m_{k-1}+1}\cdot f_{k-1}\] for $f_k\in R_k, f_{k-1}\in R_{k-1}$
and $g_i\in G_i, h_i\in H_i$, for $R_i\in\mathcal{R}$, $G_i, H_i\in\mS$
the maps
\[\wh R_k\cdot\wh G_{n_k}\cdots\wh G_{n_{k-1}+1},\qquad\wh H_{m_k}\cdots\wh
H_{m_{k-1}+1}\cdot\wh R_{k-1}\] are equal on the
$\delta_0$-neighborhood of $\be(g_{n_{k-1}+1})$.

Taking $\delta_0$ sufficiently small (in particular,
$\delta_0<\epsilon$), we may assume that $\be(\wh R_k)$,
$\en(\wh R_k)$, $\be(\wh G_i)$, $\be(\wh H_i)$, $\be(\wh R_{k-1})$, $\en(\wh
R_{k-1})$ always contain the $\delta_0$-neighborhoods
of $\be(f_k)$, $\en(f_k)$, $\be(g_i)$, $\be(h_i)$, $\be(f_{k-1})$,
$\en(f_{k-1})$ respectively. Since the maps $G_i$ and $H_i$ are
contractions, this will imply that the maps
\[\wh R_k\cdot\wh G_{n_k}\cdots\wh G_{n_{k-1}+1},\qquad\wh H_{m_k}\cdots
\wh H_{m_{k-1}+1}\cdot\wh R_{k-1}\] are defined on the
$\delta_0$-neighborhood of $\en(g_{n_{k-1}})=\be(g_{n_{k-1}+1})$.
This implies the statement of the lemma (since $\wh F$ and $F$ for
$F\in\mS$ coincide on the $\epsilon$-neighborhood of $\be(g)$ for all
$g$ that are $\epsilon$-contained in $F$).
\end{proof}

We may assume that $\X$ contains open transversals $\X_0$ and $\X_0'$ such that
$\overline{\X_0}\subset\X_0'$.
Let $S$ be a generating set of $\G$ satisfying the conditions of Proposition~\ref{pr:hyperbolicgenset} for $\X_0$. Let $\mS$ be a finite
covering of $S$ by relatively compact extendable
elements of $\pG$. We assume that every $F\in\mS$ is a
$\lambda$-contraction and $\coc(g)>2\eta$ for every $g\in F$.

Let $A$ be a compact set satisfying the conditions of
Proposition~\ref{pr:nbhdtil}. Let $\Delta_2$ be such that any two
paths $(\ldots, g_2, g_1)$, $(\ldots, h_2, h_1)$, $g_i, h_i\in S$,
converging to one point of $\partial\G$ are eventually on distance not
more than $\Delta_2$ from each other. Let $Q$ be the set of elements
of $\G$ that can be represented as a product of length at most
$\Delta_2$ of elements of $\mS$ (in particular, $Q$ contains
$\X$). Denote $B=A\cdot Q$.

Find a finite covering $\mathcal{B}=\{U_i\}$ of
$B$ by open relatively compact extendable bi-Lipschitz
elements of $\pG|_{\wh\X}$. Denote $\wh
B=\bigcup_{U_i\in\mathcal{B}}U_i$.

Let $\epsilon>0$ be such that
$2\epsilon$ is a common Lebesgue's
number of the coverings $\mS$ and $\mathcal{B}$
of $S$ and $B$ respectively.

Let us apply Lemma~\ref{lem:relationsbndm} to $\mathcal{B}$,
$\mS$, and $\epsilon$. Let us fix a number
$\delta_0<\epsilon$ satisfying the conditions of the lemma. We assume that
the $\delta_0$-neighborhood of $\be(F)$ is contained in $\be(\wh F)$
for every $F\in\mS$, and that the
$\delta_0$-neighborhood of $\X_0$ is contained in $\X_0'$.

\begin{defi}
We say that $U\in\pG$ is \emph{admissible} if it is relatively
compact, extendable, and the closure of $\en(U)$ is a subset of
$\X_0'$ of diameter less than $\delta_0$. \index{admissible set}
\end{defi}

Let $U\in\pG$ be admissible and consider an arbitrary point
$\zeta=\cdots g_2g_1\cdot g\in\til_g$, where $g_i\in S$ and
$g\in U$. Let $F_i\in\mS$ be such that $g_i$ is $\epsilon$-contained
in $F_i$ for every $i$. Then the
map $F_n\cdots F_1U$ is defined on $\be(U)$.

Define then for $y\in\be(U)$
\begin{equation}
\label{eq:locprodU}
[y, \zeta]_U=\cdots F_2 F_1\cdot (U, y),
\end{equation}
where $(U, y)$, as usual, denotes the germ of $U$ at $y$.

\begin{lemma}
\label{lem:welldef} The map $[\cdot, \cdot]_U:\be(U)\times\til_g\arr\partial\G$ is well defined and injective for every $g\in U$.
\end{lemma}

\begin{proof}
Suppose that $\xi=\cdots g_2g_1\cdot g=\cdots h_2h_1\cdot g$ for
$g_i, h_i\in S$ and $g\in U$. Let $U_i$ and $V_i\in\mS$ be such
that $g_i$ and $h_i$ are $2\epsilon$-contained in $U_i$
and $V_i$ respectively. It follows then directly from Lemma~\ref{lem:relationsbndm} that  for every
$y\in\be(U)$ we have
\[\cdots U_2 U_1\cdot (U, y)=\cdots V_2 V_1\cdot (U, y),\]
which proves that the map $[\cdot, \cdot]_U$ is well defined.
Injectivity is proved in the same way.
\end{proof}

\begin{proposition}
\label{pr:injective} Let $U$ be admissible.
For fixed $y\in\be(U)$ and $g\in U$ the map
$[y, \cdot]_U:\til_g\arr\partial\G_y$ is Lipschitz with respect to
the natural log-scales on $\partial\G_{\be(g)}$ and $\partial\G_y$.
\end{proposition}

\begin{proof}
If $\xi_1, \xi_2\in\til_g$ are represented as $\xi_1=\cdots
g_{n+1}g_n\cdots g_1\cdot g$ and $\xi_2=\cdots h_{n+1}g_n\cdots
g_1\cdot g$ for some $n\ge 1$ and $g_i, h_i\in S$, then $[y, \xi_1]_U=
\cdots g_{n+1}'g_n'\cdots g_1'\cdot g'$ and $[y, \xi_2]_U=
\cdots h_{n+1}'g_n'\cdots g_1'\cdot g'$ for some germs $g_i'$
and $h_i'$ of elements of $\mS$. It follows from the definition of the
log-scale on $\partial\G_x$, Theorem~\ref{th:rigidquasigeodesics},
and Lemma~\ref{lem:quasigeodesic} that there
exists $\Delta_0$ (depending only on $(S, \X)$) such that
\[\ell_y([y, \xi_1]_U, [y, \xi_2]_U)\ge\ell_{\be(g)}(\xi_1,
\xi_2)-\Delta_0,\] where $\ell_y$ and $\ell_{\be(g)}$ are the
natural log-scales on $\partial\G_y$ and $\partial\G_{\be(g)}$
respectively. We used the fact that $\ell_y$ and $\ell_{\be(g)}$ do
not depend, up to Lipschitz equivalence, on the choice of the
generating set.
\end{proof}

Denote by $\til_g^\circ$ the interior of $\til_g\subset\partial\G_{\be(g)}$.
Denote for an admissible $U\in\pG$ and $g\in U$
\[R_{g, U}=\{[y, \xi]_U\;:\;\xi\in\til_g^\circ, y\in\be(U)\}\subset\partial\G.\]

By Lemma~\ref{lem:welldef} the map $(y, \xi)\mapsto [y, \xi]_U$ from
$\be(U)\times\til_g^\circ$ to $R_{g, U}$ is a bijection. We consider
$R_{g, U}$ as a rectangle with respect to the direct product
decomposition $R_{g, U}\approx\be(U)\times\til_g^\circ$. At this moment we do not have any topology on
$\partial\G$ and $R_{g, U}$ yet.

\begin{lemma}
\label{lem:loccompatible}
Let $\xi\in R_{g_1, U_1}\cap R_{g_2, U_2}$. Then there exist $g$ and
$U$ such that $R_{g, U}$ is defined, $\xi\in R_{g, U}\subset R_{g_1, U_1}\cap
R_{g_2, U_2}$ and the rectangle $R_{g, U}$ is compatible with $R_{g_1,
  U_1}$ and $R_{g_2, U_2}$.
\end{lemma}

\begin{proof}
Let $x=P(\xi)$. Denote $x_i=\en(g_i)$.
Since $\xi$ belongs to the intersection $R_{g_1, U_1}\cap R_{g_2,
  U_2}$, there exist points $\xi_1\in\til_{g_1}^\circ$ and
$\xi_2\in\til_{g_2}^\circ$ such that $\xi=[x, \xi_1]_{U_1}=[x,
\xi_2]_{U_2}$. Note that $x\in\be(U_1)\cap\be(U_2)$. Let $\xi_1=\cdots
s_2s_1\cdot g_1$ and $\xi_2=\cdots r_2r_1\cdot g_2$ where $s_i, r_i\in
S$. Let $S_i$ and $R_i$ be elements of $\mS$ such that $s_i$ and $r_i$
are $2\epsilon$-contained in $S_i$ and $R_i$ respectively. Then
\[\xi=\cdots S_2S_1(U_1, x)=\cdots R_2R_1(U_2, x).\]

Since the points $\xi_1$ and $\xi_2$ belong to the interior of the
sets $\til_{g_1}$ and $\til_{g_2}$, there exists $n_0$ such that for all
$k\ge n_0$, for all sequences $H_1, H_2, \ldots\in\mS$, and for all
$b\in\wh B$ we have
\[\cdots H_2H_1\cdot b\cdot S_kS_{k-1}\cdots S_1g_1\in\til_{g_1}^\circ, \quad
\cdots H_2H_1\cdot b\cdot R_kR_{k-1}\cdots R_1g_2\in\til_{g_2}^\circ,\]
whenever the corresponding products are defined, see
Proposition~\ref{pr:internalshift}.

There exist indices $k_1$ and $k_2$ both greater than $n_0$
such that \[q=(R_{k_2}\cdots R_1(U_2, x))\cdot(S_{k_1}\cdots
S_1(U_1, x))^{-1}\in Q.\]
There exists $f_1\in A$ such that $\be(f_1)=S_{k_1}\cdots S_1U_1(x)$,
$\en(f_1)\in\X_0$, and $\xi$ is an internal point of $\til_{f_1\cdot
  S_{k_1}\cdots S_1(U_1, x)}$, see Figure~\ref{fig:proof1}.
Then $f_2=f_1q$ belongs to $B$ and
\[(f_1S_{k_1}\cdots S_1U_1, x)=(f_2R_{k_2}\cdots R_1U_2, x).\]

\begin{figure}
\includegraphics{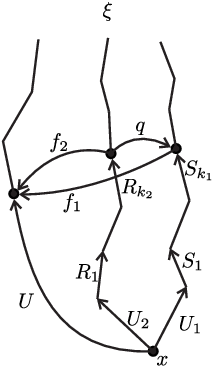}
\caption{}
\label{fig:proof1}
\end{figure}

Let $F_1, F_2\in\mathcal{B}$ such that
$f_1$ and $f_2$ are $2\epsilon$-contained in $F_1$ and $F_2$
respectively. Then the germs $(F_1S_{k_1}\cdots S_1U_1, x)$ and
$(F_2R_{k_2}\cdots R_1U_2, x)$ coincide. Let $U$ be a common
restriction of $F_1S_{k_1}\cdots S_1U_1$ and $F_2R_{k_2}\cdots R_1U_2$
to an open neighborhood of $x$. Note that
$\be(U)\subset\be(U_1)\cap\be(U_2)$.
We may assume that $U$ is admissible, since
$\en(f_1)=\en(f_2)\in\X_0$. Denote $g=(U, x)$, and consider the
rectangle $R_{g, U}$.
By the choice of $f_1$, the point $\xi$ belongs to the interior of $\til_g$,
hence $\xi\in R_{g, U}$.

It remains to show for every $i=1, 2$
that $R_{g, U}\subset R_{g_i, U_i}$ and that the direct product
structure on $R_{g, U}$ agrees with the direct product structure
on $R_{g_i, U_i}$. It is enough to prove the statements for $i=1$
(since we will use only that $f_i\in B$, and not the way they were
defined, so that both $i=1$ and $i=2$ will be equivalent).

In order to show that $R_{g, U}\subset R_{g_1, U_1}$, we have to
show that for every $y\in\be(U)$ and $\zeta\in\til_g^\circ$ we
have $[y, \zeta]_U\in R_{g_1, U_1}$, i.e., that there exists
$\zeta_1\in\til_{g_1}^\circ$ such that $[y, \zeta]_U=[y, \zeta_1]_{U_1}$.
Then, in order to show that the direct product structures of $R_{g, U}$
and $R_{g_1, U_1}$ agree, it will be enough to show that for
every $z\in\be(U)$ we have $[z, \zeta]_U=[z, \zeta_1]_{U_1}$.

Let $\zeta=\ldots h_2h_1\cdot g$, where
$h_i\in S$. Let $H_i\in\mS$ be such that $h_i$ is $2\epsilon$-contained
in $H_i$. Then
\[[y, \zeta]_U=\ldots H_2H_1(U, y)=\ldots H_2H_1F_1S_{k_1}\cdots
S_1(U_1, y).\]
By the choice of $k_1$, the point
\[\zeta_1=\ldots H_2H_1F_1S_{k_1}\cdots S_1\cdot g_1\] belongs to
$\til_{g_1}^\circ$. In particular, it can be represented in the form
$\ldots t_2t_1 g_1$ for $t_i\in S$. If $T_i\in\mS$ are such that
$t_i$ is $2\epsilon$-contained in $T_i$, then
\[[y, \zeta_1]_{U_1}=\ldots T_2T_1(U_1, y)=\ldots
H_2H_1F_1S_{k_1}\cdots S_1(U_1, y)=[y, \zeta]_U,\]
where the last equality follows by Lemma~\ref{lem:relationsbndm} from
the equality
\[\ldots T_2T_1(U_1, \be(g_1))=\ldots H_2H_1F_1S_{k_1}\cdots S_1(U_1,
\be(g_1)).\]

In fact, by Lemma~\ref{lem:relationsbndm}, we have
\[[z, \zeta_1]_{U_1}=\ldots T_2T_1(U_1, z)=\ldots
H_2H_1F_1S_{k_1}\cdots S_1(U_1, z)=[z, \zeta]_U,\]
for all $z\in\be(U)$, which finishes the proof.
\end{proof}

\begin{lemma}
\label{lem:topology}
The set of rectangles $R_{g, U}$ form a base of a topology on
$\partial\G$. The map $[\cdot, \cdot]_U:\be(U)\times\til_g^\circ\arr
R_{g, U}$ is a homeomorphism with respect to it.
\end{lemma}

\begin{proof}
The first statement follows directly from
Lemma~\ref{lem:loccompatible}. Let us prove the second statement. We
know that the map is a bijection (see
Lemma~\ref{lem:welldef}). Let $A\subset\be(U)$ and
$B\subset\til_g^\circ$ be open sets, and let $y\in A$, $\zeta\in B$ be
arbitrary points. Using Proposition~\ref{pr:nbhdtil} (as in the proof
of Lemma~\ref{lem:loccompatible}) we show that
there exist $g_1\in\G$ and admissible $U_1\in\pG$
such that $\til_{g_1}\subset B$ and $\zeta\in\til_{g_1}^\circ$,
$g_1\in U_1$, and $R_{g_1, U_1}\subset R_{g, U}$.

Then, applying Lemma~\ref{lem:loccompatible} to $R_{g_1, U_1}$ and
$\xi=[y, \zeta]_U$, we see that there exist
$U'\in\pG$ and $g'\in\G$ such that $[y, \zeta]_U\in R_{g', U'}\subset
[A, B]_U$. This shows that the image of an open subset of
$\be(U)\times\til_g^\circ$ is open, i.e., that the map inverse to
$[\cdot, \cdot]_U:\be(U)\times\til_g^\circ\arr R_{g, U}$ is continuous.

It remains to prove that the map $[\cdot, \cdot]_U$ is
continuous. Continuity on the first argument follows directly from the
definition, while continuity on the second argument follows from
Lemma~\ref{pr:injective}.
\end{proof}

\begin{lemma}
If $\delta_0$ is small enough, then
the rectangles $R_{g, U}$ form an atlas of a local product structure
on the topological space $\partial\G$.
\end{lemma}

\begin{proof}
Let $R_{g_1, U_1}$ and $R_{g_2, U_2}$ be two rectangles, and let
$\xi\in R_{g_1, U_1}\cup R_{g_2, U_2}$. We have to show that there
exists a rectangle $R_{g, U}$ containing $\xi$ and such that the direct
product structure on $R_{g, U}$ agrees with the direct product
structures on both rectangles $R_{g_i, U_i}$. Note that since the
rectangles form a basis of topology of $\partial\G$, the condition
obviously holds when $\xi$ does not belong to the intersection of the
closures of $R_{g_i, U_i}$.

Therefore, we suppose that $\xi\in\overline{R_{g_1,
    U_1}}\cap\overline{R_{g_2, U_2}}$.
We can embed $S$ into a bigger complete generating set
$S'$ so that for every $g\in\G$ the new set $\til_g'=\{\ldots
g_2g_1\cdot g\;:\;g_i\in S'\}$ is a neighborhood of $\til_g=\{\ldots
g_2g_1\;:\;g_i\in S\}$. Let $\mS'$ be the corresponding covering (we
assume that $\mS\subset\mS'$).
If $R_{h, W}$ is a rectangle defined using $S$, such that $\en(W)$ is
sufficiently small (i.e., has diameter less than the value of
$\delta_0$ corresponding to $S'$ and $\mS'$), then it is a sub-rectangle of the
rectangle $R_{h, W}'$ defined using $S'$ and $\mS'$, by
Lemma~\ref{lem:welldef}.

We can always extend the admissible elements $U_i$ to admissible
elements $\wh U_i$ such that $\overline{U_i}\subset\wh
U_i$.

It follows that if $\delta_0$ is small enough (i.e., if it satisfies
the conditions of Lemma~\ref{lem:relationsbndm} for $S'$ and $\mS'$),
then closure of every rectangle $R_{g_i, U_i}$ can be embedded into a
rectangle defined using $S'$. Then Lemma~\ref{lem:loccompatible}
(applied to $S'$ and $\mS'$) implies that rectangles $R_{g, U}$ form
an atlas of a local product structure on $\partial\G$.
\end{proof}

Independence of topology and the local product structures from the
choices of the Lipschitz structure, grading, and
the sets $S$, $\mS$, $\X$, $\X_0$, etc.,
follow directly from Lemma~\ref{lem:welldef}, since for any
two choices of the sets we can find one including the first two as
subsets (possibly after replacing $S$ and $\mS$ by $S^n\cup S^{n+1}$
and $\mS^n\cup\mS^{n+1}$, if we change the Lipschitz structures and the
grading). Independence on the choice of $\delta_0$ follows from
Lemma~\ref{lem:loccompatible}.

This finishes the proof of the theorem, since the constructed local
product and topology satisfy its conditions, and every local product
structure satisfying the conditions of the theorem coincides with the
one we have constructed (by Lemma~\ref{lem:welldef}).
\end{proof}

\section{Geodesic quasi-flow}
\label{ss:actionofgonpg}
Denote by $P:\partial\G\mapsto\G^{(0)}$ the map defined by the
condition $\xi\in\partial\G_{P(\xi)}$.

It is easy to see from the formula for the local product in
Theorem~\ref{th:localproduct} that the map $P$ is
compatible with the local product structure on $\partial\G$, i.e.,
for every point $\xi\in\partial\G$ there exists an open rectangle
$R$ containing $\xi$ such that $P:R\arr\G^{(0)}$ induces
homeomorphisms of the plaques $\proj_2(R, x)$ with $P(R)$.

Let $\xi=\cdots h_2h_1\cdot h\in\partial\G$ for $h_i\in S$ and $h\in\G$. Define
then for $g\in\G$ such that $\en(g)=\be(h)=P(\xi)$
\[\xi\cdot g=\cdots h_2h_1\cdot hg.\]
We get in this way a right action of $\G$ on $\partial\G$ over the
projection $P:\partial\G\arr\G^{(0)}$.

Consider the groupoid of this action, i.e., the set
\[\partial\G\rtimes\G=\{(\xi,
g)\in\partial\G\times\G\;:\;P(\xi)=\en(g)\}\] with topology
induced from the direct product topology on $\partial\G\times\G$ and
groupoid structure given by the multiplication rule
\[(\xi_2, g_2)\cdot (\xi_1, g_1)=(\xi_1, g_2g_1),\]
where the product is defined if and only if $\xi_2=\xi_1\cdot
g_1$ (see Definition~\ref{def:actiongroupoid}).
We naturally identify  the space of units of $\partial\G\rtimes\G$
with $\partial\G$.

We call the groupoid $\partial\G\rtimes\G$ the \emph{geodesic
  quasi-flow} \index{geodesic quasi-flow} of $\G$. It is easy to see that
$\partial\G\rtimes\G$ is a Hausdorff groupoid of germs
preserving the local direct product structure on $\partial\G$.

\begin{proposition}
If $\G_1$ and $\G_2$ are equivalent hyperbolic groupoids, then the
groupoids $\partial\G_1\rtimes\G_1$ and $\partial\G_2\rtimes\G_2$
are equivalent.
\end{proposition}

\begin{proof}
It follows from the definition of the local product structure on
$\partial\G$ that the space $\partial(\G_1\vee\G_2)$ is the disjoint
union of the spaces $\partial\G_1$ and $\partial\G_2$ and that the
restriction of the action of $\G_1\vee\G_2$ to $\partial\G_1$
coincides with the action of $\G_1$ on $\partial\G_1$. It follows
that the groupoid $\partial(\G_1\vee\G_2)\rtimes(\G_1\vee\G_2)$
defines an equivalence between $\partial\G_1\rtimes\G_1$ and
$\partial\G_2\rtimes\G_2$.
\end{proof}

\begin{proposition}
\label{prop:geodcompgen}
The groupoid $\partial\G\rtimes\G$ is compactly generated.
\end{proposition}

\begin{proof}
We will use notations from the proof of
Theorem~\ref{th:localproduct}.
Denote by $T\subset\partial\G$ the closure of the set of points of
the form $\cdots g_2g_1$ for $g_i\in F\in\mS$ and
$\be(g_1)\in\X$.

By Proposition~\ref{pr:nbhdtil}, for every $\xi\in\partial\G$
there exists $g\in\G$ such that $\xi$ is an internal point of
$\til_g$ and $\en(g)\in\X_0$. Let $\xi=\cdots g_2g_1\cdot g$ for
$g_i\in S$. Then $\xi\cdot g^{-1}=\cdots g_2g_1$ is an internal
point of $\til_{\en(g)}$. Let $U\subset\X_0'$ be an open
neighborhood of $\be(g_1)$ of diameter less than
$\delta_0$. Then $\xi\cdot g^{-1}$ belongs to the open set
$R_{\en(g), U}\subset T$. It follows that $T$ is a topological
transversal.

Let us show that $T$ is compact. Every point of $T$ can be
represented as $\ldots g_2g_1$, where $g_i$ belong to the elements
of the set $\mS$. Consider the compact space $\check
S=\bigcup_{F\in\mS}\overline F$.
The set of all composable
sequences $(\ldots, g_2, g_1)$ of elements of $\check S$ is a
closed subset of the Cartesian product $(\check S)^\infty$.
Hence the space of all composable sequences is
compact. It follows from Proposition~\ref{pr:topology} and
Theorem~\ref{th:localproduct}
that the map $(\ldots, g_2, g_1)\mapsto \ldots g_2g_1$ from the
space of composable sequences to $\partial\G$ is continuous.
Hence, $T$ is contained in a continuous image of a compact space and
is compact.

Consider the set \[T\times_P S=\{(\xi,
g)\in\partial\G\rtimes\G\;:\;\xi\in T, g\in S, P(\xi)=\en(g)\}\]
of germs of the action of elements of $S$ on $T$. It is compact,
since $T$ and $S$ are compact, and $P$ and $\be$ are continuous.

Let us show that $T\times_P S$ is a generating set of the
restriction of $\partial\G\rtimes\G$ to $T$. Let $(\xi,
g)\in\partial\G\rtimes\G$, be an arbitrary element of
$(\partial\G\rtimes\G)|_T$. Then $\en(g)=P(\xi)$, $\xi\in T$, and
$\xi\cdot g\in T$. In particular, $g\in\G|_{\X}$. By
Proposition~\ref{pr:hyperbolicgenset}, the
element $g\in\G$ can be written in the form $r_1\cdots r_k\cdot
(s_1\cdots s_l)^{-1}$ for $r_i, s_i\in S$. The points $\xi\cdot
r_1\cdots r_m$ and
$\xi\cdot g\cdot s_1\cdots s_{n-1}=\xi\cdot r_1\cdots r_k\cdot s_l^{-1}\cdots
s_n^{-1}$, for all $1\le m\le k$ and $1\le n\le l$, belong to $T$. It
follows that $(\xi, g)$ is a product of $k+l$ elements of
$T\times_P S\cup(T\times_P S)^{-1}$.

The set $S$ is a generating set of $\G|_{\X}$, and the length
of a representation of an element $g\in\G|_{\X}$ as a product
of the form $r_1\cdots r_k\cdot (s_1\cdots s_l)^{-1}$ is bounded
from above by a function of the $S$-length of $g$ (see
Proposition~\ref{pr:geodreduction}). It
follows that there exists a neighborhood $U$ of $g$ in $\G$ and a
number $N$ such that every element $h\in U\cap\G|_{\X}$ can be
written as a product $r_1\cdots r_k\cdot (s_1\cdots s_l)^{-1}$ for
$k+l\le N$. Consequently every element of $T\times_P U$ can be represented
as a product of length at most $N$ of elements of
$T\times_P S$, which shows that $T\times_P S$ is a generating
set of $(\partial\G\rtimes\G)|_T$.
\end{proof}

\chapter{Smale quasi-flows and duality}
\label{s:quasiflow}

\section{Definitions}
\begin{defi}
\label{def:hypflow}
Let $\Gh$ be a graded Hausdorff groupoid of germs with a Lipschitz
structure and a local product structure on $\Gh^{(0)}$ preserved by
$\Gh$. We assume that the Lipschitz structure and the grading agree
with the local product structure (see~\ref{s:graddstr}).

Let $\X$ be a compact topological transversal, let $|\cdot|$ be a
metric defined on a compact neighborhood $\wh\X$ of $\X$
and compatible with the Lipschitz structure. Let
$\coc:\G|_{\wh\X}\arr\R$ be an $\eta$-quasi-cocycle compatible with the grading.

We say that the graded groupoid $\Gh$ is a
\emph{Smale quasi-flow} \index{Smale quasi-flow} if there exists a
compact generating set $S$ of $\Gh|_{\X}$  such that
\begin{enumerate}
\item $\coc(g)>3\eta$ for every $g\in S$;
\item  $\be(S)=\en(S)=\X$;
\item
there exists $\lambda\in (0, 1)$ such that every element $g\in S$ has
a rectangular neighborhood $F\in\wt\Gh|_{\wh\X}$ such that
\[|F(x)-F(y)|\le\lambda |x-y|\]
for all $x, y\in\proj_1(\be(F), \be(g))$, and
\[|F(x)-F(y)|\ge\lambda^{-1}|x-y|\] for all $x, y\in\proj_2(\be(F), \be(g))$;
\item the groupoid $\Gh$ is locally diagonal with respect to the local
product structure (see Definition~\ref{def:locdiagonal});
\item for every $r>0$ and every compact subset $A\subset\wh\X$
the closure of the set $\{g\in\G\;:\;\be(g)\in A, \en(g)\in A,
|\coc(g)|\le r\}$ is compact.
\end{enumerate}
\end{defi}

We will call the first and second directions of the local product
structure of a Smale space \emph{stable} and \emph{unstable},
and denote projections $\proj_1$ and $\proj_2$
by $\proj_+$ and $\proj_-$ respectively.

The last condition of Definition~\ref{def:hypflow}
implies that $\coc$ is a quasi-isometry of the Cayley graph of $\Gh$ with
$\R$, which is the reason why we call such groupoids \emph{quasi-flows}.

\begin{defi}
We say that a covering $\mathcal{R}$ of a transversal of a Smale
quasi-flow $\Gh$ by open rectangles
is \emph{fine} \index{fine covering} if it satisfies the conditions of
Definitions~\ref{def:compatiblelipsch},~\ref{def:compatiblegraded},
and~\ref{def:locdiagonal}.
\end{defi}

It is easy to see that if a covering $\mathcal{R}$ is fine then any
covering by subrectangles of $\mathcal{R}$ is also fine.

The following statements are straightforward (compare with
Proposition~\ref{prop:invariancehyperbolic} and Corollary~\ref{cor:equivhyperbolic}).

\begin{proposition}
Let $\Gh$ be a Smale quasi-flow. Then for any compact topological
transversal $\X$ of $\Gh$, a metric on a neighborhood of
$\X$ compatible with the Lipschitz structure, and a quasi-cocycle
$\coc:\Gh|_{\wh\X}\arr\R$ compatible with the grading of $\Gh$
there exists a compact generating set $S$ of $\Gh|_{\wh\X}$ satisfying
the conditions of Definition~\ref{def:hypflow}.
\end{proposition}

\begin{corollary}
Every groupoid of germs equivalent to a Smale quasi-flow is also a
Smale quasi-flow.
\end{corollary}

\begin{examp}
Suppose that $\coc:\Gh\arr\R$ is a  cocycle, i.e., that
$\coc(g_1g_2)=\coc(g_1)+\coc(g_2)$ for all $(g_1, g_2)\in\Gh^{(2)}$
and $\coc$ is continuous, and suppose that it satisfies condition (5) of
Definition~\ref{def:hypflow}. Consider then the action of $\G$ on
$\G^{(0)}\times\R$ given by the formula
\[h\cdot(g, t)=(hg, t+\coc(h)).\]

Note that the natural action of $\R$ on $\G^{(0)}\times\R$ is free,
proper, and commutes with the defined action of $\G$.
It follows from condition (5) of
Definition~\ref{def:hypflow} that the action of $\G$ on
$\G^{(0)}\times\R$ is proper. Suppose that it is free, i.e., that conditions
$\coc(g)=0$ and $\be(g)=\en(g)$ for $g\in\G$ imply that $g$ is a
unit. Then (see~\ref{s:equivalence}) the groupoid of the action of
$\R$ on the space
$M=\G\backslash(\G^{(0)}\times\R)$ is equivalent to the groupoid
$\G$. It follows from the compact generation that $M$ is compact. It is
Hausdorff by properness of the action of $\G$ on
$\G^{(0)}\times\R$. We see that in this case the groupoid $\G$ is
equivalent to a flow on a compact metric space. The space $M$ has a
natural local decomposition into a direct product of three spaces
$P_+\times P_-\times I$, where $I$ is an interval of $\R$, so that the
plaques in the direction of $I$ are orbits of the flow, plaques in the
direction of $P_+$ are contracted, and the plaques in the direction of
$P_-$ are expanded by the flow. In this sense the action of $\R$ on
$M$ is a \emph{Smale flow}.

If the action of $\G$ on $\G^{(0)}\times\R$ is not free, then
$M=\G\backslash(\G^{(0)}\times\R)$ is an \emph{orbispace} and $\G$ is
equivalent to a \emph{Smale orbispace flow}. It is defined by the
action of $\R$ on the \emph{orbispace atlas} $\G^{(0)}\times\R$. For
more on orbispaces see~\cite{nek:book}.
\end{examp}

\section{Special generating sets}

Similarly to hyperbolic groupoids (see
Proposition~\ref{pr:hyperbolicgenset}), we will need generating sets
with special properties.

If $(A, |\cdot|_A)$ and $(B, |\cdot|_B)$ are metric spaces, then the
\emph{direct product} of the metrics $|\cdot|_A$ and $|\cdot|_B$ is
the metric on $A\times B$ given by
\[|(a_1, b_1)-(a_2, b_2)|_{A\times B}=\max(|a_1-a_2|_A, |b_1-b_2|_B).\]

\begin{proposition}
\label{prop:generatorsflow}
Every Smale quasi-flow is equivalent to a groupoid $\Gh$
satisfying the following properties.

The space of units $\Gh^{(0)}$ is a disjoint union of a finite
number of rectangles $W_1=A_1\times B_1, \ldots, W_n=A_m\times
B_m$. The covering $\mathcal{R}=\{W_i\}$ is fine, and the metric on
$\Gh^{(0)}$ is locally isometric to the direct product of metrics
on $A_i$ and $B_i$.

There exists an open transversal $\X_0$ equal to the union of open
sub-rectangles $W_i^\circ=A_i^\circ\times B_i^\circ\subset W_i$
such that closures of $W_i^\circ$ are compact. Denote by $\X$ the
union of closures of the rectangles $W_i^\circ$.

There exists a finite set $\mS$ of elements of the pseudogroup
$\wt\Gh$ such that
\begin{enumerate}
\item every $F\in\mS$ is a rectangle $A_F\times B_F$;
\item for every $F\in\mS$ there exist $i, j\in 1, \ldots, n$ such that
$\be(F)\subset W_i$, $\en(F)\subset W_j$, $\be(A_F)=A_i$, $\en(B_F)=B_j$;
\item intersections of $\be(F)$ and $\en(F)$ with $\X$ are
non-empty;
\item $A_F$ and $B_F^{-1}$ are $\lambda$-contracting for some
$\lambda\in (0, 1)$;
\item $S=\{(F, x)\;:\;x, F(x)\in\X\}$ is a
generating set of $\Gh|_{\X}$;
\item $\be(S)=\en(S)=\X$;
\item $\coc(g)>2\eta$ for all germs of elements of $\mS$;
\end{enumerate}
\end{proposition}

See Figure~\ref{fig:generators} illustrating
Proposition~\ref{prop:generatorsflow}.

\begin{figure}
\includegraphics{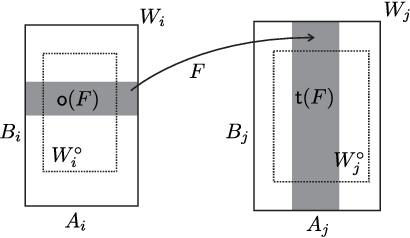}
\caption{}\label{fig:generators}
\end{figure}

\begin{proof}
Let $\Gh'$ be a Smale quasi-flow. Let $\X'$ be a
compact topological transversal of $\Gh'$. We can choose $\X'$ equal
to the union of a finite set of rectangles $\overline
R$ equal to the closure of an open rectangle
$R$ and contained in an open rectangle $\wh R$. Let us denote by
$\mathcal{R}$ the set of the rectangles $R$, and let $\wh{\mathcal{R}}$
be the set of the rectangles $\wh R$. Then the union of
the rectangles $R\in\mathcal{R}$ is a $\Gh'$-transversal, which we will denote
$\X_0'$. We assume that $\wh{\mathcal{R}} $ is fine.

According to Definition~\ref{def:compatiblelipsch}, we may
assume that the metric
$|\cdot|$ on $\X'$ is uniformly Lipschitz equivalent to metrics $|\cdot|_R$
on $\wh R$ for
$R\in\mathcal{R}$ such that
\[\left|[x_1, y_1]_R-[x_2, y_2]_R\right|_R=\max\{|x_1-x_2|_R,
|y_1-y_2|_R\}\]
for all $x_1, y_1, x_2, y_2\in\wh R$.

Passing, if necessary, to subrectangles and using Definition~\ref{def:ldps},
we may assume that for any two rectangles $R_i, R_j\in\mathcal{R}$
the products structures $[\cdot, \cdot]_i$ and
$[\cdot, \cdot]_j$ of the rectangles $\wh R_i$ and $\wh R_j$
agree on the intersection $\wh R_i\cap\wh R_j$.

Let $S'$ be a compact generating set of $\Gh'|_{\X'}$ satisfying the
conditions of Definition~\ref{def:hypflow}. We may assume (as usual, by
passing to $(S')^n\cup(S')^{n+1}$) that $S'$ satisfies the
contraction-expansion property with respect to the metrics $|\cdot|_R$.
Let us find a covering $\mathcal{V}$ of $S'$ by a finite number of
open rectangles $V$ such that the closure of each rectangle
$V\in\mathcal{V}$ is compact and extendable to a
rectangle $\wh V\in\wt{\Gh'}$; the rectangles $\wh V$ satisfy the
expansion-contraction condition with
a coefficient $\lambda\in (0, 1)$; and they are
subordinate to the covering $\{\wh R\;:\;R\in\mathcal{R}\}$.

Let $\delta$ be such that for all $V\in\mathcal{V}$ and $R_o,
R_t\in\mathcal{R}$ such that
$\be(V)\cap R_o\ne\emptyset$ and $\en(V)\cap R_t\ne\emptyset$,
the $\delta$-neighborhood of $\be(V)$ in $(\wh R_o, |\cdot|_{R_o})$
is contained in $\be(\wh V)$, and the $\delta$-neighborhood of
$\en(V)$ in $(\wh R_t, |\cdot|_{R_t})$ is contained in $\en(\wh V)$.
Such $\delta$ exists, since the metrics $|\cdot|_R$ are pairwise
Lipschitz equivalent
on intersections of their domains, and $\mathcal{V}$ is subordinate to
$\mathcal{R}$. Let now $\epsilon>0$ be such that
$\epsilon<\min(\delta(\lambda^{-1}-1), \delta)$.

We can find a finite set $\mathcal{W}=\{W_1, \ldots, W_n\}$
of open sub-rectangles of $\X'$ such that the closure $\overline W_i$ of each
$W_i\in\mathcal{W}$
is a subrectangle of a rectangle
$R(W)\in\mathcal{R}$ and is of diameter less than
$\epsilon$ (with respect to $|\cdot|_{R(W)}$),
$\X_0'=\bigcup_{W\in\mathcal{W}}W$, and $\X'=\bigcup_{W\in\mathcal{W}}\overline W$.

For every $W\in\mathcal{W}$ let $\wh W$ be the
$(\delta-\epsilon)$-neighborhood
of $\overline W$ in the metric space $(\wh{R(W)}, |\cdot|_{R(W)})$. Then, by the
properties of $|\cdot|_{R(W)}$, the set $\wh W$ is an open sub-rectangle of
$\wh{R(W)}$.

Let $\Gh$ be localization of $\Gh'$ onto the covering $\{\wh W_1,
\ldots, \wh W_n\}$. We will denote the elements
of the localization $\Gh$ by $(g, i, j)$ for $i, j\in\{1, \ldots,
n\}$ (see definition of localization in~\ref{ss:localization}).

 Let $\X_0$ be union of the copies of the rectangles
$W$, and let $\X$ be the closure of $\X_0$ (i.e., the union of the
copies of $\overline W$). Since every point of $\X_0'$ is an interior
point of a rectangle $W$, the set $\X_0$ is an open transversal of
$\Gh$.

Define a metric on $\Gh=\bigsqcup_{W\in\mathcal{W}}\wh W$ equal to
$|\cdot|_{R(W)}$ on every subset $\wh W$ (and, for example, equal to
one for any two points belonging to different elements of $\mathcal{W}$).

\begin{lemma}
The set $S$ of copies $(g, i, j)$ of elements $g\in S'$ such that $\be(g),
\en(g)\in\X$ is a generating set of $\Gh|_{\X}$ satisfying
condition (8).
\end{lemma}

\begin{proof}
Let $(g, s, r)\in\Gh|_{\X}$, where $g\in\Gh'$ and $s, r\in\{1,
\ldots, n\}$ are such that
$\be(g)\in\wh W_o$ and $\en(g)\in\wh W_t$.  Note that by definition of
$\X$, we in fact have $\be(g)\in\overline W_o=\wh W_o\cap\X$ and
$\en(g)\in\overline W_t=\wh W_t\cap\X$. It follows that
$g\in\Gh'|_{\X'}$. Consequently, there exists a neighborhood $U\in\wt{\Gh'}$ of
$g$ and $n$ such that every element of $U\cap\Gh'|_{\X'}$ can be
represented as a product of not more than $n$ elements of $S'\cup
(S')^{-1}$. The set of elements of the form $(h, s, r)$, where $h\in
U$ is a neighborhood of $(g, s, r)$.

If $(h, s, r)$ belongs to $\Gh|_{\X}$, then $\be(h)\in\overline
W_o$ and $\en(h)\in\overline W_t$. For any representation
$h=g_1\cdots g_k$ of $h$ as a product of
elements of $S'\cup (S')^{-1}$ we can then find indices
$i_k=s, \ldots, i_1, i_0=r$ such that $\be(g_m)\in\overline W_{i_m}$, and
$\en(g_m)\in\overline W_{i_{m-1}}$ for all $m=1, 2, \ldots, m$.
Then $(h, s, r)=(g_1, i_1, i_0)(g_2, i_2, i_1)\cdots (g_m,
i_m, i_{m-1})$ is a representation of $(h, s, r)$ as a product of
elements of $S$. It follows that $S$ is a generating set of
$\Gh|_{\X}$.
\end{proof}

Let $(g, s, r)\in S$ for some $r, s\in\{1, \ldots,
n\}$ and $g\in S'$. Denote $R_o=R(W_o)$ and $R_t=R(W_t)$.

Suppose that $g\in V\in\mathcal{V}$.
The sets $\be(\wh V)$ and $\en(\wh V)$ contain the $\delta$-neighborhoods of
$\be(V)$ and $\en(V)$ in the
metric spaces  $(\wh R_o, |\cdot|_{R_o})$ and $(\wh R_t, |\cdot|_{R_t}$) respectively.
The sets $\wh W_o$ and $\wh W_t$ are the
$(\delta-\epsilon)$-neighborhoods of $\overline{W_o}$
and $\overline{W_t}$, which in turn are contained in the
$\epsilon$-neighborhoods of the points $\be(g)$ and $\en(g)$
respectively. It follows that $\wh W_o$ and $\wh W_t$ are contained in
the $\delta$-neighborhoods of $\be(g)\in R_o$ and $\en(g)\in R_w$ respectively,
which implies that $\be(\wh V)\supset\wh W_o$ and
$\en(\wh V)\supset\wh W_t$.

For every $x\in\overline{W_o}$ the plaque $\proj_+(\wh W_o, x)$
has diameter less than $\delta$ and is mapped
by $\wh V$ into the intersection of the $\lambda\delta$-neighborhood
of $\wh V(x)$ with $\proj_+(\wh W_t, \wh V(x))$. But since
$\lambda\delta<\delta-\epsilon$, the image $\wh V(\proj_+(\wh W_o, x))$
is contained in $\proj_+(\wh{W_t}, \wh V(x))$. The same is true for the
inverse of $\wh V$ and the plaques $\proj_-(\wh W_o, x)$,
$\proj_-(\wh W_t, \wh V(x))$, which implies that the set of maps of
the form
\[\wh V:\wh W_o\cap(\wh V)^{-1}(\wh W_t)\arr\wh W_t\cap\wh V(\wh
W_o)\]
such that $V\cap S'\neq\emptyset$ satisfy the conditions of
the proposition.
\end{proof}

\section[Geodesic quasi-flow as a Smale quasi-flow]{Geodesic quasi-flow of a hyperbolic groupoid as
a Smale quasi-flow}
\begin{theorem}
\label{th:geodflowhyperbolic}
Let $\G$ be a hyperbolic groupoid and suppose that the geodesic quasi-flow
$\partial\G\rtimes\G$ is locally diagonal (with respect to the
local product structure on $\partial\G$). Then $\partial\G\rtimes\G$ is a
Smale quasi-flow.
\end{theorem}

\begin{proof}
By Proposition~\ref{prop:geodcompgen}, the geodesic quasi-flow
$\partial\G\rtimes\G$ is compactly generated. Let $(S, \X)$ be a
generating pair of $\G$ satisfying the conditions of
Proposition~\ref{pr:hyperbolicgenset}, and consider the atlas of the local
product structure on $\partial\G$ defined in
the proof of Theorem~\ref{th:localproduct}.

Denote by $\ell_{\be(h)}(\xi_1, \xi_2)$
the natural log-scale on $\partial\G_{\be(h)}$ (as defined in
Proposition~\ref{pr:topology}).
Let us define then on each rectangle $R_{h, U}$ the log-scale
\[\ell_{h, U}([x_1, \xi_1], [x_2, \xi_2])=
\min\{\ell_{\be(h)}(\xi_1, \xi_2), \ell(x_1, x_2)\},\]
where $\ell$ is the log-scale defining the Lipschitz structure of $\G$.

Note that by Proposition~\ref{pr:injective} there
exists a uniform constant $\Delta$ such that
\[|\ell_y([y, \xi_1]_U, [y, \xi_2]_U)-\ell_{\be(h)}(\xi_1,
\xi_2)|\le\Delta\]
for all $U$, $\xi_1, \xi_2$, and $y$ for which the corresponding
expressions are defined.

Suppose that $g\in\G$ is such that $\en(g)\in\be(U_i)$ and
$\be(g)\in\be(U_j)$, so that $g$ moves the plaque $[\en(g),
\til_{h_i}^\circ]_{U_i}\subset\partial\G_{\en(g)}$ of $R_{h_i, U_i}$ to the
plaque $[\be(g), \til_{h_i}^\circ]_{U_j}\subset\partial\G_{\be(g)}$ of
$R_{h_j, U_j}$. For all $\zeta_1, \zeta_2\in\partial\G_{\en(g)}$ we
have
\begin{equation}
\label{eq:nataction}
\ell_{\be(g)}(\zeta_1\cdot g, \zeta_2\cdot
g)\doteq\ell_{\en(g)}(\zeta_1, \zeta_2)+\coc(g),
\end{equation}
by the definition of the natural log-scale.

It follows that for all $\zeta_1, \zeta_2\in R_{h_i, U_i}$ such that
$\proj_+(\zeta_i)=\en(g)$ and $\zeta_1\cdot g, \zeta_2\cdot g\in
R_{h_j, U_j}$ we have
\begin{equation}
\label{eq:ellonpartial}
\ell_{h_j, U_j}(\zeta_1\cdot g, \zeta_2\cdot g)\doteq\ell_{h_i U_i}(\zeta_1,
\zeta_2)+\coc(g).
\end{equation}

It follows that the log-scales $\ell_{h_i, U_i}$ define a Lipschitz
structure on $\partial\G\rtimes\G$.

Let us define a quasi-cocycle on $\partial\G\rtimes\G$ just by
lifting it from the quasi-cocycle $\coc$ on $\G$:
\begin{equation}\label{eq:liftofcoc}\wt\coc(\xi, g)=\coc(g).\end{equation}
Let us show that this quasi-cocycle is compatible with the local product
structure. We will use the covering by the rectangles $R_{h_i, U_i}$ again.
If two elements of $\partial\G\rtimes\G$ have equal projections onto
the second coordinate of the local product structure,
then they are of the form $(\xi_1, g)$ and
$(\xi_2, g)$, hence the values of $\wt\coc$ on them are equal.

If two elements have equal projections on the first coordinate, then
the corresponding elements $g_1$ and $g_2$ of $\G$ act on the
corresponding factors $\til_{h_i}^\circ\arr\til_{h_j}^\circ$ of the
rectangles by transformations coinciding on some non-empty open
sets. But then using Proposition~\ref{pr:injective}
and~\eqref{eq:nataction} we get a universal upper bound on the difference
$|\coc(g_1)-\coc(g_2)|$, which shows that the quasi-cocycle is compatible
with the local product structure.

It follows from~\eqref{eq:ellonpartial} that
after passing, if necessary, to $\mS^n\cup\mS^{n+1}$, we will
obtain a generating set of $\G$ acting by expanding maps on the first
direction, and contracting maps on the second direction of the rectangles
$R_{h_i, U_i}$.

It remains to show that the last condition of
Definition~\ref{def:hypflow} is satisfied. Let $A\subset\partial\G$
be a compact set, and let $r>0$. We can cover the set $A$ by a finite
set of rectangles $R_{h_i, U_i}$ for $h_i\in\G$ and
$U_i\in\pG$. Suppose that $\xi_1\cdot g=\xi_2$ for some $\xi_1,
\xi_2\in A$ and $g\in\G$ such that $|\coc(g)|<r$. Let $\xi_1=\ldots
g_2'g_1'\cdot h_i$ and $\xi_2=\ldots g_2''g_1''\cdot h_j$ for some
$g_k', g_k''\in S$. Then $\ldots g_2'g_1'\cdot h_ig=\ldots
g_2''g_1''\cdot h_j$. Let $M$ be the maximal value of $\coc$ on an
element of $S$ and suppose that $\coc(g)$ is non-negative.
There exists a uniform constant $\Delta_1$ and an index
$k$ such that
\[0\le\coc(h_ig)-\coc(g_k''\cdots
g_2''g_1''\cdot h_j)\le 2M,\quad\text{and}\quad |g_k''\cdots
g_2''g_1''h_j\cdot g^{-1}h_i^{-1}|\le\Delta_1,\]
where $|\cdot|$ denotes the length of the shortest representation as a
product of elements of $S\cup S^{-1}$.
But then
\[|g|\le |h_j|+k+\Delta_1+|h_i|,\]
and \[(k-1)\eta\le\coc(g_k''\cdots
g_1'')\le\coc(g)+\coc(h_i)-\coc(h_j)+2M+3\eta,\]
hence we get a uniform bound on $k$ and $|g|$, which implies
condition (5) of Definition~\ref{def:hypflow}.
\end{proof}

\section{Hyperbolicity of the Ruelle groupoids}

\begin{defi}
Let $\G$ be a Smale quasi-flow. Its \emph{Ruelle groupoids}
\index{Ruelle groupoid} \index{groupoid!Ruelle}
$\proj_i(\G)$ for $i=+, -$ are the groupoids $\proj_i(\G,
\mathcal{R})$, where $\mathcal{R}$ is a compressible
covering of a topological $\G$-transversal by rectangles.
\end{defi}

By Proposition~\ref{prop:generatorsflow}, compressible coverings exist. By
Proposition~\ref{prop:compressible}, Ruelle groupoids
$\proj_i(\G)$ are unique, up to an equivalence of groupoids.

Since the Lipschitz structure on $\G^{(0)}$ is compatible with the
local product structure, we get naturally defined Lipschitz
structures on the Ruelle groupoids.

\begin{theorem}
\label{th:Ruellehyperbolic}
Let $\Gh$ be a Smale quasi-flow and let $\coc:\Gh|_{\X}\arr\R$ be
an $\eta$-quasi-cocycle compatible the grading, where $\X$ is a compact
topological transversal. Let $\mathcal{R}$ be a finite compressible
covering of $\X$ by open rectangles. Then there exists a
quasi-cocycle $\coc_+:\proj_+(\Gh, \mathcal{R})\arr\R$
such that the groupoid $\proj_+(\Gh, \mathcal{R})$ is hyperbolic,
$\coc_+$ is a Busemann cocycle, and $|\coc_+(\proj_+(g))-\coc(g)|$ is
uniformly bounded for all $g\in\Gh|_{\X}$.
\end{theorem}

\begin{proof}
Let us pass to an equivalent groupoid satisfying the conditions of
Proposition~\ref{prop:generatorsflow} and use its notations.
The fact that the generating set $\proj_+(S)$ is a set of germs of contractions
follows directly from the definitions.

Let $\Delta_1$ be such that $\coc(g)\le\Delta_1-\eta$ for all $g\in F\in\mS$.
Let $\X'$ be a compact neighborhood of $\X$, and let
$\nuke$ be the closure of the set of elements $g\in\Gh$ such that
$|\coc(g)|<\Delta_1$ and $\be(g), \en(g)\in\X'$.
There exists $\delta_0>0$ and a finite collection
$\mathcal{U}\subset\wt{\Gh}$
such that for every $g\in\nuke$ there exists $U\in\mathcal{U}$ such that
$g\in U$, and $\be(U)$ and $\en(U)$ contain the
$\delta_0$-neighborhoods of $\be(g)$ and $\en(g)$ respectively.
We assume that the elements $U\in\mathcal{U}$ are $\Lambda$-Lipschitz
for some common constant $\Lambda$.

Let $\proj_+(s_1)\cdots\proj_+(s_n)$ be any product of elements of
$\proj_+(S)\cup\proj_+(S^{-1})$. Let
$F_i\in\mS\cup\mS^{-1}$ be such that $s_i\in F_i$.
Let $R_1, \ldots, R_{n+1}\in\mathcal{R}$ be such
that $\be(F_n)\subset R_{n+1}$, $\en(F_n)\cup\be(F_{n-1})\subset
R_n$, \ldots, $\en(F_1)\subset R_1$ (see Figure~\ref{fig:ruelle}).

\begin{figure}
\centering
\includegraphics{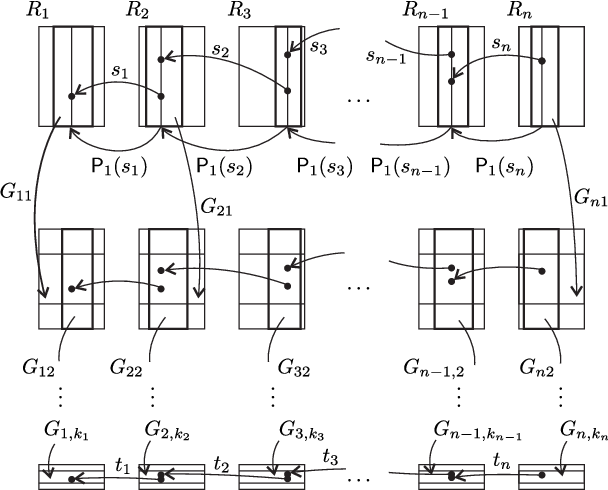}
\caption{}\label{fig:ruelle}
\end{figure}

For every $i=1, \ldots, n+1$ consider a sequence $G_{ik}\in\mS^{-1}$ such
that \[\{\en(s_i), \be(s_{i-1})\}\subset\be(G_{ik}\cdots
G_{i2}G_{i1})\] and $G_{ik}\cdots G_{i1}(\en(s_i))\in\X$. Such
sequences exist by conditions (2), (5), (6) of
Proposition~\ref{prop:generatorsflow}. Then
$\proj_-(G_{ik}\cdots G_{i1})$ are $\lambda^k$-contractions, which are
defined on whole side $\proj_-(R_i)$ of the rectangle $R_i$.

We will use notation
\[\prod_l^nG_i=G_{i, n}G_{i, n-1}\cdots G_{i, l}.\]

We can find a sequence $k_1, k_2, \ldots, k_{n+1}$ such that
the values of $\coc$ on the elements of the
set \[\left(\prod_1^{k_i}G_i\right)F_i\left(\prod_1^{k_{i+1}}G_{i+1}\right)^{-1}\]
are not more than $\Delta_1$ in absolute value.

Moreover, we can make the indices $k_i$ as large as we wish.
We may hence assume that diameters of
$\proj_-\left(\en\left(\prod_1^{k_i}G_i\right)\right)$ are less than
$\epsilon=\delta_0\frac{\Lambda-1}{\Lambda^n-\Lambda}$.

If $k_i$ and $k_{i+1}$ are big enough, then the germs
\[t_i=\left(\prod_1^{k_i}G_i\right)s_i\left(\prod_1^{k_{i+1}}G_{i+1}\right)^{-1}\]
are such that $\be(t_i), \en(t_i)\in\X'$, hence $t_i\in\nuke$.
Consequently, there exist $\Lambda$-Lipschitz
elements $H_i\in\mathcal{U}$ such that $t_i\in H_i$
and $\be(H_i)$ and $\en(H_i)$ contain the $\delta_0$ neighborhoods
of $\be(t_i)$ and $\en(t_i)$ respectively.

Distance between $\en(t_{i+1})$ and $\be(t_i)$ is
less than $\epsilon$. It follows that
\begin{eqnarray*}
|H_n(\be(t_n))-\be(t_{n-1})|<&\epsilon&<\delta_0,\\
|H_{n-1}H_n(\be(t_n))-\be(t_{n-2})|<&\Lambda\epsilon+\epsilon&<\delta_0,\\
\vdots\\
|H_2H_3\cdots H_n(\be(t_n))-\be(t_1)|<&\Lambda^{n-2}\epsilon+\cdots+\Lambda\epsilon+\epsilon&<\delta_0,
\end{eqnarray*}
hence $H_1\cdots H_n(\be(t_n))$ is defined, and
\[
|H_1H_2\cdots H_n(\be(t_n))-\en(t_n)|<\Lambda^{n-1}\epsilon+\Lambda^{n-2}\epsilon+\cdots+\Lambda\epsilon=\delta_0.\]

Since the maps $F_i$ and $G_{i, j}$ are rectangles, i.e., agree with
the product structure on the rectangles of $\Gh^{(0)}$, we get that
\[\proj_+\left(\prod_1^{k_1}G_1\right)\proj_+(s_1)\cdots
\proj_+(s_n)\proj_+\left(\prod_1^{k_{n+1}}G_{n+1}\right)^{-1}=
\proj_+(r_1\cdots r_n)\] for some germs $r_i$ of $H_i$ such that
$\be(r_n)=\be(t_n)$.

By the same argument,
\[\proj_+\left(\prod_1^{k_1}G_1\right)\proj_+(s_1)\cdots
\proj_+(s_n)\proj_+\left(\prod_1^{k_{n+1}}G_{n+1}\right)^{-1}=
\proj_+(r_1'\cdots r_n'),\]
for some germs $r_i'$ of $H_i$ such that $\en(r_1')=\en(t_1)$.

We can find $l_1>k_1$ and $l_{n+1}>k_{n+1}$
such that \[\left|\coc\left(\prod_{k_1+1}^{l_1}G_1\cdot
r_1\cdots
r_n\cdot\left(\prod_{k_{n+1}+1}^{l_{n+1}}G_{n+1}\right)^{-1}\right)\right|
\le\Delta_1,\]
and hence
\[\prod_{k_1+1}^{l_1}G_1\cdot r_1\cdots
r_n\cdot\left(\prod_{k_{n+1}+1}^{l_{n+1}}G_{n+1}\right)^{-1}\in\nuke.\]
We have proved therefore the following result.

\begin{lemma}
\label{lem:reductiontop}
For every $g\in\proj_+(\Gh)$, and every pair of sequences $G_i',
G_i''\in\mS^{-1}$ such that $\be(g)\in\be(G_i'\cdots G_1')$,
$\en(g)\in\be(G_i''\cdots G_1'')$, $G_i''\cdots G_1''(\en(g))\in\X$,
there exist $l, k\ge 1$ and $r\in Q$ such
that \[g=\proj_+(G_l''\cdots G_1'')^{-1}\proj_+(rG_k'\cdots G_1').\]
Moreover, $r$ and $k$ exist for all sufficiently big $l$, and $r$ and
$l$ exist for all sufficiently big $k$.
\end{lemma}

Suppose that the groupoid $\proj_+(\Gh)$ is not Hausdorff.
Then there exist two elements $h_1, h_2$ that
can not be separated by neighborhoods, hence there exists an element
$g=h_1^{-1}h_2\in\proj_+(\Gh)$ such that $\be(g)=\en(g)$, and
$g$ is not a unit, but $g$ can not be separated by disjoint
neighborhoods from the unit $\be(g)$. This means that for every
neighborhood $U\in\wt{\proj_+(\Gh)}$ of $g$ there exists an open
subset $W\subset U$ which is an identical homeomorphism.

By Lemma~\ref{lem:reductiontop}, if $G_1, G_2, \ldots$ is a sequence of elements
of $\mS$ such that $\be(g)=\en(g)\in\be(G_kG_{k-1}\cdots G_1)$ and
$G_kG_{k-1}\cdots G_1(\be(g))\in\X$ for all $k$, then there exist
$k$ and $l$, and $r, r'\in\nuke$ such that
\[g=\proj_+(G_k\cdots G_1)^{-1}\proj_+(r)\proj_+(G_l\cdots
G_1),\quad\be(r)\in\en(G_l\cdots G_1)\]
and
\[g=\proj_+(G_k\cdots G_1)^{-1}\proj_+(r')\proj_+(G_l\cdots
G_1),\quad\en(r')\in\en(G_k\cdots G_1).\]

For any neighborhood $W\in\wt{\Gh}$ of $r$ the map \[W_1=\proj_+(G_k\cdots
G_1)^{-1}\proj_+(W)\proj_+(G_l\cdots G_1)\] is a neighborhood of $g$
belonging to $\wt{\proj_+(\Gh)}$. Hence, $\be(g)$ is a limit of a
sequence of trivial germs of $W_1$.

If $l>k$, then
\begin{multline*}
W_1=\proj_+(G_k\cdots G_1)^{-1}\proj_+(W)\proj_+(G_l\cdots G_1)=\\
\proj_+(G_k\cdots G_1)^{-1}\bigl(\proj_+(WG_l\cdots
G_{k+1})\bigr)\proj_+(G_k\cdots G_1),
\end{multline*}
hence the trivial germs of the transformation $\proj_+(WG_l\cdots
G_{k+1})$ accumulate on $x_k=\proj_+(G_k\cdots G_1)(\be(g))$. But then, by
condition (4) of Definition~\ref{def:hypflow}, the trivial germs of the
transformation $WG_l\cdots G_{k+1}$ accumulate on the points of the set
$\proj_+^{-1}(x_k)\cap\be(WG_l\cdots G_{k+1})$.

Consider the point $z=\be((G_l\cdots G_1)^{-1}r)$. Then $z\in\be(G_k\cdots
G_1)\cap\proj_+^{-1}(\be(g))$. The point $G_k\cdots G_1(z)$ belongs to
$\be(WG_l\cdots G_{k+1})$ and $WG_l\cdots G_1(z)=G_k\cdots G_1(z)$,
since the trivial germs of $WG_l\cdots G_{k+1}$ accumulate on
$G_k\cdots G_1(z)$. Then the local homeomorphism
$(G_k\cdots G_1)^{-1}WG_l\cdots G_1(z)$ is
defined (i.e., is not empty),
hence $g=\proj_+((G_k\cdots G_1)^{-1}rG_l\cdots G_1)$. It
follows that every neighborhood of $(G_k\cdots G_1)^{-1}rG_l\cdots
G_1$ contains trivial germs, which by Hausdorffness of $\Gh$ implies
that $(G_k\cdots G_1)^{-1}rG_l\cdots
G_1$ is trivial, hence $g$ is also trivial, which is a contradiction.
The case $k<l$ is treated in the same way, but using $r'$.
Thus, $\proj_+(\Gh)$ is Hausdorff.

Let us define a grading on $\proj_+(\Gh)$. Consider an arbitrary element
$g\in\proj_+(\Gh)$. According to
Lemma~\ref{lem:reductiontop}, for any two sequences $G_1',
G_2', \ldots$ and $G_1'', G_2'', \ldots$ of elements of $\mS^{-1}$ such that
$\be(g)\in\be(G_i'\cdots G_1')$, $\en(g)\in\be(G_i''\cdots G_1'')$,
and $G_i''\cdots G_1''(\en(g))\in\X$ there exist $k, l\ge 1$ and
$r\in\nuke$ such that
\[g=\proj_+(G_l''\cdots G_1'')^{-1}\proj_+(rG_k'\cdots G_1').\]
Define
\[\coc_+(g)=\coc(G_k'\cdots G_1', z_1)-\coc(G_l''\cdots G_1'', z_2),\]
where $z_1$ and $z_2$ are arbitrary points in the domains of the
corresponding homeomorphisms. Since we assume that the covering
$\mathcal{R}$ by rectangles of $\Gh$ (see
Proposition~\ref{prop:generatorsflow}) is fine, the value of $\coc_+(g)$,
up to a uniformly bounded constant, does not depend on the choice of
the points $z_1$ and $z_2$.

Let us show that $\coc_+$ is well defined, up to strong equivalence.
Suppose that $H_1', H_2', \ldots$, $H_1'', H_2'', \ldots$, is another
pair of sequences of elements of $\mS$, and let $m, n\in\N$, and
$t\in\nuke$ are such that
\[g=\proj_+(H_n''\cdots H_1'')^{-1}\proj_+(tH_m'\cdots H_1').\]

We can find indices $k_1>k$, $l_1>l$, $m_1>m$, and $n_1>n$ such that
all the differences
\begin{eqnarray*}
|\coc(G_{k_1}'\cdots G_{k+1}', y_1)&-&\coc(G_{l_1}''\cdots G_{l+1}'', y_2)|,\\
|\coc(H_{m_1}'\cdots H_{m+1}', y_3)&-&\coc(H_{n_1}''\cdots H_{n+1}'', y_4)|,\\
|\coc(G_{k_1}'\cdots G_1', y_5)&-&\coc(H_{m_1}'\cdots H_1', y_6)|
\end{eqnarray*}
are less than $\Delta_1+2c$,
where $y_i$ are arbitrary points in the domains of
the corresponding maps, and $c$ is as in
Definition~\ref{def:compatiblegraded}

Then the elements
\begin{gather*}
r'=G_{k_1}'\cdots G_{k+1}'r^{-1}(G_{l_1}''\cdots G_{l+1}'')^{-1},\quad
u=H_{m_1}'\cdots H_1'z(G_{k_1}'\cdots G_1')^{-1},\\
t'=H_{n_1}''\cdots H_{n+1}''t(H_{m_1}'\cdots H_{m+1}')^{-1}
\end{gather*}
belong to a fixed compact set $\nuke'$,
where $z\in\proj_+^{-1}(\be(g))$ is such that
$G_k'\cdots G_1'(z)=\be(r)$ and $H_n'\cdots H_1'(z)=\be(t)$ (see
Figure~\ref{fig:uniqueness}). Here we identify $z$ with
the corresponding unit of the groupoid. In particular, for $F\in\wt{\Gh}$
the product $Fz$ coincides with the germ $(F, z)$.
Note that then the product
$t'ur'$ is defined and belongs to $(\nuke')^3$.

\begin{figure}
\includegraphics{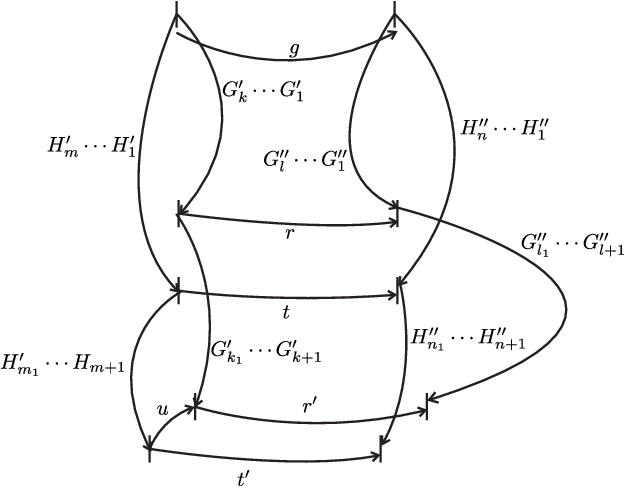}
\caption{}\label{fig:uniqueness}
\end{figure}

It follows from
\[g=\proj_+(G_l''\cdots G_1'')^{-1}\proj_+(r)\proj_+(G_k'\cdots G_1')
=\proj_+(H_n''\cdots H_1'')^{-1}\proj_+(t)\proj_+(H_m'\cdots H_1')\]
that
\begin{multline*}
\proj_+((\nuke')^3)\ni
\proj_+(H_{n_1}''\cdots H_{n+1}''t(H_{m_1}'\cdots H_{m+1}')^{-1})\cdot
\proj_+(H_{m_1}'\cdots H_1'z(G_{k_1}'\cdots
G_1')^{-1})\cdot\\ \proj_+(G_{k_1}'\cdots G_{k+1}'r^{-1}(G_{l_1}''\cdots
G_{l+1}'')^{-1})=\\
\proj_+(H_{n_1}''\cdots H_{n+1}'')\proj_+(t)\proj_+(H_m'\cdots
H_1')\be(g)\proj_+(G_k'\cdots G_1')^{-1}\proj_+(r)^{-1}(G_{l_1}''\cdots
G_{l+1}'')^{-1}=\\
\proj_+(H_{n_1}''\cdots H_{n+1}'')\cdot\proj_+(H_n''\cdots
H_1'')gP(H_m'\cdots H_1')^{-1}\cdot\\
P_1(H_m'\cdots H_1')\be(g)P_1(G_k'\cdots G_1')^{-1}\cdot
\proj_+(G_k'\cdots G_1')g^{-1}\proj_+(G_l''\cdots G_1'')^{-1}\cdot\\
(G_{l_1}''\cdots G_{l+1}'')^{-1}=
\proj_+(H_{n_1}''\cdots H_1'')\en(g)\proj_+(G_{l_1}''\cdots G_1'')^{-1}.
\end{multline*}
Let $M$ be an upper bound of the value of $|\coc|$ on elements of
$(\nuke')^3$. Then the values of $|\coc|$ on $H_{n_1}''\cdots
H_1''z_2(G_{l_1}''\cdots G_1'')^{-1}$ are bounded above by $M+C$ for
every $z_2\in\proj_+^{-1}(\en(g))$, where
$c$ is as in Definition~\ref{def:compatiblegraded}. Consequently,
\[|\coc(H_{n_1}''\cdots H_1'', z_2)-\coc(G_{l_1}''\cdots G_1'',
z_2)|<M+C+\eta.\]
By the same arguments,
\[|\coc(H_{m_1}'\cdots H_1', z)-\coc(G_{k_1}'\cdots G_1',
z)|<M+C+\eta.\]

We have
\[\coc(G_{k_1}'\cdots G_{k+1}', \wt z)-\eta\le\coc(G_{k_1}'\cdots
G_1', z)-\coc(G_k'\cdots G_1', z)\le
\coc(G_{k_1}'\cdots G_{k+1}', \wt z)+\eta,\]
where $\wt z=G_k'\cdots G_1'(z)$;
and
\[
\coc(G_{l_1}''\cdots G_{l+1}'', \wt z_2)-\eta\le\coc(G_{l_1}''\cdots
G_1'', z_2)-\coc(G_l''\cdots G_1'', z_2)\le\coc(G_{l_1}''\cdots
G_{l+1}'', \wt z_2)+\eta,\]
where $\wt z_2=G_l''\cdots G_1''(z_2)$.
Therefore,
\begin{multline*}
|\bigl(\coc(G_{k_1}'\cdots G_1', z)-\coc(G_k'\cdots G_1', z)\bigr)-
\bigl(\coc(G_{l_1}''\cdots G_1'', z_2)-\coc(G_l''\cdots G_1'', z_2)\bigr)|\le\\
|\coc(G_{k_1}'\cdots G_{k+1}', \wt z)-\coc(G_{l_1}''\cdots
G_{l+1}'', \wt z_2)|+2\eta\le\Delta_1+2C+2\eta.\end{multline*}
Similarly,
\begin{multline*}
|\coc(H_{m_1}'\cdots H_1', z)-\coc(H_m'\cdots H_1', z)-
\coc(H_{n_1}''\cdots H_1'', z_2)+\coc(H_n''\cdots H_1'', z_2)|\le\\
\Delta_1+2C+2\eta.
\end{multline*}

It follows that
\begin{multline*}|\bigl(\coc(G_l''\cdots G_1'', z_2)-\coc(G_k'\cdots
G_1', z)\bigr)-\bigl(\coc(H_n''\cdots H_1'', z_2)-\coc(H_m'\cdots H_1',
z)\bigr)|\le\\
|\coc(G_{l_1}''\cdots G_1'', z_2)-\coc(H_{n_1}''\cdots H_1'', z_2)|+|
\coc(H_{m_1}'\cdots
H_1', z)-\coc(G_{k_1}'\cdots
G_1', z)|+\\ |\coc(G_{k_1}'\cdots G_1', z)-\coc(G_k'\cdots G_1', z)-
\coc(G_{l_1}''\cdots G_1'', z_2)+\coc(G_l''\cdots G_1'', z_2)|+\\
|\coc(H_{n_1}''\cdots H_1'', z_2)-\coc(H_n''\cdots
H_1'', z_2)+\coc(H_m'\cdots H_1', z)-\coc(H_{m_1}'\cdots H_1',
z)|\le\\
2M+2C+2\eta+2\Delta_1+4C+4\eta,
\end{multline*}
hence the map $\coc_+$ is well defined, up to an additive constant. This
also implies that $\coc_+$ is a quasi-cocycle.

Theorem~\ref{th:contraction} and Lemma~\ref{lem:reductiontop} imply
that the directed Cayley graph of
$\proj_+(\Gh)$, where oriented edges correspond to the elements of
$\proj_+(S^{-1})$, is Gromov hyperbolic, and that the quasi-cocycle
$\coc_+$ is a Busemann quasi-cocycle associated with the limit of the directed paths.
\end{proof}

Replacing $\Gh$ by the \emph{inverted} Smale quasi-flow (i.e., the
quasi-flow graded by $-\coc$ in which the names of projections $\proj_+$
and $\proj_-$ are swapped) we get the following corollary.

\begin{corollary}
Let $\Gh$ be a Smale quasi-flow and let $\coc:\Gh|_{\X}\arr\R$ be
the corresponding quasi-cocycle. Then there exists a quasi-cocycle
$\coc_-:\proj_-(\Gh)\arr\R$ such that the groupoid $\proj_-(\Gh)$
graded by $\coc_-$ is hyperbolic and $|\coc_-(\proj_-(g))+\coc(g)|$ is
uniformly bounded for all $g\in\Gh|_{\X}$.
\end{corollary}

\section{Duality theorems}

\begin{theorem}
\label{th:everyhypflowgeod}
Let $\Gh$ be a Smale quasi-flow. Then $\Gh$ is equivalent (as a
graded groupoid with a local product structure)
to the groupoid $\partial\proj_+(\Gh)\rtimes\proj_+(\Gh)$.
\end{theorem}

\begin{proof} We assume that $\Gh$, $\X$, $\X_0$, and $\mS$ satisfy
the conditions of Proposition~\ref{prop:generatorsflow} and use its notations.
Let us denote $\G=\proj_+(\Gh)$.

Let $\eta$ be such that $\coc$ is an $\eta$-quasi-cocycle.
Let $\Delta_1-\eta$ be an upper bound on the value of $|\coc|$ on elements
of $\mS$, and let $\nuke$ be a compact set containing all
elements of $\Gh$ such that $|\coc(g)|\le\Delta_1$ and $\be(g),
\en(g)\in\X'$, where $\X'$ is a compact neighborhood of $\X$.
Let us cover $\nuke$ by a finite set $\mathcal{U}$ of Lipschitz
rectangles $U\in\wt{\Gh}$. Let $\delta>0$ be such
that for every element $r\in\nuke$ there exists $U\in\mathcal{U}$ such
that $U\ni r$ and the $\delta$-neighborhoods of $\be(r)$ and
$\en(r)$ are contained in $\be(U)$ and $\en(U)$ respectively.

Let $\cdots g_2g_1$ be a composable sequence of germs of elements of
$\proj_+(\mS)$. Let $G_k\in\mS$ be such that $g_k\in\proj_+(G_k)$.
Then the homeomorphism $F_n=G_n\cdots G_1$
is a rectangle such that $\proj_+(\be(F_n))=\be(A_{G_1})$ and
$\proj_-(\en(F_n))=\en(B_{G_n})$ (here $A_F=\proj_+(F)$ and
$B_F=\proj_-(F)$, and $\be(A_F)$ and $\en(B_F)$ are sides of the
rectangles $W_i$ from Proposition~\ref{prop:generatorsflow}), see
Figure~\ref{fig:boundary}.

Diameter of $\proj_-(\be(F_n))$ is not more than $\lambda^n\cdot p$,
where $p$ is an upper bound of diameters of $\proj_-(R)$ for
$R\in\mathcal{R}$. The set $\proj_-(\be(G_n\cdots G_1))$ contains the
closure of the set $\proj_-(\be(G_{n-1}\cdots G_1))$. It follows that
for every infinite composable sequence $\ldots g_2g_1$ the intersection
\[\bigcap_n\proj_-(\be(G_n\cdots G_1))\] consists of
just one point $\xi$. Denote by $\Phi(\ldots, g_2, g_1)$ the point of
$\Gh^{(0)}$ such that \[\proj_+(\Phi(\ldots, g_2, g_1))=\be(g_1),\quad
\proj_-(\Phi(\ldots, g_2, g_1))=\xi.\]

\begin{figure}
\centering
\includegraphics{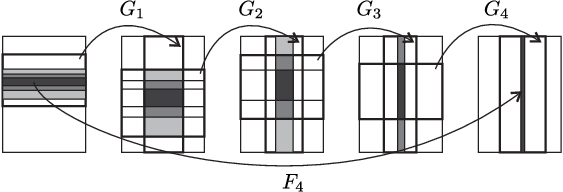}
\caption{}\label{fig:boundary}
\end{figure}

\begin{lemma}
\label{lem:Phiwelldefined}
Let $\ldots g_2g_1, \ldots h_2h_1$ be points of $\partial\G$, where
$g_i, h_i$ are germs of elements of $\proj_+(\mS)$. Suppose that
$g\in\G$ is such that $\ldots g_2g_1\cdot g=\ldots h_2h_1$. Then there
exists a unique element $\wt g\in\Gh$ such that
$\proj_+(\wt g)=g$, $\be(\wt g)=\Phi(\ldots, g_2, g_1)$, and
$\en(\wt g)=\Phi(\ldots, h_2, h_1)$.
\end{lemma}

\begin{proof}
By Lemma~\ref{lem:reductiontop} there exist sequences $G_i',
G_i''\in\mS$ and an element $r\in\nuke$ such that
$g=\proj_+(G_l''\cdots G_1''r)\proj_+(G_k'\cdots G_1')^{-1}$.
Let $G_i$, $H_i\in\mS$ be such that $g_i\in\proj_+(G_i)$ and
$h_i\in\proj_+(H_i)$. Let $U\in\mathcal{U}$ be such that $r\in U$.

By Lemma~\ref{l:relationsboundary} and Lemma~\ref{lem:reductiontop}
there exists a sequence of elements $r_i\in\nuke$ such that
$\proj_+(r_i)\cdot g_{n_i}\cdots g_1\cdot g=h_{m_i}\cdots
h_1$ for some increasing sequences $n_i$ and $m_i$.
Let $U_i\in\wt{\Gh}$ be such that $r_i\in U_i$.

We have
\[\proj_+(r_i)\cdot g_{n_i}\cdots g_1\cdot \proj_+(G_l''\cdots
G_1''r)=h_{m_i}\cdots h_1\cdot\proj_+(G_k'\cdots G_1')\]

Consider the products $U_i\cdot G_{n_i}\cdots G_1G_l''\cdots G_1''U$ and
$H_{m_i}\cdots H_1G_k'\cdots G_1'$. Denote
$z_1=(G_{n_i}\cdots G_1G_l''\cdots G_1''U)^{-1}(\be(r_i))$
and $z_2=(H_{m_i}\cdots H_1G_k'\cdots G_1')^{-1}(\en(r_i))$. Then
$\proj_+(z_1)=\proj_+(z_2)$, and the germ of
\[\proj_+((G_1')^{-1}\cdots (G_k')^{-1}H_1^{-1}\cdots
H_{n_i}^{-1}U_iG_{n_i}\cdots G_1G_l''\cdots G_1''U)\] at
$\proj_+(z_1)$ is trivial. It follows, by condition (4) of
Definition~\ref{def:hypflow}, that the germ of $(G_1')^{-1}\cdots
(G_k')^{-1}H_1^{-1}\cdots H_{n_i}^{-1}U_iG_{n_i}\cdots G_1G_l''\cdots G_1''U$
is trivial at $z_1=z_2$. Consequently, the domains
$\be(U_iG_{n_i}\cdots G_1G_l''\cdots G_1''U)$ and
$\be(H_{m_i}\cdots H_1G_k'\cdots G_1')$ have a
non-empty intersection. Passing to the
limit as $i\to\infty$ and using the fact that $\proj_-(G_i)$ and
$\proj_-(H_i)$ are expanding, we conclude that
there exists $y$ such that $\proj_+(y)=\proj_+(\be(r))$ and
\[y\in\be(H_n\cdots H_1G_k'\cdots G_1')\] for all $n$. Taking $r'=(U,
y)$, we get an element
$\wt g=G_l''\cdots G_1''r'(G_k'\cdots G_1')^{-1}$ such that $\be(\wt
g)=\Phi(\ldots, g_2, g_1)$, $\en(\wt g)=\Phi(\ldots, h_2, h_1)$ and
$\proj_+(\wt g)=g$.
Uniqueness of $g$ follows from the fact that $\Gh$ is locally diagonal
and the covering $\mathcal{R}$ is fine.
\end{proof}

Denote by $T$ the set of points of $\partial\G$ that can be
represented as $\ldots g_2g_1$ for $g_i\in\proj_+(G_i)$ for some
$G_i\in\mS$.

Applying Lemma~\ref{lem:Phiwelldefined} to the case
when $g$ is a unit, we conclude that
$\Phi(\ldots, g_2, g_1)$ depends only on $\ldots g_2g_1\in T$, and we get a
well defined map from $T$ to $\Gh^{(0)}$.

Moreover, Lemma~\ref{lem:Phiwelldefined} implies that $(\ldots g_2g_1,
g)\mapsto\wt g$ is a well defined map from
$(\partial\G\rtimes\G)|_T$ to $\Gh$. We will denote this map also by
$\Phi$, so that $\Phi(\ldots g_2g_1)=\Phi(\ldots, g_2, g_1)$ and, in
conditions of Lemma~\ref{lem:Phiwelldefined}, $\wt g=\Phi(\ldots
g_2g_1, g)$. It is easy to see (using uniqueness of $\wt g$) that
$\Phi:(\partial\G\rtimes\G)|_T\arr\Gh$ is a homomorphism of groupoids
(i.e., a functor of the corresponding small categories).

Note also that the homomorphism $\Phi$ is uniquely determined by its
restriction to $T$, since $\wt g=\Phi(\xi, g)$ is uniquely determined by
the condition $\be(\wt g)=\Phi(\xi)$, $\en(\wt g)=\Phi(\xi\cdot g)$,
and $\proj_+(\wt g)=g$.

The definition of $\proj_-(\Phi(\cdots g_2g_1))$ depends
only on the sequence $G_k\in\mS$ such that
$g_k\in\proj_+(G_k)$. It
follows that the homomorphism $\Phi$ agrees with the local product
structures on the geodesic quasi-flow and $\Gh$.

\begin{lemma}
\label{lem:injectivity}
Suppose that $h\in\Gh$ is such that $\be(h)=\Phi(\ldots g_2g_1)$
and $\en(h)=\Phi(\ldots h_2h_1)$. Then $\ldots g_2g_1=\ldots
h_2h_1\cdot\proj_+(h)$.
\end{lemma}

\begin{proof}
Let $G_i, H_i\in\mS$ be such that $g_i\in\proj_+(G_i)$ and
$h_i\in\proj_+(H_i)$. Denote $z=\be(h)$.
There exist increasing sequences $m_k$ and $n_k$ such
that \[|\coc((H_{m_k}\cdots H_1h)\cdot (G_{n_k}\cdots
G_1)^{-1})|\le\Delta_1.\]
Then $(H_{m_k}\cdots H_1)h(G_{n_k}\cdots
G_1)^{-1}\in\nuke$, and for \[r_k=\proj_+((H_{m_k}\cdots H_1)h
(G_{n_k}\cdots G_1)^{-1})\in\proj_+(\nuke),\] we have $r_k\cdot g_{n_k}\cdots
g_1=h_{m_k}\cdots h_1\proj_+(h)$, hence
$\cdots g_2g_1=\cdots h_2h_1\proj_+(h)$.
\end{proof}

\begin{lemma}
The homomorphism $\Phi$ is continuous.
\end{lemma}

\begin{proof}
It is enough to show that $\Phi$ is continuous on $T$, by the above
remark on uniqueness of the definition of $\wt g$. Moreover, it is enough to
show that it is continuous on each
$\partial\G_x$, since $\Phi$ agrees with the local product structure
and is obviously continuous on the first direction of the local
product structure of $\partial\G$.

Then the proof of continuity of $\Phi$ becomes essentially the same as of
Lemma~\ref{lem:Phiwelldefined}.
Suppose that $\ldots g_2g_1$ and $\ldots h_2h_1$ are close in
$\partial\G_x$, where $x=\be(g_1)=\be(h_1)$. Then there exists
$r\in\nuke$ and large indices $m$ and $n$ such that
$\proj_+(r)g_n\cdots g_1g=h_m\cdots h_1$. Then, as in the proof of
Lemma~\ref{lem:Phiwelldefined}, we conclude that the domains of
$G_n\cdots G_1$ and $H_m\cdots H_1$ are close to each
other, hence $\Phi(\ldots g_2g_1)$ and $\Phi(\ldots h_2h_1)$ are close.
\end{proof}

\begin{lemma}
The homomorphism $\Phi$ is injective, and the inverse partial map is
a continuous homomorphism.
\end{lemma}

\begin{proof}
Again, it is enough to show that it is injective on $T$ and that the
map inverse to $\Phi|_T$ is continuous.

Injectivity and functoriality
of the inverse follow directly
from Lemma~\ref{lem:injectivity}.
Continuity of $\Phi^{-1}|_T$ is proved in a way similar to
the proof of Lemma~\ref{lem:injectivity}.
Namely, suppose that distance between
$z_1=\Phi(\cdots g_2g_1)$ and $z_2=\Phi(\cdots
h_2h_1)$ is less that $\Lambda^{-n}\delta$, where
$\Lambda$ is such that the maps $\proj_-(G)$ for
$G\in\mS$ are $\Lambda$-Lipschitz. Then there exist $G_i, H_i\in\mS$
such that $g_i\in G_i$, $h_i\in H_i$, and
\[\{z_1, z_2\}\subset\be(G_k\cdots G_1)\cap\be(H_k\cdots H_1)\]
for all $k\le n$.
Then in the same way as in the proof of injectivity of $\Phi|_T$, we show
that for $n$ big enough there exist large indices $m_k$ and $n_k$ and
$r_k\in\proj_+(\nuke)$ such that $r_k\cdot g_{n_k}\cdots
g_1=h_{m_k}\cdots h_1$, which implies that
$\ldots g_2g_1$ and $\ldots h_2h_1$
are close to each other.
\end{proof}

\begin{lemma}
The set $\X_0$ belongs to the range of $\Phi$, and for every
$x\in\X_0$ the point $\Phi^{-1}(x)$ belongs to the interior of $T$.
\end{lemma}

\begin{proof}
By Proposition~\ref{prop:generatorsflow}, for every point $x\in\X$
there exists $h\in S$ such that $\be(h)=x$ and
$\en(h)\in\X$. Applying this fact infinitely many times we will get
a composable sequence $\ldots h_2h_1$ of elements of $S$ such that
$\be(h_1)=x$. Then $\Phi(\ldots\proj_+(h_2)\proj_+(h_1))=x$, hence
$\X\supset\X_0$ belong to the range of $\Phi$.

Let $\xi=\ldots g_2g_1$, and let $G_i\in\mS$ be such that
$g_i\in\proj_+(G_i)$. Suppose that $\Phi(\xi)\in\X_0$.

There exists a compact set $Q_1\subset\G|_{\X}$ such that
the set $N_n$ of points $\zeta\in\partial\G_{\be(g_1)}$ such that
$\ell_{\be(g_1)}(\zeta, \xi)\ge n$ is contained in the set of points
representable in the form $\ldots g_{n+2}'g_{n+1}'ag_ng_{n-1}\cdots
g_1$, where $a\in Q_1$ and $g_i'\in\proj_+(S)$.

It follows then
from Lemma~\ref{lem:reductiontop} that there exists a constant $k_0$
such that the set $N_n$ is contained in the set of points of the form
\[\zeta=\ldots\proj_+(h_2)\proj_+(h_1)\proj_+(r)g_{n-k_0}\cdots g_1\]
for some composable sequence $\ldots h_2h_1r$ of elements $h_i\in
S$ and $r\in\nuke$. We can find a sequence $H_0, H_{-1}, \ldots,
H_{-m}$ of elements of $\mS$ such that $\en(r)\in\en(H_0\cdots
H_{-m})$ and $r'=(H_0\cdots H_{-m})^{-1}rG_{n-k_0}\cdots
G_1\in\nuke$. Let $h_0, h_{-1}, \ldots, h_{-m}$ be the germs of $H_0,
H_{-1}, \ldots, H_{-m}$ such that $h_0\cdots h_{-m}$ is defined and
$\en(h_0\cdots h_{-m})=\en(r)$. Then
$\en(r')=\Phi(\ldots\proj_+(h_2)\proj_+(h_1)\proj_+(h_0)\cdots\proj_+(h_{-m}))$,
and distance between $\be(r')$ and $\Phi(\zeta)$ is not more than
$\lambda^{n-k_0}p$. It follows that $\be(r')$ belongs to
$\X_0\subset\Phi(T)$, for all $n$ big enough.
Then $\be(r')=\Phi(\ldots h_2'h_1')$ for some
$h_i'\in\proj_+(S)$,
\begin{multline*}
\zeta=\ldots\proj_+(h_2)\proj_+(h_1)\proj_+(r)g_{n-k_0}\cdots g_1=\\
\ldots\proj_+(h_2)\proj_+(h_1)\proj_+(h_0)\cdots\proj_+(h_{-m})\proj_+(r'),
\end{multline*}
and
\[
\en(r')=\Phi(\ldots\proj_+(h_2)\proj_+(h_1)\proj_+(h_0)\cdots\proj_+(h_{-m})).
\]
Then by Lemma~\ref{lem:injectivity} we have
$\zeta=\ldots h_2'h_1',$
hence $\zeta$ belongs to $T$. We have shown that all points of $N_n$
belong to $T$ for all $n$ big enough. It follows
that $\xi$ is an internal point of $T$.
\end{proof}

We have shown that $\Phi$ is an isomorphism of the restrictions of the
groupoids $\partial\G\rtimes\G$ and $\Gh$ to $T$ and $\Phi(T)$ respectively.
We have also shown that $\Phi^{-1}$ maps $\X_0$ to an open
subset of $T$. Consequently, $\Phi$ implements an
isomorphism of restrictions of $\Gh$ and $\partial\G\rtimes\G$ to
open transversals, which implies that $\Gh$ and $\partial\G\rtimes\G$
are equivalent.
\end{proof}

Suppose that $\G$ is a hyperbolic groupoid such that
$\partial\G\rtimes\G$ is locally diagonal. Then
$\proj_+(\partial\G\rtimes\G)$ is equivalent to $\G$ (which
follows directly from the definition of the local product structure on
$\partial\G$, see Theorem~\ref{th:localproduct}).

\begin{defi}
Let $\G$ be a hyperbolic groupoid such that $\partial\G\rtimes\G$ is
locally diagonal. Then the \emph{dual groupoid} \index{dual groupoid}
\index{groupoid!dual} $\G^\top$ is the projection
$\proj_-(\partial\G\rtimes\G)$.
\end{defi}

Note that the grading of $\proj_-(\partial\G\rtimes\G)$ is projection
of the quasi-cocycle $-\wt\coc$, where $\wt\coc$ is the lift of $\coc$
to $\partial\G\rtimes\G$ given by~\eqref{eq:liftofcoc}.

The following theorem is a direct corollary of
Theorems~\ref{th:geodflowhyperbolic},~\ref{th:Ruellehyperbolic},
and~\ref{th:everyhypflowgeod}.

\begin{theorem}
\label{th:dualitytheorem}
Let $\G$ be a hyperbolic groupoid with locally diagonal geodesic quasi-flow.
Then its dual $\G^\top$ is also a hyperbolic groupoid and
the groupoid $(\G^\top)^\top$ is equivalent to $\G$.
\end{theorem}

Note also that for a graded hyperbolic groupoid $(\G, \coc)$ there is
a well defined, up to strong equivalence,
quasi-cocycle $\coc^\top$ on $\G^\top$. Namely, it is shown in the
proof of Theorem~\ref{th:geodflowhyperbolic} that the lift $\wt\coc$ of $\coc$
to the geodesic flow $\partial\G\rtimes\G$ is a quasi-cocycle
satisfying the definitions of the quasi-cocycle on a Smale
quasi-flow. Theorem~\ref{th:Ruellehyperbolic} shows that there are
unique, up to strong equivalence, quasi-cocycles $\coc_+$ and $\coc_-$
which are projections of $\wt\coc$ onto the Ruelle groupoids of
$\partial\G\rtimes\G$. Moreover, $\coc_+$ is strongly equivalent to
$\coc$. Then we can take
$\coc^\top=-\coc_-$. Theorem~\ref{th:everyhypflowgeod} shows that
this construction is a duality of graded groupoids, i.e., that
applying it twice we get back the original graded groupoid $(\G, \coc)$.

\section{Other definitions of the dual groupoid}
\label{s:anotherdefinition}

Let $\G$ be a hyperbolic groupoid. Let $(S, \X)$ be a complete
generating pair of $\G$ (see Proposition~\ref{pr:hyperbolicgenset}).
Let $\mS$ be a finite covering of $S$ by contracting positive elements
of $\pG$.

We denote by $\half$ the union $\G_x^{\X}\cup\partial\G_x=\G(x,
S)\cup\partial\G(x, S)\setminus\{\omega_x\}$ \index{Gx@$\half$}
of the set of vertices
$\G_x^{\X}$ of the Cayley graph $\G(x, S)$ with the boundary
$\partial\G_x$. The set $\half$ comes with the topology
defined by the natural log-scale on $\half$ (see
Proposition~\ref{pr:topology}).

Let $A\subset\G$ be a compact set satisfying the conditions of
Proposition~\ref{pr:nbhdtil}.
Suppose that for any two sequences $g_i, h_i$ of
germs of elements of $\mS$ an equality $\ldots g_2g_1\cdot g=\ldots
h_2h_1\cdot h$ for some $g, h$, $\be(g)=\be(h)\in\X$ implies
that for all sufficiently big $n$ there exists $m$ and $a\in A$ such that
$ag_n\cdots g_1g=h_m\cdots h_1h$. Existence of such a set $A$ follows
from hyperbolicity of the Cayley graphs of $\G$ and the fact that all
directed paths in $\G(x, S)$ are quasi-geodesics.

Find a finite covering $\mathcal{A}=\{U\}$ of $A$ by bi-Lipschitz
elements of $\pG$. Let $\wh A$ be the set of germs of the
elements of $\mathcal{A}$.

The following lemma is proved in the same
way as~\ref{lem:relationsbndm}.

\begin{lemma}
\label{lem:deltabijection}
Let $\epsilon$ be a common Lebesgue's number of the coverings $\mS$, $\mathcal{A}$, and $\mathcal{A}^{-1}$ of $S$, $A$, and $A^{-1}$ respectively.
There exists $0<\delta_0<\epsilon$ such that the following condition is satisfied.

Let $U_i, V_i$, $i=1, 2, \ldots$ be finite or infinite sequences of
elements of the set $\mS\cup\mathcal{A}$ in which at most one
element belongs to $\mathcal{A}$. Let $|x-y|<\delta_0$ for $x,
y\in\X$, and the $\epsilon$-neighborhoods of $U_i\cdots U_1(x)$ and
$V_i\cdots V_1(x)$ are contained in $\be(U_{i+1})$ and $\be(V_{i+1})$ respectively.
Then the equality
\[(\ldots U_2U_1, x)=(\ldots V_2V_1, x)\]
of finite or infinite products of germs implies
\[(\ldots U_2U_1, y)=(\ldots V_2V_1, y).\]
\end{lemma}

Fix $\delta_0$ satisfying the conditions of
Lemma~\ref{lem:deltabijection}. Suppose that $g\in\G|_{\X}$ and
$h\in\G$ are such that $|\en(g)-\en(h)|<\delta_0$. For a finite or
infinite product
$\xi=\ldots g_2g_1g\in\overline{T_g}$, where $g_i\in S$,
find elements $U_i\in\mS$ such that $g_i$ is $\epsilon$-contained in $U_i$.
Define then
\begin{equation}\label{eq:rgh}R_g^h(\xi)=\ldots U_2U_1\cdot
  h,\end{equation}
see Figure~\ref{fig:dualgroupoid}.
By Lemma~\ref{lem:deltabijection}, $R_g^h(\xi)$ depends only on
$g$, $h$, and $\xi$ (and does not depend on the choice of the generators
$g_i$ or the choice of the elements $U_i$). Note that $R_g^h(\xi)
\notin\G|_{\X}$ in general (even for $\xi\in\G_x^{\X}$).

\begin{figure}
\centering
\includegraphics{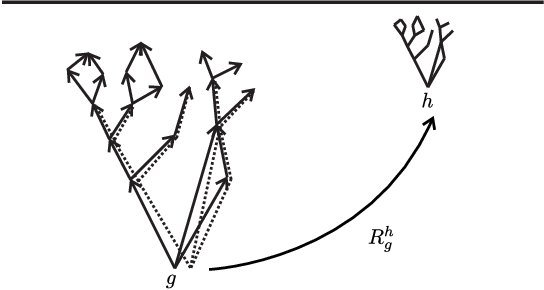}
\caption{The map $R_g^h$}
\label{fig:dualgroupoid}
\end{figure}

\begin{theorem}
\label{th:holonomies}
The space $\bdry$ of germs of restrictions of the maps $R_g^h$,
for $g\in\G|_{\X}, h\in\G$, to open subsets of the disjoint union
$\bigsqcup_{x\in\X}\partial\G_x$ is a
groupoid (i.e., is closed under taking compositions and inverses), and
depends only on $\G$ and $\X$. Restriction of $\bdry$ to any set
$\bigsqcup_{x\in Y}\partial\G_x$ for $Y\subset\X$ does
not depend on $\X$.
\end{theorem}

\begin{proof}
If $x, y$ are units of $\G$ such that $R_x^y$ is defined, then
restriction of the transformation $R_x^y$ to $\partial\G_x$ is equal
to the transformation $\xi\mapsto [y, \xi]_V$, where $V$ is a
neighborhood of points $x, y$ of diameter less than
$\delta_0$. See~\eqref{eq:locprodU} on page~\pageref{eq:locprodU}.

The groupoid $\mathfrak{H}$
of germs of the pseudogroup acting on the disjoint union
$\bigsqcup_{x\in\X}\partial\G_x$ and generated by the transformations
$[\cdot, \cdot]_V$ and $\xi\mapsto\xi\cdot g$ is equivalent to the
projection $\proj_-(\partial\G\rtimes\G)$ of the geodesic quasi-flow onto
the direction of the boundaries of the Cayley graphs. Every element of the groupoid $\mathfrak{H}$
is, by Lemma~\ref{lem:reductiontop}, equal to a composition of
$\xi\mapsto\xi\cdot g$, followed by $\xi\mapsto [x, \xi]_U$, and then
by $\xi\mapsto\xi\cdot h$ for some $g, h\in\G$.

Restriction of the transformation $R_g^h$ to an open subset of
$\partial\G_{\be(g)}$ is equal to the composition of the transformations
\[\xi\mapsto\xi\cdot g^{-1}\mapsto [\en(h), \xi\cdot
g^{-1}]_V\mapsto[\en(h), \xi\cdot g^{-1}]_V\cdot h,\]
where $V$ is a neighborhood of diameter less than $\delta_0$ of
$\en(g)$ in $\G^{(0)}$.

It follows that the set of germs of $R_g^h$ is equal to the groupoid
$\mathfrak{H}$.
The last statement of the theorem is straightforward.
\end{proof}

\begin{proposition}
\label{prop:newdual}
Let $\G$ be a hyperbolic groupoid. The dual groupoid $\G^\top$
is equivalent to $\bdry$.
\end{proposition}

\begin{proof}
Follows directly from the proof of
Theorem~\ref{th:holonomies}.
\end{proof}

Below we describe a reformulation of the last proposition, which is less
explicit, but is probably more elegant and self-contained.

Let $\G$ be a hyperbolic groupoid. Let $(S, \X)$ be any compact
generating pair of $\G$. Let
$S'$ and $\X'$ be compact neighborhoods of $S$ and $\X'$
respectively. Let $|\cdot|$ be a metric on $S'$.

In the following definition we write
$g_n\to\infty$ if $g_n$ eventually leaves every finite set.

\begin{defi}
We say that a map $F:T_1\arr T_2$, where $T_1\subset\half$ and
$T_2\subset\overline{\G_y^{\X'}}$ are compact
neighborhoods of points $\xi_1\in\partial\G_x, \xi_2\in\partial\G_y$, is an
\emph{asymptotic automorphism of $\G$}, if the following condition holds.
If sequence $g_n, h_n\in T_1$ are such that $g_n\to\infty$,
$h_n\to\infty$, and $g_n^{-1}h_n\in S$, then
$F(g_n)F(h_n)^{-1}$ eventually belongs to $S'$,
and \[|g_nh_n^{-1}-F(g_n)F(h_n)^{-1}|\to 0\]
as $n\to\infty$.
\end{defi}

Note that the condition of the definition does not depend on the
choice of the metric $|\cdot|$ on $S'$, since $S'$ is compact.

\begin{proposition}
\label{pr:asautom}
The set of germs of partial maps $\partial\G_x\arr\partial\G_y$
equal to restrictions of asymptotic automorphisms of $\G$ coincides with $\bdry$ and is
hence equivalent to $\G^\top$.
\end{proposition}

\begin{proof}
It is easy to check that the set of asymptotic automorphisms does not
depend on the generating sets $S$ and $S'$, hence we may assume that
they satisfy the conditions of
Proposition~\ref{pr:hyperbolicgenset} (it depends, however, on $X$ and
$X'$).

Let $\mS$ be, as in the proof of Theorem~\ref{th:localproduct}, a
finite covering of $S$ by elements of $\pG$. Let $\delta$ be a number
smaller than the Lebesgue's number of the covering $\mS$.
Then $|g_nh_n^{-1}-F(g_n)F(h_n)^{-1}|\to 0$ implies that for all
$n$ big enough we have $F(g_n)F(h_n)^{-1}\in U_n$, where $U_n\in\mS$ are
such that $g_nh_n^{-1}$ is $\delta$-contained in $U_n$. It follows
that every asymptotic automorphism, up to a finite set can
be covered by transformations of the form $R_g^h$.

On the other hand, it is easy to see that every transformation $R_g^h$ is
an asymptotic automorphism. Consequently, the set of germs of
asymptotic automorphisms on the boundaries of the Cayley graphs
coincides with the set of germs of the maps $R_g^h$, i.e., with $\bdry$.
\end{proof}

\section{Minimal hyperbolic groupoids}

Let $\G$ be a hyperbolic groupoid and let $(S, \X)$ be its
generating pair. We
say that a Cayley graph $\G(x, S)$ is \emph{topologically mixing} if
for every point $\xi\in\partial\G_x$ and every neighborhood $U$ of
$\xi$ in $\half$ the set of accumulation points
of $\en(U\cap\G_x^{\X})$ contains the interior of $\X$.

\begin{proposition}
\label{pr:minimal}
Let $\G$ be a hyperbolic groupoid. Then the following conditions are
equivalent.
\begin{enumerate}
\item Some Cayley graph of $\G$ is topologically mixing.
\item Every Cayley graph of $\G$ is topologically mixing.
\item Every $\G$-orbit is dense in $\G^{(0)}$.
\end{enumerate}
\end{proposition}

\begin{proof}
Note that (2) obviously implies (3). It remains to prove that (1)
implies (2) and that (3) implies (1).

Let us show that (1) implies (2). We will split the proof into four lemmas.

\begin{lemma}
If $\G(x, S)$ is topologically mixing, then $\G(x, S')$ is topologically mixing for any generating set $S'$ of $\G|_{\X}$.
\end{lemma}

\begin{proof} The Lipschitz class of the log-scale on $\half$ does not depend on the choice of the generating set, hence the set of open
neighborhoods of points of $\half$ does not depend on $S$.
\end{proof}

We may assume now that all generating pairs in our proof satisfy the conditions of Proposition~\ref{pr:hyperbolicgenset}.

\begin{lemma}
Let $(\X', S')$ be a compact generating pair of $\G$ such that $\G(x, S')$ is topologically mixing. Let $(\X, S)$ be a generating pair such that $\X\subset\X'$.
Then $\G(x, S)$ is topologically mixing.
\end{lemma}

\begin{proof}
For every neighborhood $W\subset\half$ of a point
$\xi\in\partial\G_x$ and every point
$y$ of the interior of $\X$, the point $y$ is a limit of a sequence
of $\en(g_n)$ for pairwise different $g_n\in W$.
Since $y$ is an internal point of $\X$, for all sufficiently big $n$
we will have $\en(g_n)\in\X$, hence
$g_n\in\G|_{\X}$. Consequently, the set of accumulation points of
$\en(W\cap\G|_{\X})$ contains the interior of $\X$, which shows that
$\G(x, S)$ is topologically mixing.
\end{proof}

\begin{lemma}
\label{l:increasing}
Let $(\X, S)$ and $(\X', S')$ be compact generating pairs of $\G$ such that $\X\subset\X'$ and $\G(x, S)$ is topologically mixing. Then $\G(x, S')$ is topologically mixing.
\end{lemma}

\begin{proof}
The identical embedding $\G_x^{\X}\arr\G_x^{\X'}$ is a quasi-isometry
of the Cayley graphs $\G(x, S)\arr\G(x, S')$.

There exists a finite collection
$\mathcal{U}$ of elements $U\in\pG$ such that
$\be(U)$ cover $\X'$ and $\en(U)\subset\X$. Denote
$A=\bigcup_{U\in\mathcal{U}}U\cap\G|_{\X'}$.
It follows from elementary properties of
Gromov hyperbolic graphs that for every neighborhood $W'$ of
$\xi\in\partial\G_x$ in  $\half[\X']$ there exists a
neighborhood $W$ of $\xi$ in
$\half$ such that for every $g\in W$ and $h\in A$
we have $h^{-1}g\in W'$. Let $y$ be an arbitrary internal point of
$\X'$. Let $U\in\mathcal{U}$ be such that $y\in\be(U)$. Then $U(y)$
is an internal point of $\X$, and hence it is a limit of a sequence
of points of the form $\en(g_n)$ for a sequence of different elements
$g_n\in W$. Then $\en(g_n)$ belongs to $\en(U)$
and $U^{-1}(\en(g_n))$ belongs to $\X'$ for all sufficiently big
$n$. If $h$ is the germ of $U$ at $U^{-1}(\en(g_n))$, then $h\in A$,
$h^{-1}g_n$ belongs to $W'$, and $y$ is equal to the limit of the
sequence $\en(h^{-1}g_n)$. Consequently, the Cayley graph $\G(x, S')$
is also topologically mixing.
\end{proof}

The following lemma will finish the proof that (1) implies (2).

\begin{lemma}
If a Cayley graph $\G(x, S)$ is topologically
mixing for some
$x\in\X$, then the Cayley graph $\G(x', S)$ is topologically mixing
for every $x'\in\X$.
\end{lemma}

\begin{proof}
Since the property of being topologically mixing does not depend on
the choice of the generating set $S$, we may assume that $S$ is complete. Let $\mS$ be
a finite open covering of $S$ by positive contracting elements of $\pG$.

Note that if $\G(x, S)$ is topologically mixing, then for any
point $y\in\X$ in the orbit of $x$ the Cayley graph $\G(y, S)$ is
topologically mixing. It follows that we may assume that $x'$ belongs
to an open transversal $\X_0\subset\X$ and that $x$ is arbitrarily
close to $x'$ (since the orbit of $x$ is dense in the interior of
$\X$).

Let $\xi'\in\partial\G_{x'}$ be an arbitrary point. Let $X'\supset \X$ be a compact set such that $\en(F)$ is contained in $\X'$ for every $F\in\mS$.
Let $W$ be a neighborhood of $\xi'$ in
$\overline{\G_{x'}^{\X'}}$. Note that then the set of values of $R_x^{x'}$ belongs to $\overline{\G_{x'}^{\X'}}$. Denote $\xi=R_{x'}^x(\xi')$. We have $R_x^{x'}(\xi)=\xi'$. There exists
$g\in\G_x$ such that $\overline{T_g}$ is a neighborhood of $\xi$ in
$\partial\G_x$ and $R_x^{x'}(\til_g)\subset W\cap\partial\G_{x'}$.
The set $\overline{T_g}$ is a neighborhood of
$\xi$ in $\half$, hence $\en(T_g)$ is dense in the
interior of $\X$. Denote $T_g'=R_x^{x'}(T_g)$. The set of accumulation points of $T_g$
in $\half$ is equal to $\til_g$. The set of accumulation points of $T_g'$ in $\half[X']$ is $R_x^{x'}(\til_g)$. Consequently, all but finitely many points of $T_g'$ belong
to $W$. But the set of accumulation points of $\en(T_g)$ is equal to
the set of accumulation points of $\en(T_g')$,
due to the fact that elements of $\mS$ are contracting.
Consequently, the set of accumulation points of
$\en(W\cap\G_{x'})$ contains the set of accumulation points of
$\en(T_g)$, hence it contains the interior of $\X$. We have shown that for every neighborhood $W$ of $\xi'$ in $\overline{\G_{x'}^{\X'}}$ the set of accumulation points of $\en(W\cap\G_{x'})$ contains the interior of $\X$. Repeating the argument from the proof of Lemma~\ref{l:increasing}, we conclude that the set of accumulation points of $\en(W\cap\G_{x'})$ contains the interior of $\X'$.
\end{proof}

Finally, let us show that (3) implies (1). Let $\G(x, S)$ be a Cayley
graph, where $(S, \X)$ is a generating pair.
We have to show that
$\en(\G_x^\X)$ is dense in the interior of $\X$. Let
$U\subset\X$ be an arbitrary open subset. Since every orbit of $\G$
is dense, $U$ is a transversal. Consequently, the set of elements
$g\in\G_x^\X$ such that $\en(g)\in U$
is a net in the Cayley graph $\G(x, S)$. This implies that
for every neighborhood $W$ of a point of $\partial\G_x$ in $\half$,
the set $\en(W)$ intersects $U$.
\end{proof}

\begin{defi}
We say that a hyperbolic groupoid $\G$ is \emph{minimal}
\index{groupoid!minimal} \index{minimal groupoid}
if it satisfies the equivalent conditions of Proposition~\ref{pr:minimal}.
\end{defi}

\begin{proposition}
\label{pr:dualgroupoid}
If $\G$ is minimal, then the groupoid $\bdry[x]$ equal to
restriction of $\bdry$ to $\partial\G_x$ is equivalent to the
groupoid $\bdry$.
\end{proposition}

\begin{proof}
The set $\partial\G_x$ is an open subset of the space of units of the
groupoid $\bdry$. It follows from minimality that for every
$y\in\G^{(0)}$ there exist elements $g\in\G$ and $x\in\X$ such that
$\be(g)=y$ and $|\en(g)-x|<\delta_0$. Consequently, $\partial\G_x$
is a $\bdry$-transversal.
\end{proof}

Hence, we can define, in the case of minimal groupoids,
the dual groupoid $\G^\top$ as the groupoid of
germs of restrictions of the maps $R_g^h$ to open subsets of $\partial\G_x$.

\begin{proposition}
If $\G$ is minimal, then $\G^\top$ is also minimal.
\end{proposition}

\begin{proof}
It is enough to show that the groupoid $\bdry[x]$
is minimal. By definition of topologically mixing Cayley graphs, for
any $g, h\in\G_x^{\X_0}$ and any neighborhood $U$ of $\en(g)$
there exists $f\in T_h$ such that $\en(f)\in U$ and
$\til_f\subset\til_h$.
Then restriction of $R_g^f:\til_g\arr\til_h$ to the interior of
$\til_g$ is an element of $\wt{\bdry}$. It
follows that orbit of any internal point of $\til_g$ is dense.
\end{proof}

For definition of a locally diagonal groupoid,
see~\ref{def:locdiagonal}.

\begin{proposition}
\label{pr:irredlocdiag}
If $\G$ is minimal, then the geodesic quasi-flow
$\partial\G\rtimes\G$ is locally diagonal.
\end{proposition}

\begin{proof}
Consider the covering of $\partial\G\rtimes\G$ by the rectangles
$R_{U, g}$, as defined in the proof of~Theorem~\ref{th:localproduct}.
Suppose that this covering does not satisfy the conditions
of Definition~\ref{def:locdiagonal}.
Then there exist a rectangle $R_{U, g}$ and a non-unit element $h\in\G$
such that $[\be(h), \xi]_U=\xi\cdot h$ for all $\xi\in\til_g^\circ$.

Let $\xi=\ldots g_2g_1\cdot g$ for $g_i\in S$ be an arbitrary point of
$\til_g^{\circ}$, and let $G_i\in\mS$ be such that
$g_i$ is $\epsilon$-contained in $G_i$. Then
$\be(h), \en(h)\in\be(U)$, and $[\be(h),
\xi]_U$ is represented by the product $\ldots G_2G_1\cdot (U,
\be(h))$. If $g'=(U, \be(h))$ and $g_i'$ are the germs of $G_i$ such
that $\ldots G_2G_1\cdot (U, \be(h))=\ldots g_2'g_1'\cdot g'$, then
$\ldots g_2g_1\cdot gh=\ldots g_2'g_1'\cdot g'$.

There exists a constant $\Delta$
such that for any element $h\in\G$ which has trivial projection on the
second direction of the geodesic quasi-flow we have
$|\coc(h)|\le\Delta$ (see~\eqref{eq:nataction} on page~\pageref{eq:nataction}).
Let $\nuke$ be the set of elements $(\xi, h)$ of the
geodesic quasi-flow such that $\xi\in\til_{\en(h)}$, $\xi\cdot
h\in\til_{\be(h)}$, and $|\coc(h)|\le\Delta$. Then $\nuke$ has compact
closure.

Consider the sequence
\begin{gather*}
h_0=gh(g')^{-1}=UhU^{-1},\\
h_1=g_1gh(g_1'g')^{-1}=G_1Uh(G_1U)^{-1},\\ \vdots\\
h_n=g_n\cdots g_1gh(g_n'\cdots g_1'g')^{-1}=G_n\cdots G_1Uh(G_n\cdots
G_1U)^{-1}.
\end{gather*}
For each element $h_n$ we have \[\ldots
g_{n+2}g_{n+1}\cdot h_n=\ldots g_{n+2}'g_{n+2}'\]
and
\[[\be(h_n), \ldots g_{n+2}g_{n+1}]_{G_n\cdots G_1U}=\ldots
g_{n+2}g_{n+1}.\]
There exists $n_0$ such that $\til_{g_n\cdots g_1\cdot
  g}\subset\til_g^{\circ}$ for all $n\ge n_0$. It follows that all
elements $h_n$ have trivial projections onto the second coordinate of
the geodesic quasi-flow. Note that all elements $h_n$ are non-units and
belong to $\nuke$ for $n\ge 1$. Let $H$ be the union of the sequences
$h_n$ for all possible choices of $\xi\in\til_g^{\circ}$,
representations $\xi=\ldots g_2g_1\cdot g$, and $G_i\in\mS$.

Then, by definition of minimality, the sets $\be(H)$ and
$\en(H)$ are dense in the interior of $\X$. Note also that
$|\be(h_n)-\en(h_n)|\to 0$ as $n\to\infty$.

The set $\nuke$ can be covered by a finite set $\mathcal{U}$
of extendable homeomorphisms $U\in\pG$. Choose for each $U\in\mathcal{U}$ an extension
$\wh U\in\pG$ such that $\overline U\subset\wh U$.
For every $U\in\mathcal{U}$ the set
of points $x\in\be(\overline{U})$ such that the germ $(\wh U, x)$ is
non-trivial and $\overline U(x)=x$ is nowhere dense. It follows that
there exists an internal point $x$ of $\X$ such that for every
$U\in\mathcal{U}$ either $(U, x)$ is trivial (hence $x$ belongs to an
open subset of the set of fixed points of $U$), or $U(x)\ne x$.
Then there exist numbers $r>0$ and $\epsilon_0>0$
such that for every $g\in\nuke$ such that $|x-\be(g)|<r$ either $g$ is
trivial, or $|\be(g)-\en(g)|>\epsilon_0$.

But then we get a contradiction, since there will exist $h_n\in H$
such that $|x-\be(h_n)|<r$ and $|\be(h_n)-\en(h_n)|<\epsilon_0$.
\end{proof}

\chapter{Examples of hyperbolic groupoids and their duals}

\section{Gromov hyperbolic groups}

Let $G$ be a Gromov hyperbolic group, i.e., a finitely generated group
such that its Cayley graph is hyperbolic. The group $G$ acts from
the right on its left Cayley graph by isomorphisms, hence it acts on the
boundary $\partial G$ of the Cayley graph by homeomorphisms. Suppose
that for any non-trivial element $g\in G$ and for any $\xi\in\partial
G$ the germ $(g, \xi)$ is non-trivial. This is true, for instance,
when $G$ is torsion-free.

Let $\G$ be the groupoid of germs of the action, which will coincide
with the groupoid $\partial G\rtimes G$ of the action. It is generated
by the compact set of germs of generators of $G$. In order to maintain
correct order of multiplication, we will denote by $(g, \xi)$ the germ
of the transformation $\zeta\mapsto \zeta\cdot g^{-1}$ at $\xi\in\partial G$.

\begin{proposition}
The groupoid $\G$ of the action of $G$ on $\partial G$ is hyperbolic.
\end{proposition}

\begin{proof}
Fix a finite generating set $A$ of $G$ and let $|\cdot|$ be the
word length function on $G$ defined by $A$. We assume that $A=A^{-1}$
and $1\in A$. Then distance between $g$ and $h\in G$ is equal to
$|gh^{-1}|=|hg^{-1}|$.

Let us use the standard log-scale on $\partial G$ equal to the Gromov product $\ell(\xi_1, \xi_2)$ computed with respect to the basepoint $1$. Recall, that it is given by
\[\ell(\xi_1, \xi_2)\doteq\lim_{n\to\infty}\frac 12(
|g_n|+|h_n|-|g_nh_n^{-1}|),\]
where $g_n\to\xi_1$ and $h_n\to\xi_2$ as $n\to\infty$.

Define $\coc(g, \xi)$ as the value of the Busemann quasi-cocycle
$\beta_\xi(g, 1)$, i.e., as any partial limit of the sequence
$|g_ng^{-1}|-|g_n|$, for $g_n\to\xi$ (see Figure~\ref{fig:busemann}).

\begin{figure}
\includegraphics{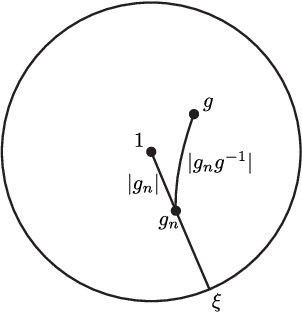}
\caption{}\label{fig:busemann}
\end{figure}

We have then
\[\coc(g, \xi)\doteq-2\ell(g, \xi)+|g|.\]

Let $\delta_2$ be a constant such that for any triple $z_1, z_2,
z_3\in G\cup\partial G$ we have
\[\ell(z_1, z_3)>\min(\ell(z_1, z_2), \ell(z_2, z_3))-\delta_2.\]

Suppose that $\zeta$ is sufficiently close to $\xi$, so that
\begin{equation}\label{eq:ellzetaxi}
\ell(\zeta, \xi)>\ell(g, \xi)+\delta_2.
\end{equation}
Then
\[\ell(g, \zeta)>\min(\ell(g, \xi), \ell(\zeta, \xi))-\delta_2=\ell(g,
\xi)-\delta_2,\]
and
\[\ell(g, \xi)>\min(\ell(g, \zeta), \ell(\zeta, \xi))-\delta_2,\]
which implies that
\[\ell(g, \xi)>\ell(g, \zeta)-\delta_2,\]
since the other case $\ell(g, \xi)>\ell(\zeta, \xi)-\delta_2$ is not
possible, due to~\eqref{eq:ellzetaxi}. Consequently,
\[|\coc(g, \xi)-\coc(g, \zeta)|\doteq 2|\ell(g, \zeta)-\ell(g,
\xi)|<2\delta_2,\]
hence there exists a constant $\delta_3$ such that
\[|\coc(g, \xi)-\coc(g, \zeta)|<\delta_3\]
for all $\xi, \zeta\in\partial G$ such that $\ell(\xi, \zeta)>\ell(g^{-1},
\xi)+\delta_2$. In particular, we conclude that $\coc:\G\arr\R$ is
locally bounded.

Suppose that $(g, \xi)\cdot (h, \zeta)$ is a composable pair of
elements of $\partial G\rtimes G$, i.e., that $\zeta\cdot h^{-1}=\xi$. Then
$(g, \xi)\cdot (h, \zeta)=(gh, \zeta)$, and
\begin{multline*}
\coc(\xi, gh)\doteq\lim_{h_n\to\zeta}|h_nh^{-1}g^{-1}|-|h_n|=\\
\lim_{h_n\to\zeta}(|h_nh^{-1}g^{-1}|-|h_nh^{-1}|)+\lim_{h_n\to\zeta}(|h_nh^{-1}|-|h_n|)=\\
\lim_{g_n\to\xi}(|g_ng^{-1}|-|g_n|)+\lim_{h_n\to\zeta}(|h_nh^{-1}-|h_n|)\doteq\\
\coc(\xi, g)+\coc(\zeta, h),
\end{multline*}
i.e., $\coc$ is a quasi-cocycle.

It follows from the definitions and right-invariance of the metric $|\cdot|$
that
\[\ell(\xi\cdot g^{-1}, \zeta\cdot g^{-1})-\ell(\xi, \zeta)\doteq\frac
12(\coc(g, \xi)+\coc(g, \zeta)).\]

Consequently, if $\xi$ and $\zeta$ are close enough, so that
$\ell(\xi, \zeta)>\ell(\xi, g)+\delta_2$, then we have
\[\ell(\xi\cdot g^{-1}, \zeta\cdot g^{-1})\doteq\ell(\xi,
\zeta)+\coc(g, \xi),\]
since $|\coc(g, \zeta)-\coc(g, \xi)|<\delta_3$.

It follows that $g$ is contracting on a neighborhood of a point
$\xi\in\partial G$ if $\coc(g, \xi)$ is big enough.

Therefore, it is sufficient to find for every $c>0$ a compact generating set of $\G$
consisting of germs $(g, \xi)$ such that $\coc(g, \xi)>c$. It will
satisfy then conditions of Definitions~\ref{def:hyperbolic},
which will show that $\G$ is hyperbolic.

Let $\delta$ be such that any geodesic triangle with vertices in
$G\cup\partial G$ is $\delta$-thin. Let $\delta_4$ be such that for
any $\xi\in\partial G$, $g\in G$, and a sequence $g_n\in G$ such that
$g_n\to\xi$, difference between any two partial limits of
$|g_ng^{-1}|-|g_n|$ is less than $\delta_4$. 

We say that an infinite product $\ldots g_2g_1$ of elements of $A$ is
a geodesic path converging to $\xi$ if the sequence $1, g_1, g_2g_1,
g_3g_2g_1, \ldots$ is a geodesic path converging to $\xi$. For every
point $\xi\in\partial G$ there exists a geodesic path $\ldots g_2g_1$
converging to $\xi$.

Let $S_n$ be the set of germs $(g, \xi)$ such that there exists a
geodesic path $\ldots g_2g_1$ converging to $\xi\cdot g^{-1}$ where
$g^{-1}=g_n\cdots g_1$. Then for every
$(g, \xi)\in S_n$ we have $\coc(g, \xi)>n-\delta_4$.

For every geodesic path $\ldots g_2g_1=\xi$ and for every $k\ge 1$ the
germ of the transformation $(g_k\cdots g_1)^{-1}$ at the point
$\ldots g_{k+2}g_{k+1}$ can be written as a
product of elements of $S_n\cup S_{n+1}$ and their inverses. Note also
that since the length of elements of $S_n$ (with respect to the
generating set $A$) is bounded above by $n$, the set $S_n\cup S_{n+1}$
has compact closure.

Let $(g, \xi)$ be an arbitrary element of $\G$. Let $\ldots
g_2g_1=\xi$ and $\ldots h_2h_1=\xi\cdot g^{-1}$ be geodesic paths. There
exist indices $n_1$ and $n_2$ both greater than $m$ and such that
distance from $g_{n_1}\cdots g_1\cdot g$ to $h_{n_2}\cdots h_1$ is
less than $\delta$. Denote \[r=h_{n_2}\cdots
h_1\cdot g\cdot (g_{n_1}\cdots g_1)^{-1}.\] Then
$g=(h_{n_2}\cdots h_1)^{-1}rg_{n_1}\cdots g_1 $, and
$(g, \xi)=(g, \ldots g_2g_1)$ is decomposed into the product of the
germs
\[(g_{n_1}\cdots g_1,\;\ldots g_2g_1),\quad(r,\;\ldots
g_{n_1+2}g_{n_1+1}),\quad(h_{n_2}\cdots h_1,\;\ldots g_{n_1+1}g_{n_1+2}r^{-1}).\]

The first and the last germs are products of elements of $S_n\cup
S_{n+1}$ and their inverses. It follows that every element of $\G$ can
be written as a product of elements of $Z_n=(S_n\cup S_{n+1})\cup
A^\delta\cdot(S_n\cup S_{n+1})$. The length of the representation of
$(g, \xi)$ as a product of elements of $Z_n$ is bounded from above by
a function of $|g|$. Consequently, $Z_n$ is a generating set of the groupoid $\G$.

Note that there exists $R>0$ such that $|\coc(g, \xi)|<R$ for all
$g\in A^\delta$. It follows that taking sufficiently big $n$ we can
make the values of $\coc$ on $Z_n$ arbitrarily big.
\end{proof}

Denote by $\partial^2 G$ the direct product $\partial G\times\partial
G$ minus the diagonal. The space $\partial^2 G$ has a natural local
product structure as an open subset of the direct product $\partial
G\times\partial G$. The action $(\xi_1, \xi_2)\cdot g=(\xi_1\cdot g,
\xi_2\cdot g)$ of $G$ on $\partial^2 G$ obviously preserves the local
product structure.

\begin{proposition}
There is an isomorphism of the geodesic quasi-flow $\partial\G\rtimes\G$ with the groupoid of the action of $G$ on $\partial^2 G$ identifying the natural projection $P:\partial\G\arr\partial G=\G^{(0)}$ with the projection of $\partial^2 G$ onto the first coordinate $\partial G$ of the direct product $\partial G\times\partial G\supset\partial^2 G$.
\end{proposition}

\begin{proof}
For every point $\xi\in\partial G$ the Cayley graph $\G(\xi, S)$ is
isomorphic to the Cayley graph of the group $G$, where the isomorphism
is the map $\Gamma_\xi:(g, \xi)\mapsto g^{-1}$. The quasi-cocycle $\coc$
restricted to $\G_\xi$ is then identified by $\Gamma_\xi$ with the
Busemann quasi-cocycle $\beta_\xi$. Consequently, the isomorphism
$\Gamma_\xi$ induces a homeomorphism
$\Gamma_\xi:\partial\G_\xi\arr\partial G\setminus\{\xi\}$. Denote by
$\Gamma(\zeta)$ for $\zeta\in\partial\G_\xi$ the point $(\xi,
\Gamma_\xi(\zeta))$ of $\partial^2 G$. We get a bijection
$\Gamma:\partial\G\arr\partial^2 G$. The fact that it is continuous
and agrees with the local product structure follows directly from the
definitions and Theorem~\ref{th:localproduct}.
It is also checked directly that the map induces an isomorphism of groupoids.
\end{proof}

\begin{corollary}
The groupoid $\G^\top$ is equivalent to $\G$.
\end{corollary}

\section{Expanding self-coverings}

\subsection{General definitions}

Example~\ref{ex:expanding} on page~\pageref{ex:expanding}
is naturally generalized in the following way.
Let $f:\X\arr\X$ be a finite degree self-covering map of a compact metric
space $\X$, and suppose that it is
\emph{expanding}, i.e., that there exists a metric on $\X$ such that
$f$ locally expands the distances by a factor greater than one.

Since $f$ is a covering, it is a local homeomorphism, so that it
generates a pseudogroup $\wt{\mathfrak{F}}$ and the associated groupoid of
germs $\mathfrak{F}$. A natural compact generator of $\mathfrak{F}$ is
the set $S_f$ of all germs of
$f$. Every element $g\in\mathfrak{F}$ can be written as $(f,
y)^{-n_1}(f, x)^{n_2}$, where $x=\be(g)$, $y=\en(g)$, and $n_1,
n_2\in\N$ are such that $f^{n_1}(y)=f^{n_2}(x)$. Note that there
exists a pair $(n_1, n_2)$ such that if $(m_1, m_2)$ is
such that $g=(f, y)^{-m_1}(f, x)^{m_2}$, then $(m_1, m_2)=(n_1+k,
m_2+k)$ for some $k\ge 0$. It follows that $g$ is uniquely determined
by the triple $(y, n_2-n_1, x)$. Multiplication is given by the formula
\[(z, n_1, y)\cdot (y, n_2, x)=(z, n_1+n_2, x).\]

For every $x\in\X$ the $\mathfrak{F}$-orbit
of $x$ is the \emph{grand orbit} of $x$, i.e., the set
\[\bigcup_{n\ge 0}\bigcup_{k=0}^\infty f^{-k}(f^n(x)).\]

Since $f$ is a covering map, the Cayley graph $\mathfrak{F}(x, S_f)$
is a tree. It has a special point of the boundary
corresponding to the forward orbit of $x$. The map $\coc(y, n, x)=-n$
coincides with the Busemann cocycle on the tree associated with this
point. It is easy to see now that the groupoid $\mathfrak{F}$ is
hyperbolic.

Let us describe its dual. Denote by $\wh\X$ the inverse limit of the constant
sequences of spaces $\X$ with respect to the maps $f:\X\arr\X$. The
map $f$ naturally induces a homeomorphism $\wh f:\wh\X\arr\wh\X$
called the \emph{natural extension} of $f$.

Points of the space $\wh\X$ are sequences $(x_1, x_2, \ldots)$ of
points of $\X$ such that $f(x_n)=x_{n-1}$. Denote by $P:\wh\X\arr\X$
the projection $P(x_1, x_2, \ldots)=x_1$.

Let $r>0$ be such that for every open subset $U\subset\X$ of diameter
less than $r$ the set $f^{-1}(U)$ can be decomposed into a disjoint
union of a finite number of subsets $U_i$ such that
$f:U_i\arr U$ is an expanding
homeomorphism. Then diameters of $U_i$ are also less than $r$, and we
conclude that $P^{-1}(U)$ is the direct product of $U$ with the
boundary $C_x$ of the tree of preimages $T_{(f, x)}=\bigsqcup_{n\ge 0}f^{-n}(x)$
for $x\in U$. See Definition~\ref{def:treepreimages}.

In particular, the space $\wh\X$ is a
fiber bundle over $P:\wh\X\arr\X$ and we get a natural local product
structure defined by the covering of $\wh\X$ by the sets
$P^{-1}(U)\approx U\times C_x$. The local product structure is
obviously preserved by $\wh f$. The homeomorphism $\wh
f:\wh\X\arr\wh\X$ is expanding in the direction of the factors
$U\subset\X$ and contracting in the direction of the factors $C_x$ of
the local product structure.

It is easy to check now that the groupoid generated by the action of
the homeomorphism $\wh f$ on $\wh\X$ is a Smale quasi-flow and that
its projection onto the direction of $\X$ is equivalent to
$\mathfrak{F}$. The dual groupoid $\mathfrak{F}^\top$ acts on the
disjoint union of the Cantor sets $C_x$ and is generated by the
holonomies of the local product structure (by the partial
homeomorphisms between the Cantor sets $C_x$ coming from identifications of
the common parts of the covering $\{U\times C_x\}$
of $\wh\X$) and by the action of $\wh f$.
The sets $P^{-1}(U_1)$ and $P^{-1}(U_2)$ intersect if and only if
$U_1$ and $U_2$ intersect, and the corresponding holonomy is a
homeomorphism $C_{x_1}\arr C_{x_2}$ for $x_1\in U_1$ and $x_2\in U_2$.

If $\X$ is connected, then each set $C_x$ is an open transversal of
$\mathfrak{F}^\top$ and the groupoid $\mathfrak{F}^\top$ is
generated by the holonomy group of the fiber bundle $P:\wh\X\arr\X$
and by projection of the action of $\wh f:\wh\X\arr\wh\X$. The
holonomy group is called the \emph{iterated monodromy group} of the
map $f:\X\arr\X$.

\begin{examp}
Let $\alb=\{1, 2, \ldots, d\}$ be a finite alphabet, and let
$T\subset\alb\times\alb$ be a set of words of length $2$. The set $T$
can be described by the matrix $A=(a_{ij})_{1\le i, j\le d}$ where
$a_{ij}=0$ if $ij\notin T$ and $a_{ij}=1$ if $ij\in T$. Consider the
space $F$ of sequences $x_1x_2\ldots\in\alb^\omega$ such that
$x_ix_{i+1}\in T$ for all $i\ge 1$. Then the space $F$ is invariant
under the shift map $f(x_1x_2\ldots)=x_2x_3\ldots$ and the map
$f:F\arr F$ is an expanding self-covering.

The groupoid generated by the shift $f$ is a well known object, called
the \emph{Cuntz-Krieger} groupoid $\mathfrak{O}_A$, see~\cite{cuntzkrie}.

The natural extension of $f:F\arr F$ is the space of two-sided
sequences $\ldots x_{-1}x_0x_1\ldots$ such that $x_ix_{i+1}\in T$ for
all $i\in\Z$ together with the shift map. It is easy to see now that
the groupoid dual to $\mathfrak{O}_A$ is the Cuntz-Krieger groupoid
$\mathfrak{O}_{A^\top}$ defined by the transposed matrix $A^\top$.
\end{examp}

\begin{examp}
Let $F_2$ be the free group generated by $a$ and $b$. Then the
boundary $\partial F_2$ is naturally identified with the one-sided
shift $F$ of finite type over the alphabet $\alb=\{a, b, a^{-1}, b^{-1}\}$
defined by the set $T=\alb^2\setminus\{aa^{-1}, a^{-1}a, bb^{-1},
b^{-1}b\}$. It is easy to see that the groupoid of the action of $F_2$
on its boundary is isomorphic to the groupoid generated by the
one-sided shift $F$.
\end{examp}

\subsection{Iterated monodromy groups}
\label{ss:img}

Suppose that $f:\X\arr\X$ is an expanding self-covering and the space $\X$ is
path connected. Let, as above, $\mathfrak{F}$ be the groupoid
generated by $f$. In this case we can realize the dual groupoid
$\mathfrak{F}^\top$ as a groupoid with the space of units equal to
the boundary of the tree of preimages $T_{(f, x_0)}$ of a point
$x_0\in\X$, see Definition~\ref{def:treepreimages}. The boundary $C_{x_0}$
of the tree $T_{(f, x_0)}$ is naturally identified with the fiber
$P^{-1}(x_0)$ of the natural extension $\wh\X$, i.e., with the set of
sequences $(x_0, x_1, x_2, \ldots)$ such that $f(x_{i+1})=x_i$ for all
$i\ge 0$.

The pseudogroup associated with the dual groupoid $\mathfrak{F}^\top$
is generated by the holonomy group and the action of $\wh f$. More
explicitly, for any element $\gamma$ of the fundamental group
$\pi_1(\X, x_0)$, the corresponding homeomorphism $M_\gamma$ of
$C_{x_0}=\partial T_{(f, x_0)}$ is the limit of the monodromy actions of
$\gamma$ on the levels $f^{-n}(x_0)$ of the tree $T_{(f, x_0)}$. Namely,
for every $z\in f^{-n}(x_0)$ there exists a unique lift of $\gamma$
that starts at $z$; denote the end of this lift by $\gamma(z)$; then
$M_\gamma(x_0, x_1, \ldots)=(x_0, \gamma(x_1), \gamma(x_2), \ldots)$
for every $(x_0, x_1, \ldots)\in C_{x_0}$.

The homeomorphism $(\wh f)^{-1}$ maps $(x_0, x_1, \ldots)$ to $(x_1,
x_2, \ldots)$. If $\alpha$ is a path starting at $x_1$ and ending in
$x_0$ then we get the associated element of the pseudogroup
$\wt{\mathfrak{F}^\top}$ equal to the map $L_\alpha:(x_0, x_1,
\ldots)\mapsto (y_0, y_1, \ldots)$ where $y_n$ is the end of the lift
of $\alpha$ by $f^n$ starting at $x_{n+1}$ (note that then $y_0=x_0$).

We have then the following description of $\mathfrak{F}^\top$.

\begin{proposition}
Choose a generating set $S$ of $\pi_1(\X, x_0)$ and a collection
$\alpha_x$ of paths connecting $x\in f^{-1}(x_0)$ to $x_0$. Then the
pseudogroup $\wt{\mathfrak{F}^\top}$ is generated by the local
homeomorphisms $L_{\alpha_x}$ and $M_\gamma$
for $\gamma\in S$.
\end{proposition}

Note that
\begin{equation}\label{eq:recurn}L_{\alpha_y}^{-1}M_\gamma L_{\alpha_x}=M_{\alpha_y^{-1}\gamma_x\alpha_x}
\end{equation}
where $\gamma_x$ is the lift of $\gamma$ by $f$ starting at $x$ and
$y$ is the end of $\gamma_x$, see Figure~\ref{fig:recur}. Here we multiply paths in the same way
as functions, i.e., in the product $\gamma_x\alpha_x$ the path $\alpha_x$
is passed before $\gamma_x$.

\begin{figure}
\includegraphics{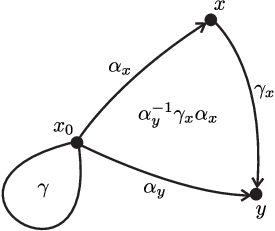}
\caption{}\label{fig:recur}
\end{figure}

We can find a homeomorphism of $C_{x_0}$ with the space of infinite
sequences $a_1a_2\ldots$ over the alphabet $\alb=f^{-1}(x_0)$
conjugating the maps $L_{\alpha_x}$ with the map $a_1a_2\ldots\mapsto
xa_1a_2\ldots$. Then equality~\eqref{eq:recurn} becomes a recurrent
formula for computing the action of the generators of $\pi_1(\X, x_0)$
on the space of infinite sequences.

Note that the obtained formulae defining the groupoid
$\mathfrak{F}^\top$ do not use the fact that $f$ is expanding. The
group of homeomorphisms of $C_{x_0}$ generated by $M_\gamma$ for
$\gamma\in\pi_1(\X)$ is called the \emph{iterated monodromy group} of the
self-covering $f$.

\begin{examp} Let us revisit Example~\ref{ex:expanding}.
Consider the circle $\R/\Z$ and its self-covering $f:x\mapsto 2x$. Let
$\mathfrak{F}$ be the groupoid generated by $f$. Note that it is equivalent to
the groupoid generated by the maps $x\mapsto x+1$ and $x\mapsto 2x$
acting on $\R$.

The natural extension of $f$ is topologically
conjugate to the Smale solenoid (see~\cite{brin:book}).

Let us compute $\mathfrak{F}^\top$ using~\eqref{eq:recurn}. Take
$x_0=0$ and let $\gamma$ be the generator of the fundamental group
$\pi_1(\R/\Z, x_0)$ equal to the image of the segment $[0, 1]$ (where
$0$ is the beginning). The point $x_0$ has two preimages $0$ and
$1/2$. Let $\alpha_0$ be the trivial path at $0$, and let $\alpha_1$ be
the path from $1/2$ to $0$ equal to the image of the segment $[0,
1/2]$. Consider the space $\{0, 1\}^\omega$ of binary infinite
sequences, and identify the transformations $L_{\alpha_0}$ and
$L_{\alpha_1}$ with the maps $a_1a_2\ldots\mapsto 0a_1a_2\ldots$ and
$a_1a_2\ldots\mapsto 1a_1a_2\ldots$ respectively. Denote
$M_\gamma=\tau$. Then it follows from ~\eqref{eq:recurn} that the
transformation $\tau$ of the space of binary sequences is defined by
the recurrent rule
\[\tau(0w)=1w,\qquad \tau(1w)=0\tau(w).\]
Note that this transformation coincides with the rule of adding 1 to a
dyadic integer, so that $\tau(a_0a_1\ldots)=b_0b_1\ldots$ is
equivalent to $1+\sum_{n=0}^\infty a_n2^n=\sum_{n=0}^\infty b_n2^n$ in
the ring of dyadic integers. The transformation $\tau$ is known as the
\emph{adding machine}, or the \emph{odometer}.

Note that transformations $L_{\alpha_0}$ and $L_{\alpha_1}$ are identified
with the maps $x\mapsto 2x$ and $x\mapsto 2x+1$ on the ring of
dyadic integers.

As a corollary we get the following description of
$\mathfrak{F}^\top$, connecting Examples~\ref{ex:expanding} and~\ref{ex:dyadic}.

\begin{proposition}
Let $\mathfrak{F}$ be the groupoid generated by the maps $x\mapsto x+1$ and $x\mapsto 2x$ on $\R$. Then it is hyperbolic and the dual groupoid $\mathfrak{F}^\top$ is the groupoid generated by the maps $x\mapsto x+1$ and $x\mapsto 2x$ on the ring $\Z_2$ of dyadic integers.
\end{proposition}
\end{examp}

\subsection{Hyperbolic rational functions}

In many cases of expanding self-coverings $f:\X\arr\X$ the fundamental
group of $\X$ is too complicated, but we can embed $\X$ into a space
$\M$ with a finitely generated fundamental group so that $f$ can be
extended to a covering $f:\M_1\arr\M$ of $\M$ by a subset
$\M_1\subset\M$. Let $f^n:\M_n\arr\M$ be the $n$th iteration of the
partial map $f$. If $\X=\bigcap_{n\ge 1}\M_n$, then the iterated
monodromy group of $f:\X\arr\X$ can be computed directly on
$f:\M_1\arr\M$ using the recurrent relation~\eqref{eq:recurn}. This is
a particular case of a more general method of approximation of
expanding self-coverings described in~\cite{nek:models}. For more on
iterated monodromy groups see the monograph~\cite{nek:book}
and~\cite{bgn,nek:filling}.

A particular class of examples comes from holomorphic dynamics. Let
$f\in\C(z)$ be a rational function seen as a self-map of the Riemann
sphere $\wh\C$. Let $C_f$ be the set of critical values of $f$, and
let $P_f=\bigcup_{n\ge 1}f^n(C_f)$ be the \emph{post-critical set} of
$f$.

Suppose that $f$ is \emph{hyperbolic}, i.e., is expanding on a
neighborhood of its Julia set in some metric. Then the post-critical
set $P_f$ accumulates on a union of a finite number of cycles, which
are disjoint with the Julia set. We get hence a covering map
$f:\M_1\arr\M$, where $\M=\wh\C\setminus\overline P_f$ and
$\M_1=f^{-1}(\M)\subset\M$ are open neighborhoods of the Julia set. The
iterated monodromy group of the action of $f$ on the Julia set of $f$
can be computed from the covering $f:\M_1\arr\M$.

For example, direct computation
(see~\cite[Subsection~5.2.2]{nek:book})
show that the iterated
monodromy group of the polynomial $z^2-1$ is generated by the
transformations $a$ and $b$ of the space of infinite binary sequences
$\{0, 1\}^\omega$ given by the recurrent rules:
\[a(0w)=1w,\quad a(1w)=0b(w),\qquad b(0w)=0w,\quad b(1w)=1a(w).\]

It follows that the groupoid dual to the groupoid generated by the
action of $z^2-1$ on its Julia set is the groupoid generated
by the transformations $a$, $b$, and by the one-sided shift
$x_1x_2\ldots\mapsto x_2x_3\ldots$. The Julia set is shown on Figure~\ref{fig:basilicaj}.

\begin{figure}
\includegraphics{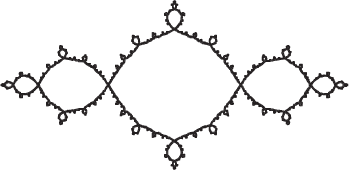}
\caption{Julia set of $z^2-1$} \label{fig:basilicaj}
\end{figure}

Iterated monodromy groups have interesting group theoretic properties
and are useful tools in the study of symbolic dynamics of expanding
self-coverings,
see~\cite{nek:book,nek:filling,bartnek:rabbit,nek:models,nek:dendrites}.

\section{Contracting self-similar groups}

We have seen above that dual groupoids of groupoids generated
by  expanding self-coverings
of connected topological spaces are generated by the shift map and by
a group $G$ of homeomorphisms
$g:\alb^\omega\arr\alb^\omega$ satisfying recurrent rules of the form
\[g(xw)=yg_x(w)\]
for $x, y\in\alb$ and $g, g_x\in G$.

\begin{defi}
\label{def:contractinggroups}
A group $G$ acting faithfully on $\alb^\omega$ is said to be \emph{self-similar}
if for every $g\in G$ and $x\in\alb$ there exist $h\in G$ and
$y\in\alb$ such that $g(xw)=yh(w)$ for all $w\in\alb^\omega$.
\end{defi}

It follows from the definition that for every $g\in G$ and for every
finite word $v\in\alb^*$ there exists an element $h\in G$ such that
$g(vw)=uh(w)$ for some word $u$ of length equal to the length of
$v$. We denote $h=g|_v$.

Denote by $S_x$ the map $S_x(w)=xw$. Then $S_x^{-1}$ is restriction of
the shift to the set of sequences starting with $x$. We will denote
$S_{x_1x_2\ldots x_n}=S_{x_1}S_{x_2}\cdots S_{x_n}$. Then the condition
$g(vw)=uh(w)$ is equivalent to the equality
\[gS_v=S_uh.\]
In particular, $g|_v=S_u^{-1}gS_v$.

\begin{proposition}
\label{pr:contractinggroups}
Let $G$ be a self-similar group acting on $\alb^\omega$ and suppose
that it is \emph{contracting}, i.e., that there exists a finite set
$\nuke\subset G$ such that for every $g\in G$ there exists a
positive integer $n$ such that $g|_v\in\nuke$ for all words of
length at least $n$.

If the groupoid of germs of $G$ is Hausdorff, then the groupoid $\G$ of germs
of the pseudogroup generated by the shift and $G$ is
hyperbolic.
\end{proposition}

\begin{proof}
Every element of the inverse semigroup generated by the
transformations $g\in G$ and $S_x$, for $x\in X$, is equal to a product
of the form $S_vgS_u^{-1}$, since every product $gS_x$ can be
rewritten as $S_yg|_x$, the product $S_x^{-1}S_x$ is equal to the
identical homeomorphism, and the product $S_x^{-1}S_y$ is empty for
$x\ne y$. Moreover, since every homeomorphism $g\in G$ is written as a
disjoint union of local homeomorphisms $S_yg|_xS_x^{-1}$, any element
of $\G$ is a germ of a local homeomorphism $S_ugS_v^{-1}$ for
$g\in\nuke$. We conclude that $\G$ is compactly generated by
$\nuke$ and $S_x$ for $x\in\alb$.

Define the degree $\coc$ of any germ of a local homeomorphism
$S_vgS_u^{-1}$ as length of $v$
minus length of $u$. It is a well defined continuous cocycle
$\coc:\G\arr\Z$.
It is easy to see that every element of $\G$ of positive degree has a
neighborhood which is a contracting map. Hence, the groupoid $\G$ has a
compact contracting generating set. Hyperbolicity follows then from
Theorem~\ref{th:contraction}.
\end{proof}

The groupoid dual to the groupoid associated with a self-similar
contracting group is generated by the \emph{limit dynamical
  system}. Consider the space of left-infinite sequences
$\alb^{-\omega}$ with the direct product topology. We say that
sequences $\ldots x_2x_1, \ldots y_2y_1$ are \emph{asymptotically
  equivalent} (with respect to the action of $G$) if there exists a
finite subset $N\subset G$ and a sequence $g_k\in N$ such that
$g_k(x_k\ldots x_1)=y_k\ldots y_1$ for all $k$. The quotient of
$\alb^{-\omega}$ by the asymptotic equivalence relation is called the
\emph{limit space} of $G$. The shift $\ldots x_2x_1\mapsto\ldots
x_3x_2$ agrees with the asymptotic equivalence relation, and hence it
induces a continuous self-map of the limit space. This self-map is the
\emph{limit dynamical system} of the contracting self-similar group
$G$. 

If the groupoid of germs of $G$ is principal, then the limit dynamical
system is a self-covering (see~\cite{nek:cpalg}). 
More generally, if the groupoid of germs of $G$ is Hausdorff, then the limit dynamical system is a
self-covering of an orbispace, and the groupoid dual to $\G$ is still
generated by the self-covering (but the definitions should be modified
to include the orbispace case).

Since the groupoid $\G^\top$ acts on the limit space of $G$, the
boundaries of the Cayley graphs of $\G$ are locally homeomorphic to
the limit space. The fact that the limit space is the boundary of a
naturally defined Gromov hyperbolic graph $\Gamma(G)$ was noted for
the first time in~\cite{nek:hyplim}. Definition of this graph is very
similar to the definition of the Cayley graphs of $\G$. In fact,
$\Gamma(G)$ is locally isomorphic to the Cayley graphs of $\G$ (in
particular, positive cones of the Cayley graphs are isomorphic to
positive cones of $\Gamma(G)$). It is defined in the following
way. The set of vertices of $\Gamma(G)$ is the set $\alb^*$ of all
finite words over the alphabet $\alb$. Two vertices are connected by
an edge if and only if either they are of the form $v, s(v)$ for a
generator $s$ of $G$ and a word $v\in\alb^*$, or they are of the form
$v, xv$ for $v\in\alb^*$ and $x\in\alb$. If $G$ is contracting,
then $\Gamma(G)$ is Gromov hyperbolic and its boundary is naturally
homeomorphic to the limit space of $G$.

See Figure~\ref{fig:complex}, where a part of the graph $\Gamma(G)$
for $G$ equal to the iterated monodromy group of $z^2-1$ is
shown. Compare it with the Julia set of $z^2-1$ on
Figure~\ref{fig:basilicaj}. Here we draw only edges corresponding to
generators $a, b$, and $S_1$, since $S_0=a^{-1}S_1$.

\begin{figure}
\includegraphics[width=3in]{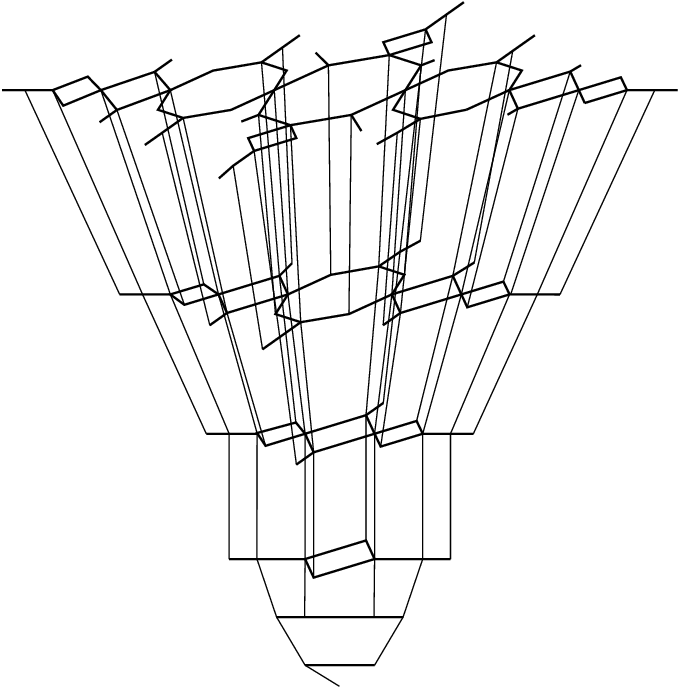}
\caption{Graph $\Gamma(\mathrm{IMG}(z^2-1))$} \label{fig:complex}
\end{figure}

There are many examples of contracting self-similar groups for
which the groupoid of germs of the action on $\alb^\omega$ is not
Hausdorff. For such groups the limit dynamical system is a covering of
an orbispace by its \emph{sub-orbispace}, and it does not generate any
\'etale groupoid. However, theory of iterated monodromy groups for
such partial self-coverings of orbispaces is well developed
(see~\cite{nek:book,nek:models}), which suggests that a more general
duality theory for hyperbolic groupoids exists that
includes non-Hausdorff (and non-\'etale) groupoids, e.g.,
sub-hyperbolic rational functions.

\section{Smale spaces}
\subsection{Definitions}

Smale spaces were defined by D.~Ruelle (see~\cite{ruelle:therm}) as
synthetic models of
hyperbolic dynamical systems, and are generalizations of axiom
A diffeomorphisms restricted to the non-wandering set.

\begin{defi}
A \emph{Smale space} is a compact metric space $\X$ together with a
local product structure and a homeomorphism $f:\X\arr\X$ that
preserves the local product structure, is locally contracting in
the first direction, and locally expanding in the second direction of
the local product structure (see condition (3) of
Definition~\ref{def:hypflow}).
\end{defi}

Just by definition, the groupoid of germs of the homeomorphism
$f:\X\arr\X$ is compactly generated and has a natural continuous
cocycle $\coc$ equal to the degree of the iteration of $f$. We see
immediately that this groupoid satisfies almost all conditions of
Definition~\ref{def:hypflow}. It is not assumed, however, in
the standard definition
of a Smale space that $f$ is locally Lipschitz and that the local
product structure agrees with the metric in the sense of
Definition~\ref{def:compatiblelipsch}.

But any Smale space carries a natural log-scale, defined by
D.~Fried~\cite{fried:naturalmetric}, which satisfies all the
compatibility conditions.

Let $U\subset\X\times\X$ be a neighborhood of the diagonal such that
for $(f^k(x), f^k(y))\in U$ for all $k\in\Z$ implies that $x=y$. Such
a neighborhood exists for every Smale space.

Define then $\ell(x, y)$ to be the maximal value of $n$ such that
$(f^k(x), f^k(y))\in U$ for all $k\in\Z$ such that $|k|<n$. If
such $n$ does not exist, we set $\ell(x, y)=0$.

Denote by $U_n$ the set of pairs $(x, y)$ such that $\ell(x, y)\ge
n$. Then $U_n=\bigcap_{|k|<n}f^k(U)$, where $f^k(x, y)=(f^k(x),
f^k(y))$. The sets $U_n$ are neighborhoods of the diagonal, and
$\bigcap_nU_n$ is equal to the diagonal.

\begin{lemma}
The defined function $\ell$ is a log-scale compatible with the
topology on $\X$.
The log-scale $\ell$ does not depend, up to Lipschitz equivalence, on
the choice of $U$.
\end{lemma}

\begin{proof}
Since $\X$ is compact, and intersection of the sets $U_n$ is equal to
the diagonal, there exists $n_0$ such that $U_{n_0}\circ
U_{n_0}\subset U$. Then $f^k(U_{n_0})\circ f^k(U_{n_0})\subset f^k(U)$
for every $k$, hence $U_{n+n_0}\circ U_{n+n_0}\subset U_n$ for every
$n$. It follows that for every positive integer $n$ and for all $x, y,
z\in\X$ the inequalities $\ell(x, y)\ge n+n_0$ and $\ell(y, z)\ge n+n_0$
imply $\ell(x, z)\ge n$. Consequently, $\ell(x,
z)\ge\min(\ell(x, y), \ell(y, z))-n_0$ for all $x, y, z\in\X$.

If $U'$ is another neighborhood of the diagonal, and $\ell'$ is the
corresponding log-scale, then there exists $n_1$ such that
$U_{n_1}\subset U'$, and then $f^k(U_{n_1})\subset f^k(U')$, hence
$U_{n_1+n}\subset U'_n$, so that $\ell'(x, y)\ge\ell(x, y)-n_1$.
\end{proof}

We call the log-scale $\ell$ the \emph{natural log-scale} on the Smale
space. It is obvious that $f$ is Lipschitz with respect to the
natural log-scale. In fact, $|\ell(f(x), f(y))-\ell(x, y)|\le 1$.

\begin{lemma}
The natural log-scale is compatible with the local product structure.
\end{lemma}

\begin{proof}
Let $\epsilon>0$ be such that the inequalities $|f^k(x)-f^k(y)|<2\epsilon$
for all $k\in\Z$ imply that $x=y$. Let $U$ be the
set of pairs $(x, y)\in\X\times\X$ such that $|x-y|<\epsilon$ and
let $U'$ be the set of pairs such that $|x-y|<2\epsilon$. Let $\ell$
and $\ell'$ be the corresponding log-scales. We know that they are
Lipschitz equivalent.

Let $R\subset\X$ be a rectangle of diameter less than $\epsilon$. Let
us show that there exists a constant $D$ such that
\[|\ell(x, y)-\min(\ell(x, [x, y]), \ell(y, [x, y]))|<D\]
for all $x, y\in R$.

Suppose that $k$ is a positive integer such that $(f^k(x),
f^k(y))\notin U'$, i.e., $|f^k(x)-f^k(y)|\ge 2\epsilon$. Since $f$
is contracting in the first direction of the local product structure,
$|f^k([x, y])-f^k(y)|<\epsilon$, therefore $|f^k([x,
y])-f^k(x)|>\epsilon$, by the triangle inequality.

If $k$ is a positive integer such that $|f^k([x,
y])-f^k(x)|>2\epsilon$,
then by the same argument, $|f^k(x)-f^k(y)|>\epsilon$.

Similarly, if $k$ is negative and $|f^k(x)-f^k(y)|\ge 2\epsilon$, then
$|f^k([x, y])-f^k(y)|>\epsilon$; and if
$|f^k([x, y])-f^k(y)|>2\epsilon$, then $|f^k(x)-f^k(y)|>\epsilon$.

It follows that \[\ell'(x, y)\ge\min(\ell([x, y], x), \ell([x, y], y)),\qquad
\min(\ell'([x, y], x), \ell'([x, y], y))\ge\ell(x, y).\] But the
difference $|\ell'-\ell|$ is uniformly bounded, hence the difference
$|\ell(x, y)-\min(\ell([x, y], x), \ell([x, y], y))|$ is uniformly
bounded too.

Similar arguments show that there exists a constant $D$ such that
if $x_1, x_2$ are such that $x_2\in\proj_+(R, x_1)$, then for every
$y\in R$ we have
\[|\ell(x_1, x_2)-\ell([x_1, y], [x_2, y])|<D,\quad |\ell(x_1, [x_1,
y])-\ell(x_2, [x_2, y])|<D.\]
This finishes the proof of the lemma.
\end{proof}

We have thus proved the following proposition.

\begin{proposition}
The groupoid of germs generated by the homeomorphism $f:\X\arr\X$ of a
Smale space is a Smale quasi-flow.
\end{proposition}

\begin{corollary}
Ruelle groupoids of the groupoid of germs of a Smale space
are mutually dual hyperbolic groupoids.
\end{corollary}

All examples above, except for Gromov hyperbolic groups, are
Ruelle groupoids of some Smale spaces.

The Smale space associated with an expanding self-covering
$f:\X\arr\X$ is the natural extension $\wh f:\wh\X\arr\wh\X$. The
Smale space associated with a contracting group is called the
\emph{limit solenoid} and is defined in a way similar to the
definition of the limit space, but starting from bi-infinite sequences
(see~\cite[Section~5.7]{nek:book}). Some other examples are considered
below.

\subsection{Quadratic irrational rotation}

Consider the classical Anosov diffeomorphism of the two-torus
$\R^2/\Z^2$ defined by the matrix $A=\left(\begin{array}{cc}2 & 1\\ 1 &
  1\end{array}\right)$. Let
$\Gh$ be the groupoid generated by it. Lifting $\Gh$ to the universal
covering $\R^2$ of the torus we get a groupoid equivalent to
$\Gh$ generated by the action of $\Z^2$ on $\R^2$ and the linear map
$A$. In other words, it is the groupoid of germs of the group of
affine transformations of the form $\vec v\mapsto A^n\vec v+\vec a$,
where $n\in\Z$ and $\vec a\in\Z^2$.

The eigenvalues of $A$ are $\frac{3\pm\sqrt{5}}2$. The eigenvectors
$\left(\begin{array}{c}1\\ \frac{-1+\sqrt{5}}2\end{array}\right)$
and $\left(\begin{array}{c}1\\ \frac{-1-\sqrt{5}}2\end{array}\right)$
are orthogonal. Ratio of lengths of projections of
$\left(\begin{array}{c} 1\\ 0\end{array}\right)$ and
$\left(\begin{array}{c} 0\\ 1\end{array}\right)$ onto either eigenspace
is the golden ratio $\varphi=\frac{1+\sqrt{5}}2$. Note that
$\varphi^2$ and $\varphi^{-2}$ are the eigenvalues of $A$.

It follows the that Ruelle groupoids of the Smale quasi-flow $\Gh$ are
equivalent to the groupoid of germs of the group of affine
transformations of $\R$ generated by $x\mapsto x+1$, $x\mapsto
x+\varphi$, and $x\mapsto\varphi^2x$. It is the groupoid of germs of
the group of affine transformations of the form $x\mapsto
\varphi^{2n}x+\alpha$ for $n\in\Z$ and $\alpha\in\Z[\varphi]$.

In general, if $\theta$ is a root of the polynomial $x^2+bx+1$ for
$b\in\Z$, $|b|>2$, then the groupoid generated by $x\mapsto x+1$,
$x\mapsto x+\theta$, and $x\mapsto\theta x$ is hyperbolic and
self-dual.

\subsection{Vershik transformations, Williams solenoids, and aperiodic tilings}

A \emph{Bratteli diagram} is defined by two sequences $(V_0, V_1,
\ldots)$ and $(E_0, E_1, \ldots)$ of finite sets and sequences
$\be_n:E_n\to V_n$, $\en_n:E_n\to V_{n+1}$ of maps. We interpret
$\bigcup_{i\ge 0}V_i$ and $\bigcup_{i\ge 0}E_i$ as the sets of
vertices and edges respectively. An edge $e\in E_n$ connects the
vertex $\be_n(e)\in V_n$ with $\en(e)\in V_{n+1}$.

A Bratteli diagram is called \emph{stationary} if the sequences $V_n=V$,
$E_n=E$, $\be_n=\be$, and $\en_n=\en$ are constant. If we 
identify $V$ with
$\{1, 2, \ldots, d\}$ for $d=|V|$, we the
stationary Bratteli diagram can be described by the matrix $A=(a_{ij})$, where $a_{ij}$
is the number of edges $e\in E$ such that $\be(e)=i$ and $\en(e)=j$.

The \emph{space of paths} in a Bratteli diagram is the space of all
sequences $(e_0, e_1, \ldots)$ such that $\en(e_n)=\be(e_{n+1})$ for
all $n\ge 0$ with the topology of a subset of the direct
product $E_0\times E_1\times\cdots$.

A \emph{Vershik-Bratteli diagram} is a Bratteli diagram in which for
every $v\in V_n$, $n\ge 1$, there is a linear ordering of the
set of edges $\en_{n-1}^{-1}(v)$. A \emph{stationary Vershik-Bratteli
  diagram} is a stationary Bratteli diagram with the same orderings on
each level.

Suppose that $(e_0, e_1, \ldots)$ is a path in a Vershik-Bratteli
diagram consisting not only of maximal edges. Find
the first non-maximal edge $e_n$. Define then
\[\tau(e_0, e_1, \ldots)=(f_0, f_1, \ldots, f_n, e_{n+1}, e_{n+2},
\ldots),\]
where $f_n$ is the next after $e_n$ edge in the linear order of
$\en_n^{-1}(\en_n(e_n))$, and all $f_i$ for $i<n$ are minimal. This
condition uniquely determines a continuous map $\tau$ from the set of
all non-maximal paths to the set of all non-minimal
paths of the diagram. If the diagram has only one minimal path (i.e.,
a path consisting of minimal edges only) and
one maximal path, then we set the image of the maximal path
under $\tau$ to be the
minimal path. In this way we get a homeomorphism of the space of paths
of the diagram, called the \emph{Vershik transformation}.

For example, the binary adding machine transformation is conjugate to the
Vershik transformation defined by the diagram consisting of single
vertices and two edges on each level.

Consider a stationary Vershik-Bratteli diagram containing only one
minimal and only one maximal paths. Let $\tau:F\arr F$ be the
corresponding Vershik transformation. Since the diagram is stationary,
$F$ is shift-invariant. Let $\G$ be the groupoid of germs of the
pseudogroup generated by $\tau$ and the shift. It is easy to check
using Theorem~\ref{th:contraction} (in the same way as in
Proposition~\ref{pr:contractinggroups}) that the groupoid $\G$
is hyperbolic. Its geodesic quasi-flow is an example of a
\emph{Williams solenoid}, see~\cite{williams:attractors}.
Instead of describing the general situation, let us just analyze one example.

Consider the stationary Bratteli diagram defined by the matrix
$\left(\begin{array}{cc} 2 & 1 \\ 1 & 1\end{array}\right)$. Let us label the edges of the diagram as it is shown on Figure~\ref{fig:bratteli}. We order the edges by the relation $1<2<3$ and $4<5$.

\begin{figure}
\centering
\includegraphics{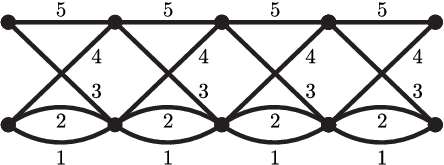}
\caption{} \label{fig:bratteli}
\end{figure}

The space $F$ of infinite paths is the subset of the space $\{1, 2, 3,
4, 5\}^\omega$ consisting of sequences $a_0a_1\ldots$ such that
\[a_ia_{i+1}\in\{11, 12, 14, 21, 22, 24, 31, 32, 34, 43, 45, 53, 55\}\]
for all $i$. The
corresponding Vershik transformation is given by the recurrent rules
\begin{gather*}
\tau(1w)=2w,\quad\tau(2w)=3w,\\
\tau(31w)=1\tau(1w),\quad\tau(32w)=4\tau(2w),\quad\tau(34w)=4\tau(4w),\\
\tau(4w)=5w,\quad\tau(5w)=1\tau(w).
\end{gather*}

Let $\G$ be the groupoid of germs of the pseudogroup generated by the
shift and $\tau$. We will see later that it is hyperbolic. The
corresponding Smale space $\partial\G$ can be
described in the following way. Consider the mapping torus of the
homeomorphism $\tau$, i.e., direct product $I\times F$, where
$I=[0, 1]$ is the unit interval, modulo the equivalence identifying
every point $(1, w)$ with $(0, \tau(w))$. Let us denote it by $T$. We
will denote by $I_w$ the subset $I\times\{w\}$ of $T$. The space $T$
has a natural local product structure inherited from the direct
product $I\times F$.

Denote by $A_w$ the union $I_{1w}\cup I_{2w}\cup I_{3w}$ and by $B_w$
the union $I_{4w}\cup I_{5w}$. Note that $A_w$ and $B_w$ are
homeomorphic to intervals. Note that $A_w$ is defined if and
only if $w$ starts with $1$, $2$, or $4$; $B_w$ is defined if and only
if $w$ starts with $3$ or $5$.

Applying the recursive definition of $\tau$ twice, we get that the sets
\[I_{11w}\cup I_{21w}\cup I_{31w}\cup I_{12w}\cup I_{22w}\cup I_{32w}\cup I_{43w}\cup I_{53w}=A_{1w}\cup A_{2w}\cup B_{3w}\]
and
\[I_{14w}\cup I_{24w}\cup I_{34w}\cup I_{45w}\cup I_{55w}=A_{4w}\cup B_{5w}\]
are intervals.
Let us change the metric on $T$ so that all intervals
$A_w$ have length $1$ and the intervals $B_w$ have length
$\frac{\sqrt{5}-1}2$. Then length of the interval $A_{1w}\cup
A_{2w}\cup B_{3w}$ is equal to
$2+\frac{\sqrt{5}-1}2=\frac{3+\sqrt{5}}2$, and length of $A_{4w}\cup
B_{5w}$ is equal to $1+\frac{\sqrt{5}-1}2=\frac{1+\sqrt{5}}2=
\frac{\sqrt{5}-1}2\cdot\frac{3+\sqrt{5}}2$.

Consider now the homeomorphism $f:T\arr T$ mapping $A_{1w}\cup
A_{2w}\cup B_{3w}$ by an affine orientation-preserving map to $A_w$,
and $A_{4w}\cup B_{5w}$ to $B_w$. Then $f$ contracts the distances
inside the intervals by $\frac{3+\sqrt{5}}2$, and acts as the shift on
the fibers $F$ of the local product structure. It follows that $(T,
f)$ is a Smale space. It follows directly from the definition that the
holonomy pseudogroup of the stable lamination is generated  by the
homeomorphism $\tau:F\arr F$, and that the unstable Ruelle groupoid of
$(T, f)$ is equivalent to $\G$.

Let us describe the dual groupoid $\G^\top$, which is the
Ruelle groupoid acting on the stable direction. The leaves of the stable
lamination are the path connected components of $T$, are homeomorphic
to $\R$, and are unions of intervals $A_w$ and $B_w$. Each leaf is
tiled by the intervals $A_w$ and $B_w$ in some order. The map
$f^{-1}:T\arr T$ will map each tile $A_w$ to the union of the tiles
$A_{1w}\cup A_{2w}\cup B_{3w}$, and will map each tile $B_w$ to the
union of the tiles $A_{4w}\cup B_{5w}$. We get a
\emph{self-similarity}, or \emph{inflation rule} for the obtained
class of tilings of $\R$. In fact this inflation rule determines the
class of tilings in a unique way. Let us write for every leaf of the
stable foliation the sequence of letters $A$ and $B$ according to the
types ($A_w$ or $B_w$) of the corresponding intervals of the
tiling. Then it follows from the inflation rule that the obtained set
of bi-infinite sequences is the \emph{substitution shift}: it is the
space of all bi-infinite sequences $w$ such that every finite subword
of $w$ is a subword of an element of the sequence $A, AAB, AABAABAB,
\ldots$ obtained from $A$ by iterations of the endomorphism $A\mapsto
AAB, B\mapsto AB$ of the free semigroup generated by $A$ and $B$.

All the tilings of leaves of $T$ are \emph{aperiodic}, i.e., have no
translational symmetries. On the other hand, they have many \emph{local symmetries}:
every finite portion of the tiling appears infinitely often in it (and
in any other tiled leaf of $T$). The Ruelle pseudogroup of the stable
direction of $T$ is the pseudogroup generated by all such local
symmetries and the self-similarity given by the inflation rule.

Note that the stable Ruelle groupoid $\G^\top$ of the Smale space $(T,
f)$ is equivalent to a sub-groupoid of the groupoid generated by the
transformations $x\mapsto x+1, x\mapsto x+\frac{1+\sqrt{5}}{2}$, and
$x\mapsto \frac{3+\sqrt{5}}2x$, which is the Ruelle groupoid of the
Anosov diffeomorphism defined by $\left(\begin{array}{cc} 2 & 1\\ 1&
    1\end{array}\right)$, described above.

Another
classical example of a self-similar aperiodic tilings is the Penrose tiling
(see~\cite{penrose,mgard}).
See~\cite{nek:smale} for description of the pseudogroup dual to the pseudogroup
generated by local symmetries and self-similarity of the Penrose
tiling.

Paper~\cite{nek:smale} studies a general situation of a
\emph{self-similar inverse semigroup}, in particular contracting
self-similar inverse semigroups (in the spirit of
Definition~\ref{def:contractinggroups}).
Groupoid of germs of the inverse semigroup generated by a contracting self-similar
inverse semigroup and the shift is hyperbolic (if it is Hausdorff).

\backmatter

\def\cprime{$'$}
\providecommand{\bysame}{\leavevmode\hbox to3em{\hrulefill}\thinspace}
\providecommand{\MR}{\relax\ifhmode\unskip\space\fi MR }
\providecommand{\MRhref}[2]{%
  \href{http://www.ams.org/mathscinet-getitem?mr=#1}{#2}
}
\providecommand{\href}[2]{#2}

\printindex
\end{document}